\documentclass[reqno]{amsart}
\usepackage{amsthm,amsmath,amsfonts,amssymb,amscd,mathrsfs,graphics,diagrams}
\usepackage{txfonts}
\usepackage{hyperref,supertabular}

\newcommand{\A}{{\mathscr A}}
\newcommand{\BG}{\B G} 
\newcommand{\B}{{\mathscr B}}
\newcommand{\CP}{{\mathbb{CP}}}
\newcommand{\C}{{\mathbb{C}}}

\newcommand{\D}{{\mathscr D}}

\newcommand{\F}{{\mathscr F}}
\newcommand{\Gss}[1]{H^{mid}(\C^{N_{J^{#1}}}, W_{J^{#1}}^{\infty}, \Q)} 

\newcommand{\Hess}{\operatorname{Hess}}

\newcommand{\IC}{{\bigwedge\cC}}
\newcommand{\Ipi}{{I\pi}}
\newcommand{\J}{{\mathscr J}}
\newcommand{\Jac}{\operatorname{Jac}} 
\newcommand{\K}{{\mathscr K}}
\newcommand{\LL}{{\mathscr L}} 

\newcommand{\MM}{\overline{\M}} 
\newcommand{\M}{{\mathscr M}}
\newcommand{\N}{{\mathscr N}}
\newcommand{\Pic}{\operatorname{Pic}}
\newcommand{\Q}{{\mathbb{Q}}}
\newcommand{\R}{{\mathbb{R}}}
\newcommand{\Res}{{\operatorname{Res}}} 

\newcommand{\W}{\overline{\mathscr W}}  
\newcommand{\Z}{{\mathbb{Z}}}
\newcommand{\aut}{\operatorname{Aut}}
\newcommand{\balpha}{\boldsymbol{\alpha}}
\newcommand{\barGamma}{\overline{\Gamma}}
\newcommand{\be}{\mathbf{e}}
\newcommand{\bgamma}{\boldsymbol{\gamma}} 
\newcommand{\bga}{\bgamma}
\newcommand{\bGa}{{\boldsymbol{\Ga}}}
\newcommand{\bGagkW}{{\bGa_{g,k,W}}}
\newcommand{\bGagkWcut}{{\bGa_{g,k,W,\mathrm{cut}}}}

\newcommand{\bone}{\be}

\newcommand{\bt}{\mathbf{t}}

\newcommand{\bx}{\mathbf{x}}
\newcommand{\cCbb}{\cCb} 
\newcommand{\cCb}{\overline{\mathscr C}} 
\newcommand{\cC}{\mathscr C} 

\newcommand{\cS}{\mathscr{S}}

\newcommand{\chat}{\hat{c}}
\newcommand{\ch}{{\mathscr H}}
\newcommand{\cht}{{\widetilde{Ch}}}

\newcommand{\coker}{\operatorname{coker}}
\newcommand{\corf}[2]{{\langle #2\rangle_{0}^{#1}}} 
\newcommand{\ct}{\tilde{c}}  

\newcommand{\cZ}{{\mathscr{Z}}}

\newcommand{\dsand}{\quad \text{ and }\quad }

\newcommand{\eul}{{e}} 
\newcommand{\fL}{\mathfrak{L}}
\newcommand{\fix}{\operatorname{Fix}}
\newcommand{\frkc}{\mathfrak{C}} 

\newcommand{\g}{\varepsilon} 
\newcommand{\ga}{{\gamma}}
\newcommand{\gap}{{\gamma_{+}}}
\newcommand{\Ga}{{\Gamma}}
\newcommand{\Gac}{{\Ga_{\mathrm{cut}}}}
\newcommand{\genj}{\langle J \rangle}
\newcommand{\gerbeprod}{\omega} 

\newcommand{\grp}{G_W} 
\newcommand{\hGamma}{\widehat{\Gamma}}

\newcommand{\iit}{\hat{i}}

\newcommand{\inc}{adm} 
\newcommand{\ind}{\operatorname{index}}
\newcommand{\irightarrow}{\widetilde{\longrightarrow}}
\newcommand{\itt}{\tilde{i}}

\newcommand{\kappat}{\widetilde{\kappa}}
\newcommand{\lgr}[1]{\langle\gamma_{#1}\rangle}
\newcommand{\mfm}{\mathfrak{m}}
\newcommand{\milnor}{\mathscr{Q}}

\newcommand{\nut}{r}

\newcommand{\psit}{\widetilde{\psi}}
\renewcommand{\r}{{\upsilon}}
\newcommand{\res}{\operatorname{res}}
\newcommand{\restr}[2]{\left.#1\right|_{#2}} 

\newcommand{\so}{so} 
\newcommand{\spa}[1]{\operatorname{span}\{#1\}}
\newcommand{\spanl}{\operatorname{span}(}
\newcommand{\spanr}{)}
\newcommand{\st}{st} 
\newcommand{\stt}{\widehat{\st}}
\newcommand{\suth}{:} 
\newcommand{\tGamma}{\widetilde{\Gamma}}
\newcommand{\tcC}{\widetilde{\cC}}
\newcommand{\tdt}{{\widetilde{Td}}}
\newcommand{\tW}{\widetilde{W}}
\newcommand{\union}{\cup}
\newcommand{\unit}{\boldsymbol{1}}

\newcommand{\ve}{{\varepsilon}}

\newcommand{\varrhob}{\overline{\varrho}}
\newcommand{\vpt}{\tilde{\varphi}}
\newcommand{\wit}{{\D}} 
\newcommand{\wt}{\operatorname{wt}}

\newcommand{\x}{{\mathbf x}}

\renewcommand{\O}{{\mathscr{O}}}
\renewcommand{\hom}{\operatorname{Hom}}

\renewcommand{\Re}{\mathfrak{Re}}

\newcommand{\Nb}{{\mathbb{N}}}

\newtheorem{thm}{Theorem}[subsection]
\newtheorem{lm}[thm]{Lemma}
\newtheorem{prop}[thm]{Proposition}
\newtheorem{crl}[thm]{Corollary}
\newtheorem{conj}[thm]{Conjecture}

\theoremstyle{definition}
\newtheorem{rem}[thm]{Remark}

\newtheorem{df}[thm]{Definition}
\newtheorem{ex}[thm]{Example}
\theoremstyle{remark}

\begin{document}

\date{\today}

\title[Quantum singularity theory]{The Witten equation, mirror symmetry
and quantum singularity theory}
\author{Huijun Fan}
\thanks{Partially Supported by NSFC 10401001, NSFC 10321001, and NSFC 10631050}
\address{School of Mathematical Sciences, Peking University and  and Beijing International Center For Mathematical Research, Beijing 100871, China,}
\email{fanhj@math.pku.edu.cn}
\author{Tyler Jarvis}
\thanks{Partially supported by National Science Foundation grants DMS-0605155 and DMS-0105788 and the Institut Mittag-Leffler (Djursholm, Sweden)}
\address{Department of Mathematics, Brigham Young University, Provo, UT 84602, USA}
\email{jarvis@math.byu.edu}
\author{Yongbin Ruan}
\thanks{Partially supported by the National Science Foundation and
the Yangtze Center of Mathematics at Sichuan University}
\address{Yangtz Center of Mathematics at Sichuan University, Chengdu,China  and Department of Mathematics, University of Michigan Ann Arbor, MI 48105 U.S.A
} \email{ruan@umich.edu}

\begin{abstract}
For any non-degenerate, quasi-homogeneous hypersurface singularity, we describe a family of moduli spaces, a virtual cycle,
and a corresponding cohomological field theory associated to the singularity.
This theory is analogous to Gromov-Witten theory and generalizes the theory of $r$-spin curves,
which corresponds to the simple singularity $A_{r-1}$.

We also resolve two outstanding conjectures of
Witten.  The first conjecture is that ADE-singularities are self-dual; and the second conjecture is that the total potential functions of ADE-singularities satisfy
corresponding ADE-integrable hierarchies. Other cases of integrable hierarchies
are also discussed.
\end{abstract}

\maketitle

\tableofcontents

\section{Introduction}
 The study of singularities has a long history in mathematics. For
 example, in algebraic geometry it is often necessary to study
 algebraic varieties with singularities even if the initial goal was
 to work only with smooth varieties. Many important surgery
 operations, such as flops and flips, are closely associated with
 singularities. In lower-dimensional topology, links of singularities
 give rise to many important examples of 3-manifolds.  Singularity
 theory is also an important subject in its own right.  In fact,
 singularity theory has been well-established for many decades (see
 \cite{AGV}). One of the most famous examples is the
 ADE-classification of hypersurface singularities of zero modality. We
 will refer to this part of singularity theory as \emph{classical
   singularity theory} and review some aspects of the classical theory
 later. Even though we are primarily interested in the quantum aspects
 of singularity theory, the classical theory always serves as a source
 of inspiration.

Singularity theory also appears in physics. Given a polynomial $W:
\C^n\rTo \C$ with only isolated critical (singular) points, one can
associate to it the so-called \emph{Landau-Ginzburg model}.  In the
early days of quantum cohomology, the Landau-Ginzburg model and
singularity theory gave some of the first examples of Frobenius
manifolds.  It is surprising that although the Landau-Ginzburg model
is one of the best understood models in physics, there has been no
construction of Gromov-Witten type invariants for it until now.
However, our initial motivation was not about singularities and the
Landau-Ginzburg model.  Instead, we wanted to solve the \emph{Witten
  equation}:
    $$\bar{\partial}u_i+\overline{\frac{\partial W}{\partial
    u_i}}=0,$$ where $W$ is a quasi-homogeneous polynomial, and $u_i$
is interpreted as a section of an appropriate orbifold line bundle
on a Riemann surface $\cC$.  

The simplest Witten equation is the $A_{r-1}$ case.  This is of the
form
    $$\bar{\partial}u+r\bar{u}^{r-1}=0.$$ It was introduced by Witten
\cite{Wi2} more than fifteen years ago as a generalization of
topological gravity. Somehow, it was buried in the literature without
attracting much attention. Several years ago, Witten generalized his
equation for an arbitrary quasi-homogeneous polynomial \cite{Wi3} and
coined it the ``Landau-Ginzburg A-model.'' Let us briefly recall the
motivation behind Witten's equation.  Around 1990, Witten proposed a
remarkable conjecture relating the intersection numbers of the
Deligne-Mumford moduli space of stable curves with the KdV hierarchy
\cite{Wi1}. His conjecture was soon proved by Kontsevich
\cite{K}. About the same time, Witten also proposed a generalization
of his conjecture. In his generalization, the stable curve is replaced
by a curve with a root of the canonical bundle ($r$-spin curve), and
the KdV-hierarchy was replaced by more general KP-hierarchies called
nKdV, or Gelfand-Dikii, hierarchies.  The $r$-spin curve can be
thought of as the background data to be used to set up the Witten equation in
the $A_{r-1}$-case. Since then, the moduli space of $r$-spin curves
has been rigorously constructed by Abramovich, Kimura, Vaintrob and
the second author \cite{AJ,J1,J2,JKV1}.  The more general Witten
conjecture was proved in genus zero several years ago \cite{JKV1}, in
genus one and two by Y.-P. Lee \cite{Lee}, and recently in higher
genus by Faber, Shadrin, and Zvonkine \cite{FSZ}.

The theory of $r$-spin curves (corresponding to the $A_{r-1}$-case of our theory) does not need the Witten equation at
all. This partially explains the fact that the Witten equation has been neglected in the
literature for more than ten years.   In the $r$-spin case the algebro-geometric data is an
orbifold line bundle $\LL$ satisfying the equation $\LL^r=K_{\log}$.
Assume that all the orbifold points are marked points.  A marked
point with trivial orbifold structure is called a \emph{broad} (or \emph{Ramond} in our old notation) 
marked point, and a marked point with non-trivial orbifold
structure is called a \emph{narrow} (or \emph{Neveu-Schwarz} in our old notation) marked point.  Contrary
to intuition, broad marked points are much harder to study than
narrow marked points. If there is no broad marked point, a
simple lemma of Witten's shows that the Witten equation has only
the zero solution. Therefore, our moduli problem becomes an
algebraic geometry problem.  In the $r$-spin case the contribution from the  broad marked point to the corresponding field
theory is zero (the decoupling of the broad
sector).  This was conjectured by Witten and proved true for genus zero in \cite{JKV1} and
for higher genus in \cite{P}.  This means that in the $r$-spin case, there is no need for the Witten equation, which partly explains why the moduli space
of higher spin curves has been around for a long time while the
Witten equation seems to have been lost in the literature.

In the course of our investigation, we discovered that in the
$D_{n}$-case the broad sector gives a nonzero contribution.
Hence, we had to develop a theory that accounts for the
contribution of the solution of the Witten equation in the
presence of broad marked points.

It has taken us a while to understand the general picture, as well as
various technical issues surrounding our current theory. In fact, an
announcement was made in 2001 by the last two authors for some special cases 
coupled with an orbifold target. We apologize for the
long delay because we realized later that (1) the theory admits a vast
generalization to an arbitrary quasi-homogeneous singularity, and (2) the
broad sector has to be investigated. We would like to mention that
the need to invesigate the broad sector led us to the space of
Lefschetz thimbles and other interesting aspects of the
Landau-Ginzburg model, including Seidel's work on the Landau-Ginzburg
A-model derived category \cite{Si}. In many ways, we are happy to have
waited for several years to arrive at a much more complete and more
interesting theory!

To describe our theory, let's first review some classical singularity
theory.  Let $W: \C^N \rTo \C$ be a quasi-homogeneous polynomial.
    Recall that $W$ is a quasi-homogeneous polynomial if there are
    positive integers $d, n_1, \dots, n_n$ such that $W(\lambda^{n_1}
    x_1, \dots, \lambda^{n_N} x_n)=\lambda^d w(x_1, \dots, x_N).$ We
    define the \emph{weight} (or \emph{charge}), of $x_i$ to be
    $q_i:=\frac{n_i}{d}.$ We say $W$ is \emph{nondegenerate} if (1)
    the choices of weights $q_i$ are unique, and (2) $W$ has a singularity
    only at zero. There are many examples of non-degenerate
    quasi-homogeneous singularities, including all the nondegenerate
    homogeneous polynomials and the famous $ADE$-examples.
    \begin{ex}
\
    \begin{description}
    \item[$A_n$] $W=x^{n+1}, \ n\geq 1;$ \glossary{An@$A_n$ & The singularity defined by $W=x^{n+1}$}
    \item[$D_n$] $W=x^{n-1}+xy^2, \ n\geq 4;$ \glossary{Dn@$D_n$ & The singularity defined by $W=x^{n-1}+xy^2$}
    \item[$E_6$] $W=x^3+y^4;$\glossary{E6@$E_6$ & The singularity defined by $W=x^{3}+y^4$}
    \item[$E_7$] $W=x^3+xy^3;$\glossary{E7@$E_7$ & The singularity defined by $W=x^{3}+xy^3$}
    \item[$E_8$] $W=x^3+y^5;$\glossary{E8@$E_8$ & The singularity defined by $W=x^{3}+y^5$}
    \end{description}
    \end{ex}

 The simple singularities (A, D, and E) are the only examples
 with so-called \emph{central charge} $\chat_W<1$. There are many more
 examples with $\chat_W \geq 1$.

In addition to the choice of a non-degenerate singularity $W$, our theory also depends on a choice of subgroup $G$ of the the group $\aut(W)$ of diagonal
matrices $\gamma$ such that $W(\gamma x)=W(x)$.  We often use the notation $G_W:=\aut(W)$, and we call this group
the \emph{maximal diagonal symmetry group} of $W$. The group  $G_W$
always contains the \emph{exponential grading} (or \emph{total
monodromy}) \emph{element} $J=diag(e^{2\pi i q_1}, \dots, e^{2\pi
i q_N})$, and hence is always nontrivial.

Given a choice of non-degenerate $W$ and a choice of admissible (see Section~\ref{sec:changeGroup}) subgroup $G\le G_W$  with $\genj \le G \le \aut(W)$, we construct a cohomological field theory whose state space is defined as follows.  For each $\gamma\in G$, let $\C^{N_{\gamma}}$ be the fixed point set
of $\gamma$ and $W_{\gamma}=\restr{W}{\C^{N_{\gamma}}}$. Let
$\ch_{\gamma,G}$ be the $G$-invariants of the middle-dimensional
relative cohomology 
    $$\ch_{\gamma,G}= H^{mid}(\C^{N_{\gamma}}, (\Re W)^{-1}(M, \infty),
\C)^{G}$$ of $\C^{N_{\gamma}}$ for $M>\!\!>0$, as described in Section~\ref{sec:QcohGroup}. The \emph{state space} of our theory is the sum
    $$\ch_{W,G}=\bigoplus_{\gamma\in G} \ch_{\gamma,G}.$$ The state space
$\ch_{W,G}$ admits a
grading and a natural non-degenerate
pairing.

For $\alpha_1, \dots, \alpha_k\in \ch_{W,G}$ and a sequence of
non-negative integers $l_1, \dots, l_k$, we define (see
Definition \ref{df:correlator}) the genus-$g$ correlator
    $$\langle\tau_{l_1}(\alpha_1),\dots,
\tau_{l_k}(\alpha_k)\rangle_g^{W,G}$$ by integrating over a
certain virtual fundamental cycle.  In this paper we describe
the axioms that this cycle satisfies and the consequences of those
axioms.  In a separate paper \cite{FJR-WEVFC} we construct
the cycle and prove that it satisfies the axioms.
       \begin{thm}
    The correlators $\langle\tau_{l_1}(\alpha_1),\dots,
    \tau_{l_k}(\alpha_k)\rangle^{W,G}_g$ satisfy the usual axioms of
    Gromov-Witten theory (see
    Subsection~\ref{sec:CohFT}), but where the divisor axiom is replaced with another axiom that facilitates computation.
    \end{thm}
In particular, the three-point correlator together with the
pairing defines a Frobenius algebra structure on $\ch_{W,G}$ by the
formula
    $$\langle\alpha\star\beta, \gamma\rangle=\langle\tau_0(\alpha),
\tau_0(\beta), \tau_0(\gamma)\rangle_0^{W,G}.$$

One important point is the fact that our construction depends crucially on the Abelian
automorphism group $G$. Although there are at least two choices of group that might be considered canonical (the group generated by the exponential grading operator $J$ or the maximal diagonal symmetry group $G_W$), we do not know how to construct a Landau-Ginzburg
A-model defined by $W$ alone. In this sense, the orbifold LG-model
$W/G$ is more natural than the LG-model for $W$ itself.

We also remark that our theory is also new in physics. Until now there has been no description of the closed-string sector of the Landau-Ginzburg model.

Let's come back to the Witten-Kontsevich theorem regarding the KdV hierarchy in geometry.
Roughly speaking, an integrable hierarchy is a system of differential
    equations
    for a function of infinitely many time variables $F(x, t_1, t_2, \cdots)$ where $x$ is
    a spatial variable and $t_1, t_2, \cdots,  $ are time variables. The PDE is a system of evolution
    equations of the form
    $$\frac{\partial F}{\partial t_n}=R_n(x, F_x, F_{xx}, \cdots),$$
    where $R_n$ is a polynomial. Usually, $R_n$ is constructed
    recursively. There is an alternative formulation in terms of the so-called Hirota bilinear equation
    which $e^F$ will satisfy. We often say that $e^F$ is a \emph{$\tau$-function} of hierarchy.
    It is well-known that KdV is the $A_1$-case of more general $ADE$-hierarchies. As far as we know, there are two versions of ADE-integrable hierarchies: the first constructed by
    Drinfeld-Sokolov \cite{DS} and the second constructed by Kac-Wakimoto
    \cite{KW}. Both of them are constructed from integrable
    representations of affine Kac-Moody algebras.  These two constructions are equivalent by the work of Hollowood and Miramontes \cite{HM}.

Witten's original motivation was to generalize the geometry of Deligne-Mumford space to realize ADE-integrable hierarchies.
Now, we can state his integrable hierarchy conjecture rigorously.
Choose a basis $\alpha_i$ ($i\leq s$) of
$\ch_{W,G}$. Define the genus-$g$ generating function
$${\F}_{g, W, G}=\sum_{k\geq 0} \langle\tau_{l_1}(\alpha_{i_1}), \dots,
\tau_{l_n}(\alpha_{i_n})\rangle^{W,G}_g\frac{t^{l_1}_{i_1} \cdots
t^{l_n}_{i_s}}{n!}.$$
     Define the total potential function
    $${\D}_{W,G}=\exp(\sum_{g\geq 0} h^{2g-2}{\F}_{g, W,G}).$$

\begin{conj}[Witten's ADE-Integrable Hierarchy Conjecture:] The total potential functions of  the A, D, and E singularities
with the symmetry group $\genj$
generated by the exponential grading operator, are $\tau$-functions of the corresponding A, D, and E integrable hierarchies.
\end{conj}

In the $A_n$ case, this conjecture is often referred as \emph{the Generalized Witten conjecture}, as compared to the original
 Witten conjecture proved by Kontsevich \cite{K}. As mentioned earlier, the conjecture for the $A_n$-case has
 been established recently by Faber, Shadrin, and Zvonkine \cite{FSZ}. The original Witten conjecture also inspired
 a great deal of activity related to Gromov-Witten theory of more general spaces. Those cases are $2$-Toda for $\CP^1$ by Okounkov-Pandharipande \cite{OP} and the Virasoro constraints for toric manifolds by Givental \cite{Gi4}, Riemann surfaces by Okounkov-Pandharipande \cite{OP1}. In some sense, the ADE-integrable hierarchy conjecture is analogous to these lines of research but where the targets are singularities.

 The main application of our theory is the resolution of the ADE-integrable hierarchy conjecture, as manifested by the following two theorems.
\begin{thm}
The total potential functions of the singularities $D_n$ with even $n\ge 6$, and $E_6$, $E_7$, and $E_8$,  with the group $\genj $  are $\tau$-functions of the corresponding
Kac-Wakimoto/Drinfeld-Sokolov hierarchies.
\end{thm}
    We expect the conjecture for $D_4$ to be true as well. However, our calculational tools are not strong enough to prove it at
    this moment. We hope to come back to it at another occasion.

Surprisingly, the Witten conjecture for  $D_n$ with $n$ odd is false.  Note that in the case of $n$ even, the subgroup $\genj$ has index two in the maximal group $G_{D_{n}}$ of diagonal symmetries, but in the case that $n$ is odd, $\genj$ is equal to $G_{D_{n}}$.
In this paper we prove
    \begin{thm}
    \item[(1)] For all $n>4$ the total potential function of the $D_n$-singularity  with
    the maximal diagonal symmetry group $G_{D_{n}}$ is a $\tau$-function of the
    $A_{2n-3}$-Kac-Wakimoto/Drinfeld-Sokolov hierarchies.
    \item[(2)] For all $n> 4$ the total potential function of $W=x^{n-1}y+y^2$ ($n\geq 4$)
    with the maximal diagonal symmetry group is a $\tau$-function of the $D_n$-Kac-Wakimoto/Drinfeld-Sokolov hierarchy.
    \end{thm}

 The above two theorems realize the $ADE$-hierarchies completely in our
theory.   Moreover, it illustrates the important role that the group of symmetries plays in our constructions: when the symmetry group is $G_{D_{n}}$ we have the $A_{2n-3}$-hierarchy, but when the symmetry group is $\genj$, and when $\genj$ is a proper subgroup of $G_{D_n}$, we have the $D_n$-hierarchy.

     Readers may wonder about the singularity $W=x^{n-1}y+y^2$ (which is isomorphic to $A_{2n-3}$). Its appearance reveals a deep connection between integrable hierarchies and mirror symmetry (see more in Section~\ref{sec:six}).

Although the simple singularities are the only singularities with
central charge $\chat_W<1$, there are many more examples of
singularities. It would be an extremely interesting problem to
find other integrable hierarchies corresponding to singularities
with $\chat_W\ge 1$.

Witten's second conjecture is the following ADE self-mirror
conjecture which interchanges the A-model with the B-model.
   \begin{conj}[ADE Self-Mirror Conjecture]
If $W$ is a simple singularity, then for the symmetry group
$\genj$, generated by the exponential grading operator,
 the  ring $\ch_{W,\genj}$ is isomorphic to the Milnor ring of
 $W$.
\end{conj}

The second main theorem of this paper is the following.

\begin{thm}\label{thm:mirror} { } \
\begin{description}
    \item[(1)] Except for $D_n$ with $n$ odd, the
    ring $\ch_{W,\genj}$ of any simple (ADE) singularity $W$ with group
    $\genj $ is isomorphic to the Milnor ring $\milnor_W$ of the same singularity.
    \item[(2)] The
    ring $\ch_{D_n,G_{D_n}}$  of $D_n$ with the maximal diagonal
    symmetry group $G_{D_n}$ is isomorphic to the Milnor ring $\milnor_{A_{2n-3}}$ of $W = x^{n-1}y+y^2$.
    \item[(3)] The
  ring $\ch_{W,G_W}$ of $W=x^{n-1}y+y^2$ ($n\geq 4$)
    with the maximal diagonal symmetry group $G_W$ is isomorphic to the    Milnor ring $\milnor_{D_n}$ of $D_n$.
    \end{description}
\end{thm}

    The readers may note the similarities between the statements of the above mirror symmetry theorem and our integrable hierarchies theorems. In fact, the mirror symmetry theorem is the first
    step towards  the proof of integrable hierarchies theorems.

Of course we cannot expect that most singularities will be
self-mirror, but we can hope for mirror symmetry beyond just the
simple singularities.  Since the initial draft of this paper, much progress
has been made 
\cite{FJJS, Krawitz, Pr} for invertible singularities. An invertible singularity has the property that the number of monomials is equal to the
number of variables. This is a large class of quasi-homogeneous singularities.

In general, it is a very difficult problem to compute
Gromov-Witten invariants of compact Calabi-Yau manifolds. While there are
many results for low genus cases \cite{Gi3,LLY, Z}, there are only 
a very few compact examples \cite{MP,OP1} where one knows how to
compute Gromov-Witten invariants in all genera by either
mathematical or physical methods (for some recent advances see
\cite{HQK}).

Note that a Calabi-Yau hypersurface  of weighted
projective space defines a quasi-homogenenous singularity and
hence an LG-theory. This type of singularity has $\sum_i q_i=1$.
 In the early `90s, Martinec-Vafa-Warner-Witten proposed a
famous conjecture \cite{Mar}, \cite{VW}, \cite{Wi4} connecting
these two points of view.

\begin{conj}[Landau-Ginzburg/Calabi-Yau Correspondence] The LG-theory of a generic quasi-homogeneous singularity
    $W/\langle J \rangle$ and the corresponding Calabi-Yau theory are isomorphic
  in a certain sense.
\end{conj}
This is certainly one of the most important conjectures in the
subject. The importance of the conjecture comes from the physical
indication that the LG/singularity-theory is much easier to
compute than the Calabi-Yau geometry.  The precise mathematical
statement of the above conjecture is still lacking at this moment (see \cite{ChR} also).
We hope to come back to it on another occasion.

We conclude by noting that it would be a very interesting problem to explore how to extend our results to a setting like that treated by Guffin and Sharpe in \cite{GS1,GS2}.   They have considered twisted Landau-Ginzburg models without coupling to topological gravity, but over more general orbifolds; whereas our model couples to topological gravity, but we work exclusively with orbifold vector bundles.

\subsection{Organization of the paper}

A complete construction of our theory will be carried out in a
series of papers.  In this paper, we give a complete description
of the algebro-geometric aspects of our theory. The information
missing is the analytic construction of the moduli space of
solutions of the Witten equation and its virtual fundamental
cycle, which is done in a separate paper \cite{FJR-WEVFC}.
Here, we summarize the main properties or axioms of the cycle and
their consequences.  The main application is the proof of Witten's
self-mirror conjecture and integrable hierarchies conjecture for
ADE-singularities.

The paper is organized as follows. In Section 2, we will set up
the theory of $W$-structures. This is the background data for the
Witten equation and a generalization of the well-known theory of
$r$-spin curves. The analog of quantum cohomology groups and the state space of
the theory will be described in Section 3. In Section 4, we
formulate a list of axioms of our theory. The proof of Witten's
mirror symmetry conjecture is in Section 5. The proof of his
integrable hierarchies conjecture is in Section 6.

\subsection{Acknowledgments}

The third author would like to express his special thanks to E.
Witten for explaining to him his equation in 2002 and for his
support over these years.  Thanks also go to K. Hori and A. Klemm for many stimulating discussions about Landau-Ginzburg models.

The last two authors would like to thank R. Kaufmann for explaining his
work, for many helpful discussions, and for sharing a common
interest and support on this subject for these years. We also thank Marc Krawitz for showing us the Berglund-H\"ubsch mirror construction, Eric Sharpe for explaining to us some aspects of his work in \cite{GS1,GS2}, and Alessandro Chiodo for his insights.  The second
author would also like to thank T. Kimura for helpful discussions
and insights, and the Institut Mittag-Leffler, for providing a
stimulating environment for research.

The first author would like to
thank K.~C.~Chang, Weiyue Ding, and J.~Jost for their long-term
encouragement and support, and especially he wants to thank Weiyue
Ding for fruitful discussions and warm help for many years. He
also thanks Bohui Chen for many useful suggestions and comments.
Partial work was done when the first author visited MPI in Leipzig,
MSRI in Berkeley and the University of Wisconsin-Madison
respectively; he appreciates their hospitality.

The first and second authors thank H.~Tracy Hall for many helpful discussions and for his ideas and insights related to Proposition~\ref{prop:HTH-Laurent} and Equation~(\ref{eq:HTH-roots}).

Finally, the third
author would like to thank the University of Wisconsin-Madison,
where the much of the current work has taken place, for warm
support and fond memories.

\section{$W$-curves and their moduli}\label{sec:moduli}

\subsection{$W$-structures on orbicurves}
\subsubsection{Orbicurves and line bundles}

Recall that an orbicurve $\cC$\glossary{CC@$\cC$ & An orbicurve} with
marked points $p_1, \dots, p_k$ is a (possibly nodal) Riemann surface
$C$\glossary{CCC@$C$ & The coarse (underlying) curve of the orbicurve
  $\cC$} with orbifold structure at each $p_i$ and each node.  That is
to say, for each marked point $p_i$ there is a local group $G_{p_i}$
and (since we are working over $\C$) a canonical isomorphism $G_{p_i}
\cong \Z/m_i$ for some positive integer $m_i$.  A neighborhood of
$p_i$ is uniformized by the branched covering map $z \rTo z^{m_i}$.
For each node $p$ there is again a local group $G_p \cong \Z/n_j$
whose action is complementary on the two different branches.  That is
to say, a neighborhood of a nodal point (viewed as a neighborhood of
the origin of $\{z w=0\}\subset \C^2$) is uniformized by a branched
covering map $(z,w)\rTo (z^{n_j}, w^{n_j})$, with $n_j\geq 1$, and
with group action $e^{2 \pi i /n_j}(z,w)=(e^{2 \pi i /n_j}z, e^{-2\pi
  i/n_j}w)$.

\begin{df}
We will call the orbicurve $\cC$ \emph{smooth} if the underlying curve
$C$ is smooth, and we will call the orbicurve \emph{nodal} if the
underlying curve $C$ is nodal.
\end{df}
Note that this definition agrees with that of
algebraic geometers for smooth Deligne-Mumford stacks, but it differs from that of many topologists (e.g., \cite{CR1}) since orbicurves with non-trivial orbifold structure at a point will still be called smooth when the underlying curve is smooth.

We denote by $\varrho:\cC \rTo C$\glossary{rho@$\varrho$ & The
  projection to coarse (underlying) space, given by forgetting the
  orbifold structure} the natural projection to the underlying
(coarse, or non-orbifold) Riemann surface $C$. If $\LL$ is a line
bundle on $C$, it can be uniquely lifted to an orbifold line bundle
$\varrho^*{\LL}$ over $\cC$. When there is no danger of confusion, we
use the same symbol $\LL$ to denote its lifting.

\begin{df}\label{df:K-log}
Let $K_C$\glossary{K@$K$ & The canonical (relative dualizing) bundle}
be the canonical bundle of $C$.  We define the \emph{log-canonical
  bundle of $C$} to be the line bundle
 $$K_{C,\log} := K \otimes \O(p_1) \otimes \dots \otimes
  \O(p_k),\glossary{Klog@$K_{\log}$ & The log-canonical bundle}$$ where
  $\O(p)$ is the holomorphic line bundle of degree one whose sections
  may have a simple pole at $p$.  This bundle $K_{C,\log}$ can be
  thought of as the canonical bundle of the punctured Riemann surface
  $C-\{p_1, \dots, p_k\}$.

The \emph{log-canonical bundle of $\cC$} is defined to be the pullback
to $\cC$ of the log-canonical bundle of $C$:
\begin{equation}\label{eq:Klog}
K_{\cC,\log} := \varrho^* K_{C,\log}.
\end{equation}
\end{df}

Near a marked point $p$ of $C$ with local coordinate $x$, the bundle
$K_{C,\log}$ is locally generated by the meromorphic one-form $dx/x$.
If the local coordinate near $p$ on $\cC$ is $z$, with $z^m=x$, then
the lift $K_{\cC,\log} :=\varrho^*(K_{C,\log})$ is still locally
generated by $m\, dz/z = dx/x$.  When there is no risk of confusion,
we will denote both $K_{C,\log}$ and $K_{\cC,\log}$ by $K_{\log}$.  Near
a node with coordinates $z$ and $w$ satisfying $zw=0$, both $K$
and $K_{\log}$ are locally generated by the one-form $dz/z = -dw/w$.

Note that although $\varrho^* K_{C,\log} =
  K_{\cC,\log},$ the usual canonical bundle does not pull back to itself:
\begin{equation}\varrho^* K_{C}  = K_{\cC} \otimes \O(-\sum_{i=1}^k (m_i-1)p_i)
\neq K_{\cC},\end{equation} where $m_i$ is the order of the local group  at $p_i$.  This can be seen from the fact that when $x=z^m$ we have \begin{equation}\label{eq:CoarseFineForms}dx = m z^{m-1}
  dz.\end{equation}

\subsubsection{Pushforward to the underlying curve}
\label{rem:desingularize}
If $\LL$ is an orbifold line bundle on a smooth orbicurve $\cC$, then
the sheaf of locally invariant sections of $\LL$ is locally free of
rank one, and hence dual to a unique line bundle $|\LL|$ on $\cC$. We
also denote $|\LL|$ by $\varrho_* \LL$, and it is called the
``desingularization" of $\LL$\glossary{Lbar@$"|\LL"|$ & The
  desingularization of the line bundle $\LL$} in \cite[Prop
4.1.2]{CR1}. It can be constructed explicitly as follows.

We keep the local trivialization at non-orbifold points, and change it
at each orbifold point $p$.  If $\LL$ has a local chart $\Delta \times
\C$ with coordinates $(z, s)$, and if the generator $1\in \Z/m \cong G_p$ acts locally
on $\LL$ by $(z,s) \mapsto (\exp(2\pi i/m) z, \exp(2\pi i v /m) s)$,
then we use the $\Z/m$-equivariant map $\Psi: (\Delta-\{0\})\times \C
\rTo \Delta \times \C$ given by
\begin{equation}(z, s)\rTo (z^m, z^{-v}s),
\label{eq:desing-triv}
\end{equation}
where $\Z/m$ acts trivially on the second $\Delta\times \C$.
Since $\Z/m$ acts trivially, this gives a line bundle over $C$, which is $|\LL|$.

If the orbicurve $\cC$ is nodal, then the pushforward $\varrho_*\LL$
of a line bundle $\LL$ may not be a line bundle on $C$.  In fact, if
the local group $G_p$ at a node acts non-trivially on $\LL$, then the
invariant sections of $\LL$ form a rank-one torsion-free sheaf on $C$
(see \cite{AJ}).  However, we may take the normalizations $\tcC$ and
$\widetilde{C}$ to get (possibly disconnected) smooth curves, and the
pushforward of $\LL$ from $\tcC$ will give a line bundle on
$\widetilde{C}$.  Thus $|\LL|$ is a line bundle away from the nodes of
$C$, but its fiber at a node is two-dimensional; that is, there is
(usually) no gluing condition on $|\LL|$ at the nodal points.  The
situation is slightly more subtle than this (see \cite{AJ}), but for
our purposes, it will be enough to consider the pushforward $|\LL|$ as
a line bundle on the normalization $\widetilde{C}$ where the local
group acts trivially on $\LL$.

It is also important to understand more about the sections of the
pushforward $\varrho_*\LL$. Suppose that $s$ is a section of $|\LL|$
having local representative $g(u)$.  Then $(z, z^v g(z^m))$ is a
local section of $\LL$. Therefore, we obtain a section
$\varrho^*(s)\in \Omega^0(\LL)$ which equals $s$ away from orbifold
points under the identification given by
Equation~\ref{eq:desing-triv}. It is clear that if $s$ is holomorphic,
so is $\varrho^*(s)$. If we start from an analytic section of $\LL$,
we can reverse the above process to obtain a section of $|\LL|$. In
particular, $\LL$ and $|\LL|$ have isomorphic spaces of holomorphic
sections: $$\varrho^*: H^0(C,|\LL|) \irightarrow H^0(\cC,\LL). $$ In
the same way, there is a map $\varrho^*: \Omega^{0,1}(|\LL|)\rTo
\Omega^{0,1}(\LL)$, where $\Omega^{0,1}(\LL)$ is the space of orbifold
$(0,1)$-forms with values in $\LL$. Suppose that $g(u)d\bar{u}$ is a
local representative of a section of $t\in \Omega^{0,1}(|\LL|)$. Then
$\varrho^*(t)$ has a local representative $z^v g(z^m) m \bar{z}^{m-1}
d\bar{z}$.  Moreover, $\varrho$ induces an isomorphism
$$\varrho^* :H^1(C,|\LL|)\irightarrow H^1(\cC,\LL).$$

\begin{ex}
The pushforward $|K_{\cC}|$ of the log-canonical bundle of any
orbicurve $\cC$ is again the log-canonical bundle of $C$, because at a
point $p$ with local group $G_p \cong \Z/m$ the one-form $m\, dz/z =
dx/x$ is invariant under the local group action.

Similarly, the pushforward $|K_{\cC}|$ of the canonical bundle of
$\cC$ is just the canonical bundle of $C$:
\begin{equation}\label{eq:KbarIsK}
|K_{\cC}| = \varrho_*K_{\cC} = K_C,
\end{equation}
 because the local group $\Z/m$ acts on the one-form $dz$ by $\exp(2
 \pi i /m) dz$, and the invariant holomorphic one-forms are precisely
 those generated by $mz^{m-1}dz = dx$.
\end{ex}

\subsubsection{Quasi-homogeneous polynomials and their Abelian automorphisms}

\begin{df}\label{df:qhomPoly}
A \emph{quasi-homogeneous} (or \emph{weighted homogeneous})
\emph{polynomial} $W \in \mathbb{C} [x_1, \dots, x_N]$\glossary{N@$N$
  & The number of variables in the quasi-homogeneous polynomial $W$
  and the number of line bundles in the $W$-structure}\glossary{W@$W$
  & A quasi-homogeneous polynomial} is a polynomial for which there
exist positive rational numbers $q_1, \dots, q_N \in \Q^{>0}$, such
that for any $\lambda \in \mathbb{C}^*$
\begin{equation}\label{eq:qhomocharge}
W(\lambda^{q_1}x_1, \dots, \lambda^{q_N}x_N) =\lambda W(x_1,
\dots, x_N).
\end{equation}
 We will call $q_j$ the \emph{weight}\glossary{qj@$q_j$ & The weight (charge)
   of the variable $x_j$}\glossary{qy@$q_y$ & The weight of the
   variable $y$} of $x_j$.
   We define
$d$ and $n_i$ for $i\in \{1,\dots,N\}$ to be the unique positive
integers such that $(q_1,\dots,q_N) = (n_1/d, \dots, n_N/d)$ with
$\gcd(d,n_1,\dots,n_N) =1$.

\end{df}

Throughout this paper we will need a certain nondegeneracy condition on $W$.
\begin{df}\label{df:nondegenerate} We call $W$ \emph{nondegenerate}
if \begin{enumerate}
\item\label{it:nondegen-unique}
$W$ contains no monomial of the form $x_ix_j$, for $i\neq j$ and
\item The hypersurface defined by $W$ in weighted
projective space is non-singular, or, equivalently, the affine
hypersurface defined by $W$ has an isolated singularity at the
origin.
  \end{enumerate}
    \end{df}
The following proposition was pointed out to us by N.~Priddis and follows from \cite[Thm 3.7(b)]{HK}.
\begin{prop}
If $W$ is a non-degenerate, quasi-homogeneous polynomial, then the weights
   $q_i$ are bounded by $q_i\leq \frac{1}{2}$ and are unique.
   \end{prop}
   
From now on, when we speak of a quasi-homogeneous polynomial $W$,
we will assume it to be nondegenerate.

\begin{df}
Write the polynomial $W = \sum^s_{j=1}W_j$
as a sum of monomials
$W_j = c_j \prod_{\ell=1}^N x_\ell^{b_{j\ell}}$, with $b_{j\ell} \in \Z^{\ge 0}$,
and with $c_j \neq 0$.  Define the $s\times N$ matrix
\begin{equation}\label{eq:B}
B_W := (b_{j \ell }),
\end{equation}
and let $B_W=VTQ$ be the \emph{Smith normal form} of $B_W$ \cite[\S12 Thm 4.3]{Art}.  That is, $V$ is an $s\times s$ invertible integer matrix and $Q$ is an $N\times N$ invertible integer matrix. The matrix $T = (t_{j \ell })$ is an  $s\times N$ integer matrix with $t_{j\ell} = 0$ unless $\ell=j$, and $t_{\ell, \ell}$ divides $t_{\ell+1, \ell+1}$ for each $\ell \in \{1,\dots,N-1\}$.
\end{df}

\begin{lm}\label{lm:group}
If $W$ is nondegenerate, then the group $$G_W:=\{(\alpha_1,
\dots, \alpha_N) \in (\mathbb{C}^*)^N \, | \ W (\alpha_1 x_1,
\dots, \alpha_N x_N)=W(x_1, \dots, x_N)\}$$\glossary{G@$G$ & The group of diagonal symmetries of $W$, usually the full group, but sometimes a subgroup}\glossary{GW@$G_W$ &The maximal group of diagonal symmetries of $W$}
 of diagonal symmetries
of $W$ is finite.
     \end{lm}

\begin{proof}
The uniqueness of the weights $q_i$ is equivalent to saying that
the matrix  $B_W$ has rank $N$.  We may as well assume that $B_W$ is invertible.  Now write $\gamma = (\alpha_1, \dots, \alpha_N) \in
\grp$, as $\alpha_j = \exp(u_j + v_j i)$ for $u_j\in \R$ uniquely
determined, and $v_j \in \R$ determined up to integral multiple of
$2\pi i$. The equation $W(\alpha_1x_1, \dots, \alpha_N x_N) =
W(x_1, \dots, x_N)$ can be written as $B_W (\mathbf{u}+\mathbf{v}i)
\equiv \mathbf{0} \pmod{2\pi i}$, where
$\mathbf{u}+\mathbf{v}i=(u_1+v_1 i,\dots,u_N+v_N i)$ and
$\mathbf{0}$ is the zero vector.  Invertibility of $B_W$ shows that
$u_\ell = 0$ for all $\ell$.  Thus $\grp$ is a subgroup of $U(1)^N$, and a
straightforward argument shows that the number of solutions
(modulo $2\pi i$) to the equation $B_W (\mathbf{v} i) \equiv
\mathbf{0} \pmod{2\pi i}$ is also finite.
\end{proof}

\begin{df}\label{def:J}\label{def:theta}
We write each element $\gamma \in G_W$  (uniquely) as $$\gamma =
( \exp(2 \pi i \Theta_1^{\gamma}), \dots, \exp(2 \pi i
\Theta_N^{\gamma})),$$\glossary{Theta@$\Theta^{\gamma}_i$ & The $\log$ of the action of $\gamma$ on $x_i$}
 with $\Theta_i^{\gamma} \in [0, 1) \cap
\Q$.

There is a special element $J$\glossary{J@$J$ & The exponential grading element} of the group $G_W$ which is
defined to be $$J:=(\exp(2\pi i q_1), \dots, \exp(2\pi i q_N)),$$
where the $q_i$ are the weights defined in
Definition~\ref{df:nondegenerate}.  Since $q_i\neq 0$ for all $i$,
we have $\Theta_i^{J} = q_i$.  By definition, the order of the element $J$  is $d$.

The element $J$ will play an important role in the remainder of the paper.
\end{df}

For any $\gamma\in G_W$, let
    $\C^{N_{\gamma}}:=\left(\C^N\right)^\gamma$ be the set of fixed points
    of $\gamma$ in $\C^N$, let $N_{\gamma}$
    denote its complex dimension, and let
    $W_\ga:=\restr{W}{\C^{N_\gamma}}$ be the quasi-homogeneous singularity restricted to the fixed point
    locus of $\ga$.  The polynomial $W_{\gamma}$ defines a quasi-homogeneous
    singularity of its own in $\C_\gamma^{N_\ga}$, and $W_{\gamma}$ has its own Abelian
    automorphism group. However, we prefer to think of the original group $G_W$
    acting on $\C^{N_{\gamma}}$.  Note that $G_W$ preserves the subspace $\C^{N_\ga}\subseteq \C^N$.

\begin{lm}
If $W$ is a non-degenerate, quasi-homogeneous polynomial, then for any $\ga\in G_W$, the polynomial $W_\ga$ has no non-trivial critical points.  Therefore, $W_{\gamma}$ is itself a non-degenerate, quasi-homogeneous polynomial in the variables fixed by $\ga$.

\end{lm}
\begin{proof}
Let $\mfm \subset \C[x_1,\dots,x_N]$ be the ideal generated by the variables not fixed by $\ga$, and write $W$ as
$W = W_\ga + W_{moved}$, where $W_{moved}\in \mfm$.
In fact, we have $W_{moved}\in \mfm^2$  because if any monomial in $W_{moved}$ does not lie in $\mfm^2$, it can be written as $x_mM$, where $M$ is a monomial fixed by $\ga$.  However, $\ga\in G_W$ acts diagonally, and it must fix $W$, and hence it must fix every monomial of $W$, including $x_m M$.  Since it fixes $M$ and $x_mM$, it must also fix $x_m$---a contradiction.  This shows that $W_{moved}\in \mfm^2$.

Now we can show that there are no non-trivial critical points of $W_\ga$.    For simplicity, re-order the variables so that $x_1,\dots,x_\ell$ are the fixed variables, and $x_{\ell+1},\dots,x_N$ are the remaining variables.
If there were a non-trivial critical point of $W_\ga$, say $(\alpha_1,\dots,\alpha_\ell)\in \C^{\ell}$, then the point $(\alpha_1,\dots,\alpha_\ell,0,\dots,0) \in \C^N$ would be a non-trivial critical point of $W$.  To see this, note that for any $i\in \{1,\dots,N\}$ we have  $$\restr{\frac{\partial W_{moved}}{\partial x_i}}{ (\alpha_1,\dots,\alpha_\ell,0,\dots,0)} = 0,$$ since $W_{moved}\in \mfm^2$.  This gives
$$
\restr{\frac{\partial W}{\partial x_i}}{ (\alpha_1,\dots,\alpha_\ell,0,\dots,0)} = \restr{\frac{\partial W_\ga}{\partial x_i}}{ (\alpha_1,\dots,\alpha_\ell)} + \restr{\frac{\partial W_{moved}}{\partial x_i}}{ (\alpha_1,\dots,\alpha_\ell,0,\dots,0)}\\
= 0,
$$
which shows that $(\alpha_1,\dots,\alpha_\ell,0,\dots,0)$ is a non-trivial critical point of $W$---a contradiction.
\end{proof}

\subsubsection{$W$-structures on an orbicurve}\label{sec:W-struct}

A $W$-structure on an orbicurve $\cC$ is essentially a choice of $N$ line bundles $\LL_1,\dots,\LL_N$ so that for each monomial $W_j = x_1^{b_{j,1}}\cdots x_N^{b_{j,N}}$ we have an isomorphism of line bundles
$$\varphi_j:\LL_1^{\otimes b_{j,1}}\cdots \LL_N^{\otimes b_{j,N}} \rTo K_{\log}.$$
However, the isomorphisms $\varphi_j$ need to be compatible, in the sense that at any point $p$ there exists a trivialization $\restr{\LL_i}{p} \cong \C$ for every $i$ and  $\restr{K_{\log}}{p}\cong\C \cdot dz/z$ such that for all $j\in\{1,\dots,s\}$ we have $\varphi_j(1,\dots,1) = 1\cdot dz/z \in \C$.  If $s=N$ we can choose such trivializations for any choice of maps $\{\varphi_j\}$, but if $s>N$ then the choices of $\{\varphi_j\}$ need to be related. To do this we use the Smith normal form to give us a sort of minimal generating set of isomorphisms which will determine all the maps $\{\varphi_j\}$.

\begin{df}\label{df:W-structure}
For any nondegenerate, quasi-homogeneous polynomial $W \in
\C[x_1, \dots, x_N]$, with matrix of exponents $B_W=(b_{\ell j})$ and Smith normal form $B_W=VTQ$, let $A:=(a_{j \ell}):=V^{-1}B=TQ$, and let
$u_\ell$ be the sum of the entries in the $\ell$th row of $V^{-1}$ (i.e., the $\ell$th term in the product $V^{-1} (1,1,\dots,1)^T$).

For any $\ell\in \{1,\dots,N\}$ denote by $A_{\ell}(\LL_1,\dots,\LL_N)$ the tensor product $$A_{\ell}(\LL_1,\dots,\LL_N):=\LL_1^{\otimes a_{\ell 1}} \otimes\cdots \otimes \LL_N^{\otimes a_{\ell N}}$$

We define a \emph{$W$-structure} on an
orbicurve $\cC$ to be the data of an $N$-tuple $(\LL_1, \dots,
\LL_N)$ of orbifold line bundles on $\cC$ and
 isomorphisms $$\vpt_\ell:A_{\ell}(\LL_1,\dots,\LL_N)  \irightarrow K_{\cC,\log}^{u_\ell}$$ for every  $\ell \in \{1,\dots,N\}$.

Note that for each point $p \in \cC$,  an orbifold line bundle
$\LL$ on $\cC$ induces a representation $G_p \rTo \aut(\LL)
$.  Moreover, a $W$-structure on $\cC$ will induce a
representation $r_{p}: G_p \rTo   U(1)^N$.
For all our $W$-structures we require that this representation
$r_p$ be faithful at every point.
\end{df}
The next two propositions follow immediately from the definitions.
\begin{prop}\label{prop:SmithUnique}
The Smith normal form is not necessarily unique, but for any two choices of Smith normal form $B=VTQ = V'T'Q'$ a $W$-structure $(\LL_1,\dots,\LL_N,\vpt_1,\dots,\vpt_N)$  with respect to $VTQ$ induces a canonical $W$-structure $(\LL_1,\dots,\LL_N,\vpt'_1,\dots,\vpt'_N)$ with respect to $V'T'Q'$, where the isomorphism $\vpt'_i$ is given by $$\vpt'_i = \vpt_1^{z_{i1}}\otimes\cdots\otimes\vpt_N^{z_{iN}},$$ and where $Z = (z_{ij}) := (V')^{-1}V$.
\end{prop}
\begin{prop}\label{prop:SmithWStruct}
For each $j\in\{1,\dots,s\}$ the maps $\{\vpt_\ell \}$ induce  an isomorphism
\begin{equation}\label{eq:SmithWStruct}
\begin{split}
\varphi_j:= &
\vpt_1^{v_{j 1}}\otimes\cdots \otimes\vpt_N^{v_{j N}}: W_j(\LL_1,\dots,\LL_N) \\
&
=
\LL^{\otimes b_{j 1}}_1 \otimes \cdots
\otimes \LL^{\otimes b_{j,N}}_N =\LL_1^{\sum_\ell v_{j \ell} a_{\ell 1}}\otimes \cdots \LL_N^{\sum_\ell v_{j \ell} a_{\ell N}} \rTo K_{\cC,\log},\end{split}
\end{equation}
where $V=(v_{j\ell})$.

Moreover, if $B$ is square (and hence invertible), then a choice of isomorphisms $\varphi_j: \LL^{\otimes b_{j 1}}_1 \otimes \cdots
\otimes \LL^{\otimes b_{j,N}}_N \rTo K_{\cC,\log}$ for every $j\in \{1,\dots,N\}$ is equivalent to a choice of isomorphisms $\vpt_\ell:\LL_1^{\otimes a_{\ell 1}} \otimes\cdots \otimes \LL_N^{\otimes a_{\ell N}} \rTo K_{\cC,\log}^{u_\ell}$ for every  $\ell \in \{1,\dots,N\}$.

Finally, the induced maps $\varphi_j: \LL_1^{b_{j,1}}\otimes \cdots\otimes \LL_N^{b_{j,N}} \rTo K_{\cC,\log}$ are independent of the choice of Smith normal form $VTQ$
\end{prop}
For the rest of this paper, we will assume that a choice of Smith normal form $B_W=VTQ$ has been fixed for each $W$.  
\begin{df}\label{df:W-isom}
Given any two $W$-structures $\fL:=(\LL_1, \dots, \LL_N, \vpt_1,
\dots, \vpt_N)$\glossary{Lfrak@$\fL$ & A $W$-structure} and $\fL':=(\LL'_1, \dots, \LL'_N, \vpt'_1,
\dots, \vpt'_N)$ on $\cC$, any set of morphisms
$\xi_j:\LL_j \rTo \LL'_j$ of orbifold line bundles for
$j\in\{1, \dots, N\}$ will induce a morphism
$$\Xi_l:\LL_1^{a_{\ell 1}} \otimes\cdots \otimes \LL_N^{a_{\ell N}} \rTo{}{}  {\LL'}_1^{a_{\ell 1}} \otimes\cdots \otimes {\LL'}_N^{a_{\ell N}}$$
 for every $l
\in \{1,\dots, s\}$.

An \emph{isomorphism of $W$-structures} $\Upsilon:\fL\rTo \fL'$ on $\cC$ is defined to be a collection of isomorphisms $\xi_j:\LL_j \rTo \LL'_j$ such that for every $\ell\in \{1,\dots,N\}$ we have $\vpt_\ell = \vpt'_\ell \circ  \Xi_\ell$.
\end{df}

It will be important later to know that different choices of maps $\{\vpt_j\}$ all give isomorphic $W$-structures.
\begin{prop}
For a given orbicurve $\cC$,  any two $W$-structures $\fL_1:=(\LL_1,\dots,\LL_N,\vpt_1,\dots,\vpt_N)$ and $\fL_2:=(\LL_1,\dots,\LL_N,\vpt'_1,\dots,\vpt'_N)$ on $\cC$ which have identical bundles $\LL_1,\dots,\LL_N$ 
are isomorphic.
\end{prop}
\begin{proof}
For each $j\in\{1,\dots,N\}$ the composition $\vpt^{-1}_j\circ \vpt'_j$ is an automorphism of $K_{\log}^{u_j}$ and hence defined by an element $\exp(\alpha_j) \in \C^*$.

Since $B:=B_W$ is of maximal rank,  the product $TQ$ in the Smith normal form decomposition of B consists of a non-singular $N\times N$ block $C$ on top, with all remaining rows identically equal to zero.
$$ V^{-1}B = TQ =
\left(\begin{array}{c}
C\\
-\\
\boldsymbol{0}\\
\end{array}\right).
$$
Let $(\beta_1,\dots,\beta_N)^T := C^{-1}(\alpha_1,\dots,\alpha^N)^T \in \Q^{N}$. For every $\ell\in\{1,\dots,N\}$ the collection of automorphisms  $\{\exp(\beta_j):\LL_j\rTo\LL_j\}$ induces the automorphism $\exp(\sum_{i=1}^N a_{\ell i} \beta_i) = \exp(\alpha_\ell)$ on $\LL_1^{\otimes a_{\ell 1}} \otimes\cdots \otimes \LL_N^{\otimes a_{\ell N}}$ and hence takes $\vpt_\ell$ to $\exp(\alpha_\ell)\vpt_\ell = \vpt'_\ell$. Thus the collection $\{\exp(\beta_j)\}$ induces an isomorphism of $W$-structures $\fL_1 \irightarrow \fL_2$.
\end{proof}

\begin{ex}
In the case where $W=x^r$ (the $A_{r-1}$ singularity), a $W$-structure is an \emph{$r$-spin structure} (see \cite{AJ}).
\end{ex}

\begin{df} For each orbifold marked point $p_i$ we will denote the  image $r_{p_i}(1)$ of the canonical generator $1\in \Z/m_i \cong G_{p_i}$ in $U(1)^N$ by
$$\ga_i:=\ga_{p_i} :=r_{p_i}(1) = (\exp(2\pi i \Theta^\ga_1), \dots, \exp(2\pi i \Theta^\ga_N)).$$
\end{df}

The choices of orbifold structure for the line bundles in the $W$-structure is severely restricted by $W$.
\begin{lm}
 Let $(\LL_1, \dots, \LL_N, \vpt_1,\dots,\vpt_N)$ be a
$W$-structure on an orbicurve $\cC$  at an orbifold point $p \in
\cC$.  The faithful representation $r_p: G_p \rTo U(1)^N$ factors through $G_W$, so $\ga_i \in G_W$ for all $i\in \{1,\dots,k\}$.
\end{lm}

\begin{proof}
Recall that for each $j \in \{1,\dots, s\}$, the bundle $W_j(\LL_1, \dots, \LL_N)= \LL^{\otimes b_{1,j}}_1 \otimes \dots
\otimes \LL^{\otimes b_{N,j}}_N$ is isomorphic to $K_{\log}$, and so the local group acts trivially on it.  However, the generator $\ga_p \in G_p$ acts on  $W_j(\LL_1, \dots, \LL_N) $ as
$\exp(2\pi i \sum_i b_{ij}\Theta_i^\gamma)$. Therefore $\sum_i b_{ij}\Theta_i^\gamma\in \Z$, and $\gamma$ fixes $W_j$.
\end{proof}

\begin{df}
A marked point $p$ of a $W$-curve is called \emph{narrow} if the fixed point locus $\fix(\gamma) \subseteq \C^N$ is just $\{0\}$.  The point $p$  is called \emph{broad} otherwise.
\end{df}

\begin{rem}\label{rem:principal-gamma}
Note that for any given orbicurve $\cC$, any two $W$-structures on
$\cC$ differ by line bundles $\N_1, \dots, \N_N$ with isomorphisms
$\xi_j:\N_1^{\otimes a_{\ell 1}} \otimes\cdots \otimes \N_N^{\otimes a_{\ell N}} \irightarrow \O_{\cC}$.  The set of
such tuples $(\N_1, \dots, \N_N, \xi_1, \dots, \xi_s)$, up to
isomorphism, is a group under tensor product, and is isomorphic to
the (finite) cohomology group $H^1(\cC, \grp)$.  Thus the set of
$W$-structures on $\cC$ is an $H^1(\cC, G_W)$-torsor.
\end{rem}

An automorphism of a $W$-curve $\fL$ induces an automorphism of the
 orbicurve $\cC$ and underlying (coarse) curve $C$. It is easy to see that the group of automorphisms  $\fL$ which fix the underlying (coarse) curve $C$ consists of
all elements in the group $G_W$, acting by multiplication of the
fibers of $\LL_1,\dots, \LL_N$.  This gives the exact sequence
$$1\rTo \aut_{C}(\fL) = G_W \rTo \aut(\fL)\rTo \aut(C).$$
More generally, if the stable curve $\cC$ has irreducible components $\cC_l$ for $l\in\{1,\dots,t\}$ and nodes $\nu\in E$, we denote by $\fL_i$
the restriction to $\cC_i$ of the $W$-structure.
To describe the automorphisms of the $W$-structure in this case, it will be convenient to choose an orientation on the edges of the dual graph of $C$. This amounts to choosing, for each node $\nu \in E$, one of the components passing through $\nu$ to be designated as $C_{\nu_+}$.  The other component passing through $\nu$ is designated $C_{\nu_-}$.  If the same irreducible component $C_i$ passes through $\nu$ twice, then that component will be designated both $C_{\nu_+}$ and $C_{\nu_-}$. The final result will be independent of these choices.

Let $G_\nu$ denote the local group at the node $\nu$.  Any element $g\in \aut_{C_i}(\fL_i)$ induces (by restriction) elements $g_{\nu_+}$ and $g_{\nu_-}$ in $G_\nu$. We define $\delta:\prod_i \aut_{C_i}(\fL_i) \rTo \prod_\nu G_\nu$ to be the homomorphism defined as $(\delta(g))_\nu = g_{\nu_+}g_{\nu_-}^{-1}$.  We have an exact sequence
\begin{equation}\label{eq:AutExactSeq}
1 \rTo \aut_{C}\fL \rTo \prod_i \aut_{C_i}(\fL_i) \rTo \prod_{\nu\in E} G_\nu.
\end{equation}

\begin{ex}
Consider a $W$-curve with two irreducible components $\cC_1$ and $\cC_2$ with marked points $\{p_i| i \in I_1 \}\cup \{q_+\} \subset \cC_1$ and $\{p_i| i \in I_2 \}\cup \{q_-\} \subset \cC_2$, such that the components meet at a single node $q = q_{+}=q_{-}$ and such that $I_1 \sqcup I_2 = \{1,\dots,k\}$.  Denote the local group at $q_{\pm}$ by $\lgr{\pm}$.  Note that $\ga_{-} = \ga_{+}^{-1}$.
In this case we have
\begin{equation}\label{eq:TreeAut}
\aut_{C} (\fL) = G_W \times_{G/\langle\gamma_+\rangle} G_W,
\end{equation}
where $G_W\times_{G_W/\lgr{+}} G_W$\glossary{GxG@$G_W\times_{G_W/\lgr{}}G_W$ & The group of pairs $(g_1,g_2)\in G^2_W$ with $g_1=g_2\in G_W/\lgr{+}$}
 denotes the group of pairs $(g_1,g_2)$ such that the images of $g_1$ and $g_2$ are equal in $G_W/\lgr{+}$.
\end{ex}
\begin{ex}
If $\cC$ consists of a single (possibly nodal) irreducible component, then we have
\begin{equation}\label{eq:loopAut}
\aut_{C} (\fL) = G_W.
\end{equation}
\end{ex}

\subsubsection{Pushforward of $W$-structures}

We need to understand the behavior of $W$-structures when forgetting the orbifold structure at marked points, that is, when they are pushed down to the underlying (coarse) curve.

Consider, as an initial example, the case of $W=x^r$, so that a $W$-structure consists of a line bundle $\LL^r\cong K_{\log}$.  Near an orbifold point $p$ with local coordinate $z$ the canonical generator $1\in\Z/m \cong G_p$ of the local group $G_p$ acts on $\LL$ by $(z,s) \mapsto (\exp(2 \pi i/m) z, \exp(2\pi i (v/m)) s)$ for some $v \in \{0, \dots, m-1\}$.   Since $K_{\log}$ is invariant under the local action of $G_p$, we must have $rv=\ell m$ for some  $\ell\in \{0,\dots, r-1\}$, and $\frac{v}{m}=\frac{\ell}{r}$.  Denote the (invariant) local coordinate on the underlying curve $C$ by $u=z^m$.  Any section in $\sigma\in\Omega^0(|\LL|)$ must locally be of the form $\sigma=g(u)z^{v}s$, in order to be $\Z/m$-invariant.  So
$\sigma^r$ has local representative $z^{r v
} g^r(u)\frac{dz}{z}=u^{\ell}g^r(u) \frac{du}{mu}$. Hence, $\sigma^r\in
\Omega^0(K_{\log} \otimes \O ((-\ell) p)$, and thus when $\ell \neq 0$, we have $\sigma^r\in \Omega^0(K)$.

\begin{rem}\label{rem:zeros-of-powers}
More generally, if $\LL^r \cong K_{\log}$ on a \emph{smooth} orbicurve with
action of the local group on $L$ defined by $\ell_i >0$ (as above) at
each marked point $p_i$, then we have $$(\varrho_* \LL)^r = |\LL|^r = K_{\log} \otimes \left(\bigotimes_i \O((-\ell_i) p_i)\right) = K_{\log} \otimes \left(\bigotimes_i \O((-r(v/m)) p_i)\right).$$ \end{rem}

\begin{prop}\label{prop:pushforward-W}
Let $(\LL_1, \dots, \LL_N, \vpt_1,\dots,\vpt_N)$ be a
$W$-structure on an orbicurve $\cC$ which is smooth (the
underlying curve $C$ is nonsingular) at an orbifold point $p \in
\cC$.  Suppose also that the local group $G_p\cong \Z/m$ of $p$
acts on $\LL_j$ by
$$\gamma =(\exp (2 \pi i\Theta_1^\gamma),\dots,\exp (2 \pi
i \Theta_N^\gamma));$$
that is,
$\exp(2 \pi i/m)(z,w_j)=(\exp(2 \pi i/m)z, \exp(2
\pi i \Theta_j^\gamma) w_j)$ with $1> \Theta_j^\gamma \ge0$.

Let $\cCb$ denote the orbicurve obtained from $\cC$ by making the orbifold structure at $p$ trivial (but retaining the orbifold structure at all other points).  Let $\varrhob:\cC \rTo \cCb$ be the obvious induced morphism, and let $\varrhob_*(\LL)$ denote the pushforward to $\cCb$ of an orbifold line bundle $\LL$ on $\cC$.

For any isomorphism $\psi:\LL_1^{e_1}\otimes\cdots\otimes\LL_N^{e_N}\rTo K_{\log}$ we have an induced isomorphism on the pushforward
\begin{equation}\label{eq:pushf}
\varrhob_*(\LL_1)^{e_1}\otimes\cdots\otimes\varrhob_*(\LL_N)^{e_N}\rTo  K_{\cCb, \log}\otimes
\O\left(-\sum^N_{j=1}e_j\Theta_j^\gamma
p\right).
\end{equation}

If $\cC$ is a smooth orbicurve (i.e., $C$ is a smooth curve), let
$\gamma_\ell$ define the action of the local group $G_{p_\ell}$ near $p_\ell$. For any isomorphism $\psi:\LL_1^{e_1}\otimes\cdots\otimes\LL_N^{e_N}\rTo K_{\log}$ we have a (global) induced isomorphism
$$|\psi|:|\LL_1|^{e_1}\otimes\cdots\otimes|\LL_N|^{e_N}\rTo  K_{\cCb, \log}\otimes
\O\left(-\sum_{\ell=1}^k\sum^N_{j=1}e_j\Theta_j^{\gamma_\ell}
p_\ell\right).$$

In particular, for every monomial $W_i$, the isomorphism of Equation~(\ref{eq:SmithWStruct}) induces an isomorphism
\begin{equation}
\label{eq:desing-bundles}
W_i(|\LL_1|,\dots, |\LL_N|)\cong K_{C, \log}\otimes
\O\left(-\sum_{\ell=1}^k\sum^N_{j=1}b_{ij}\Theta_j^{\gamma_\ell}
p_{\ell}\right)
\end{equation}
\end{prop}

\begin{proof}
Equation~(\ref{eq:pushf})  is a
straightforward generalization of the argument given above when $W=x^r$, the description of $\gamma$ as
$\gamma=(\exp(2 \pi i \Theta_1^{\gamma},\dots,\exp(2 \pi i
\Theta_N^{\gamma})))$, and the description of $|\LL_j|$ in terms of
the action of the local group $G_p$ given above.
\end{proof}

\subsection{Moduli of stable $W$-orbicurves}

\begin{df}\label{df:stable-W-curve}
A pair $\frkc=(\cC,\fL)$\glossary{Cfrak@$\frkc$ & A $W$-curve $(\cC,\fL)$} consisting of  an orbicurve $\cC$ with $k$ marked points and with
$W$-structure $\fL$
is called a \emph{stable
$W$-orbicurve}
if the underlying curve $C$ is a stable curve, and if
for each point $p$ of $\cC$ the representation $r_{p}: G_p
\rTo G_W$ is faithful.
\end{df}
\begin{df}
A \emph{genus-$g$, stable $W$-orbicurve with $k$ marked points over a
base $T$} is given by a flat family of genus-$g$, $k$-pointed
orbicurves $\cC \rTo T$ with (gerbe) markings $\cS_i \subset \cC$
and sections $\sigma_i: T \rTo \cS_i$, and the data $(\LL_1,\dots,\LL_N,\vpt_1,\dots,\vpt_N)$.  The sections $\sigma_i$ are required to
induce isomorphisms between
$T$ and the coarse moduli of $\cS_i$ for $i\in\{1,\dots,k\}$.  The $\LL_i$ are  orbifold line bundles on $\cC$.  And the $\vpt_j : A_j(\LL_1,\dots, \LL_N) \irightarrow
K_{\cC/T, log}^{u_j}:=(K_{\cC/T}(\sum \cS_i))^{u_j}$ are isomorphisms to the
$u_j$-fold tensor power of the relative log-canonical bundle
which, together with the $\LL_i$, induce a
$W$-structure on every fiber $\cC_t$.
\end{df}

\begin{df}
A \emph{morphism of stable $W$-orbicurves} $(\cC/T, \cS_1, \dots,
\cS_k,\LL_1,\dots,\LL_N,
\vpt_1,\dots,\vpt_N)$
and $(\cC'/T', \cS'_1, \dots, \cS'_k, \LL'_1,\dots, \LL'_N, \vpt'_1,
\dots, \vpt'_s)$
 is a tuple of morphisms $(\tau,\mu,
\alpha_1,\dots, \alpha_N)$ such that the pair $(\tau,\mu)$ forms
a morphism of pointed orbicurves:
$$\begin{diagram}
\cC&\rTo^{\mu}&\cC'\\
\dTo & & \dTo\\
T &\rTo^{\tau}&T'\\
\end{diagram}$$
and the $\alpha_j :\LL_j \irightarrow \mu^* \LL'_j$ are isomorphisms of line bundles which form an isomorphism of $W$-structures on $\cC$ (see Definition~\ref{df:W-isom}).
\end{df}

\begin{df}
For a given choice of non-degenerate $W$, we denote the stack of stable $W$-orbicurves by
$\W_{g,k}(W)$\glossary{Mgkw@$\W_{g,k}$ & The stack of $k$-pointed, genus-g $W$-curves}.  If the choice of $W$ is either clear or is unimportant, we simply write $\W_{g,k}$.
\end{df}
\begin{rem}
This definition depends on the choice of Smith normal form $B=VTQ$, but by Proposition~\ref{prop:SmithUnique} any other choice of Smith normal form for the same polynomial $W$ will give a canonically isomorphic stack.  
\end{rem}

Forgetting the $W$-structure and the orbifold structure gives a morphism $$\st:\W_{g,k
} \rTo \MM_{g,k}.\glossary{st@$\st$ & The stabilization morphism, which forgets the $W$-structure}$$  The morphism $\st$ plays a role similar to that played by the {stabilization} morphism of stable maps.  It is quasi-finite by Remark~\ref{rem:principal-gamma}.

\begin{thm}
For any nondegenerate, quasi-homogeneous polynomial $W$, the stack $\W_{g,k}$ is a smooth, compact orbifold (Deligne-Mumford stack) with projective coarse moduli.  In particular, the morphism $st:\W_{g,k} \rTo \MM_{g,k}$ is flat, proper and quasi-finite (but not representable).  \end{thm}
\begin{proof}
Denote the classifying stack of $\C^*$ by $\B\C^*$. For each orbicurve $\cC$ the line bundle $K_{\log}$ corresponds to a $1$-morphism $$\cC \rTo \B\C^*,$$ and composing with the diagonal embedding $\Delta:\B\C^* \rTo (\B\C^*)^N$, we have
\begin{equation}
\delta:={\Delta \circ K_{\log}}:\cC \rTo (\B\C^*)^N.
\end{equation}
Furthermore, each isomorphism $\vpt_i$ induces a $1$-morphism $(\B\C^*)^N \rTo \B\C^*$, and together they yield a morphism \begin{equation}
\Phi_W: (\B\C^*)^N \rTo (\B\C^*)^N.
\end{equation}
It is easy to see that the data of a $W$-structure on $\cC$ is equivalent to the data of a representable $1$-morphism $$\mathfrak{L}:\cC \rTo (\B\C^*)^N,$$ which makes the diagram
\begin{diagram}
        &                               & (\B\C^*)^N\\
        & \ruTo^{\mathfrak{L}}  & \dTo_{\Phi_W} \\
\cC & \rTo^{\delta} & (\B\C^*)^N
\end{diagram}
commute.

As in \cite[\S1.5]{AJ} we let $C_{g,k} \rTo \MM_{g,k}$ denote the universal curve, and we consider the stack $$C_{g,k,W}:=C_{g,k} \mathbin{\mathop{\times}\limits_{(\B\C^*)^N}} (\B\C^*)^N,$$ where the fiber product is taken with respect to $\delta$ on the left and $\Phi_W$ on the right.  The stack
$C_{g,k,W}$ is an \'etale gerbe over $C_{g,k}$ banded by $G_W$.  In particular, it is a Deligne-Mumford stack.

Any $W$-curve $(\cC/S, p_1,\dots,p_k,\LL_1,\dots,\LL_N,\vpt_1,\dots,\vpt_s)$ induces a representable map $\cC \rTo C_{g,k,W}$ which is a balanced twisted stable map. The homology
class of the image of the coarse curve $C$ is the class $F$ of a fiber of the
universal curve $C_{g,k}\rTo \MM_{g,k}$. The family of coarse curves $C \rTo S$
gives rise
to a morphism $S \rTo \MM_{g,k}$ and we have an isomorphism $C
\cong S \times_{\MM_{g,k}}C_{g,k}$. We thus have a base-preserving
functor from the stack
$\W_{g,k}$ of $W$-curves to
the stack $\K_{g,k}(C_{g,k,W}/
\MM_{g,k}, F)$ of balanced, $k$-pointed twisted stable maps
of genus $g$ and class $F$ into $C_{g,n,W}$ relative to the
base stack $\MM_{g,k}$ (see \cite[\S8.3]{AV}).
The image lies in the closed substack where
the markings of $C$ line up over the markings of $C_{g,n}$.  It is easy to see that the resulting functor  is an equivalence.  Thus $\W_{g,k}$  is
a proper Deligne-Mumford stack admitting a projective coarse moduli space.

Smoothness of the stack $\W_{g,k}$ follows, as in the $A_n$ case (see \cite[Prop 2.1.1]{AJ}), from the fact that the relative cotangent complex $\mathbb{L}_{\Phi_W}$ of $\Phi_W:(\B\C^*)^N \rTo (\B\C^*)^N$ is trivial.  That means that the deformations and obstructions of a $W$-curve are identical to those of the underlying orbicurves, but these are known to be unobstructed (see \cite[\S2.1]{AJ}).
\end{proof}

\subsubsection{Decomposition of $\W_{g,k}$ into components}
The orbifold structure, and the image $\ga_i = r_{p_i}(1)$ of the canonical generator $1\in\Z/m_i\cong G_{p_i}$ at each marked point $p_i$ is locally constant, and hence are constant for each component of $\W_{g,k}$. Therefore, we can use these decorations to
decompose the moduli space into components.

\begin{df}\label{df:type}
For any choice  $\bgamma:=(\gamma_1, \dots, \gamma_k) \in \grp^k$ we define $\W_{g,k}(\bgamma)\subseteq \W_{g,k}$ to be the open and closed substack with orbifold
decoration $\bgamma$.\glossary{MMgkwgamma@$\W_{g,k}(\bgamma)$ & The stack of stable $W$-curves of genus $g$, with $k$ marked points and type $\bgamma$.}
We call $\bgamma$ the \emph{type} of any $W$-orbicurve in
$\W_{g,k}(\bgamma)$.
\end{df}
We have the decomposition
$$\W_{g,k}=\sum_{\bgamma} \W_{g,k}(\bgamma).$$

Note that by applying the degree map to Equation~(\ref{eq:desing-bundles}) we gain an important selection rule.
 \begin{prop}
A necessary and sufficient condition for $\W_{g,k}(\bgamma)$ to be non-empty is
\begin{equation}\label{eq:deg-sel-rule}
q_j(2g - 2 + k)-\sum^k_{l=1}\Theta_j^{\gamma_l} \in \Z.
\end{equation}
\end{prop}
\begin{proof}
Although the degree of an orbifold bundle on $\cC$ may be a rational number, the degree of the pushforward $\varrho_*\LL_j=|\LL_j|$ on the underlying curve $C$  must be an integer, so for all $i\in\{1,\dots,s\}$ the following equations must hold for integral values of $\deg(|\LL_j|)$:
\begin{equation}
\sum_{j=1}^N {b_{ij}}\deg(|\LL_j|) = 2g - 2 + k
-\sum^k_{l=1}\sum^N_{j=1}b_{ij}\Theta_j^{\gamma_l}.
\end{equation}
Moreover, because $W$ is nondegenerate, the weights $q_j$ are
uniquely determined by the requirement that they satisfy the equations $\sum_{j=1}^N {b_{ij}}q_j =
1$ for all
$i \in \{1,\dots,s\}$,  so we find that for every $j\in\{1,\dots,N\}$ we have
\begin{equation}\label{eq:sel-rule}
\deg(|\LL_j|) = \left(q_j(2g - 2 + k)
-\sum^k_{l=1}\Theta_j^{\gamma_l} \right)\in \Z.
\end{equation}
Conversely, if the degree condition~(\ref{eq:deg-sel-rule}) holds, then for any smooth curve $C$ (not orbifolded) we may choose line bundles $E_1,\dots,E_N$ on $C$ with $\deg(E_j) = q_j(2g - 2 + k)-\sum^k_{l=1}\Theta_j^{\gamma_l}$ for each $l$.  If we take $A=(a_{ij}) = V^{-1}B$ and $u = (u_i) = V^{-1}(1,\dots,1)^T$ as in Definition~\ref{df:W-structure}, then
for each $i\in \{1,\dots, s\}$ we have a line bundle $$X_i:=E_1^{a_{i,1}}\otimes\cdots \otimes E_N^{a_{i,N}} \otimes K_{C,\log}^{-u_i}\otimes
\O\left(\sum^k_{l=1}\sum^N_{j=1}a_{ij}\Theta_j^{\gamma_l}p_l\right)$$ 
and $\deg(X_i)$ satisfies
 $$
\left(\begin{array}{c}
\deg(X_1)\\
\vdots \\
\deg(X_N)
\end{array}\right) =A 
\left(\begin{array}{c}
q_1\\
\vdots\\
q_N
\end{array} \right)(2g - 2 + k)-\sum^k_{l=1}A 
\left(\begin{array}{c}
\Theta_1^{\ga_l}\\
\vdots\\
\Theta_N^{\ga_l}
\end{array}
\right)
- V^{-1}
\left(\begin{array}{c}
1\\
\vdots \\
1
\end{array}\right)
 (2g-2+k) + \sum^k_{l=1}A\left(\begin{array}{c}
\Theta_1^{\ga_l}\\
\vdots\\
\Theta_N^{\ga_l}
\end{array}
\right)
=
\left(\begin{array}{c}
0\\
\vdots\\
0
\end{array}\right)  
$$
on $C$.  Since the Jacobian $\Pic^0(C)$ of any smooth curve $C$ is a divisible group, and since the matrix $A$ is of rank $N$, there is at least one solution $(Y_1,\dots,Y_N)\in \Pic^0(C)^N$ to the system of equations
\begin{eqnarray*}
Y_1^{a_{1,1}}\otimes\cdots\otimes Y_N^{a_{1,N}} &= X_1\\
\vdots &= \vdots\\
Y_1^{a_{N,1}}\otimes\cdots\otimes Y_N^{a_{N,N}} &= X_N\\
\end{eqnarray*}
This means that the (un-orbifolded) line bundles $L_j:=Y^{-1}_jE_j$ satisfy $L_1^{a_{i,1}}\otimes\cdots\otimes L_N^{a_{i,N}} \cong K_{C,\log}^{u_i}\otimes
\O\left(-\sum^k_{l=1}\sum^N_{j=1}a_{ij}\Theta_j^{\gamma_l}p_l\right)$
for each $i\in\{1,\dots,N\}$.

Now we may construct an orbicurve $\cC$ on $C$ with local group at $p_l$ generated by $\ga_l$ for each $l\in\{1,\dots, k\}$, and we can construct the desired orbifold line bundles $\LL_j$ on $\cC$ from $L_j$ by inverting the map described in Subsection~\ref{rem:desingularize} at each marked point.  It is easy to see that these line bundles form a $W$-structure on $\cC$, and therefore $\W_{g,k}(\bgamma)$ is not empty.
\end{proof}

\begin{ex}

For three-pointed, genus-zero $W$-curves, the choice of oribfold
line bundles $\LL_1,\dots,\LL_N$ providing the $W$-structure is
unique, if it exists at all. Hence, if the selection rule is satisfied, $\W_{0,3}(\bgamma)$ is isomorphic to
$\BG$.
\end{ex}

\subsubsection{Dual Graphs}

We must generalize the concept
of a decorated dual graph, given for $r$-spin curves in \cite{JKV1},
to the case of a general $W$-orbicurve.

\begin{df}
Let $\Gamma$ \glossary{Gamma@$\Gamma$ & A dual graph (possibly decorated)}
be a dual graph of a stable curve $(\cC,p_1,
\dots,p_k )$ as in \cite{JKV1}. A half-edge of a graph $\Gamma$ is
either a tail or one of the two ends of a ``real'' edge of
$\Gamma$.

Let $V(\Gamma)$\glossary{VGamma@$V(\Gamma)$ & The set of vertices in $\Gamma$} be the set of vertices of $\Gamma$, let
$T(\Gamma)$\glossary{TGamma@$T(\Gamma)$ & The set of tails in $\Gamma$} denote the tails of $\Gamma$, and let $E(\Gamma)$\glossary{EGamma@$E(\Gamma)$ & The set of edges in $\Gamma$} be
the set of ``real'' edges.  For each $\nu \in V(\Gamma)$  let
$g_\nu$ be the (geometric) genus of the component of $\cC$
corresponding to $\nu$, let $T(\nu)$ denote the set of all
half-edges of $\Gamma$ at the vertex $\nu$, and let $k_\nu$ be the
number of elements of $T(\nu)$.
\end{df}

\begin{df}
Let $\Gamma$ be a dual graph. The \emph{genus} of $\Gamma$ is
defined as
$$
g(\Gamma)=\dim H^1(\Gamma)+\sum_{\nu\in V(\Gamma)}g_\nu.
$$

A graph $\Gamma$ is called \emph{stable} if $2g_\nu+k_\nu \ge 3$
for every $\nu\in V(\Gamma)$.
\end{df}

\begin{df}
A \emph{$G_W$-decorated stable graph} is a stable graph $\Gamma$
with a decoration of each tail $\tau \in T(\Gamma)$ by a choice of
$\gamma_{\tau}\in \grp$.

It is often useful to decorate all the half-edges---not just
the tails.  In that case, we will require that for any edge $e \in
E(\Gamma)$ consisting of two half edges $\tau_+$ and $\tau_-$, the
corresponding decorations $\gamma_+$ and $\gamma_-$ satisfy
\begin{equation}
\gamma_- =(\gamma_+)^{-1},
\end{equation}
and we call such a graph a \emph{fully $G_W$-decorated stable graph}.
\end{df}

\begin{df}
Given a $W$-curve $\frkc:=(\cC,p_1, \dots,p_k, \LL_1,\dots, \LL_N,
\vpt_1,\dots,\vpt_N)$, the underlying (coarse) curve $C$ defines a dual
graph $\Gamma$. Each half-edge $\tau$ of $\Gamma$ corresponds to
an orbifold point $p_\tau$ of the normalization of $\cC$, and thus
has a corresponding choice of $\gamma_{\tau}\in \grp$, as given in
Proposition~\ref{prop:pushforward-W}.

We define the \emph{fully $G_W$-decorated dual graph of $\frkc$} to be the
graph $\Gamma$ where each half-edge $\tau$ is decorated with the
group element $\gamma_{\tau}$.
\end{df}

\begin{rem}
If a fully $G_W$-decorated graph $\Gamma$ is to correspond to an
actual $W$-orbicurve, the selection rules of
Equation~(\ref{eq:sel-rule}) must be satisfied on every vertex of
$\Gamma$; namely, for each $\nu\in V(\Gamma)$ and for each $j\in
\{1,\dots,N\}$ the degree of $|\LL_j|$ on the component of the
underlying curve associated to $\nu$ must be integral:
\begin{equation}\label{eq:graph-sel-rule}
\deg(|\LL_j|_{\nu}) = \left( q_j(2g_{\nu} - 2 + k_{\nu})
-\left(\sum_{\tau \in T(\nu)} \Theta_j^{\gamma_{\tau}}\right)
\right) \in \Z.
\end{equation}
\end{rem}

\begin{df}
For any $G_W$-decorated stable $W$-graph $\Gamma$, we define
$\W(\Gamma)$\glossary{MGammaW@$\W(\Gamma)$ & The closure of the substack of stable $W$-curves with dual graph $\Gamma$}
 to be the closure in $\W_{g,k}$ of the stack of stable $W$-curves with
$G_W$-decorated dual graph equal to $\Gamma$.

\end{df}

\begin{rem}
Note that no deformation of a nodal orbicurve will deform a node
with one orbifold structure to a node with a different orbifold
structure---the only possibility for change is to smooth the node
away.  This means that if $\Gamma$ is $G_W$-decorated only on the
tails and not on its edges, then the space $\W(\Gamma)$ is a
disjoint union of closed subspaces $\W(\tGamma)$ where the
$\tGamma$ run through all the choices of fully $G_W$-decorated
graphs obtained by decorating all edges of $\Gamma$ with elements of
$G_W$.

When a graph is a tree with only two vertices and one (separating) edge, then the rules of Equation~(\ref{eq:graph-sel-rule}) imply that the  decorations on the tails uniquely determine the decoration in the edge:  each $\Theta_i$ for the edge is completely determined by the integrality condition.

However, if the graph is a loop, with only one vertex and one edge,  then the rules of Equation~(\ref{eq:graph-sel-rule})  provide no restriction on the decoration $\gamma_+$ at the node.
\end{rem}

Let the genus of $\Gamma$ be $g=g(\Gamma)$, let the number of
tails of $\Gamma$  be $k$, and let the ordered $k$-tuple of  the
decorations associated to those tails be $\bgamma:=(\gamma_1,
\dots,\gamma_k)$.  In this case it is clear that
$\W(\Gamma) \subseteq \W_{g,k}(\bgamma)$
is a closed substack.

\subsubsection{Morphisms}

We have already discussed the morphism $\st:\W_{g,k} \rTo \MM_{g,k}$.  In this subsection we define several other important morphisms.

\paragraph{\underline{Forgetting tails}}
  If $\bgamma =
(\gamma_1,\dots,J,\dots,\gamma_k)$ is such that $\gamma_i=J$
for some $i \in \{1,\dots,k\}$ (that is,  $\Theta_l^{\gamma_i} =
q_l$ for every $l\in\{1,\dots,N\}$), and if $\bgamma' =
(\gamma_1,\dots,\hat{\gamma_i},\dots,\gamma_k)\in \grp^{k-1}$ is
the $k-1$-tuple obtained by omitting the $i$th component of
$\bgamma$, then the \emph{forgetting tails morphism} $$\vartheta:
\W_{g,k}(\bgamma) \rTo
\W_{g,k-1}(\bgamma')\glossary{theta@$\vartheta$ & The
forgetting tails morphism for unrigidified $W$-curves}$$ is
obtained by forgetting the orbifold structure at the point $p_i$.

We describe the morphism more explicitly as follows.  Let $\cCbb$
denote the orbicurve obtained by forgetting the marked point $p_i$
and its orbifold structure, but leaving the rest of the marked
points of the orbicurve $\cC$ unchanged.  Let $\varrhob:\cC\rTo \cCbb$
be the obvious morphism. By Proposition~\ref{prop:pushforward-W}, the pushforwards $\varrhob_*(\LL_j)$ for
$j\in\{1,\dots,N\}$ satisfy
\begin{align*}
(\varrhob_*(\LL_1))^{a_{j,1}}\otimes\cdots\otimes(\varrhob_*(\LL_N))^{a_{j,N}} & \irightarrow K_{\cC,\log}^{u_j}
\otimes \O((-\sum^N_{\ell=1}a_{j,\ell}\Theta_j^{J})p_i)\\
& = K_{\cC,\log}^{u_j}
\otimes \O(-(\sum_{j=1}^N  a_{j\ell}q_\ell) p_i)\\
& =
K_{\cC,\log}^{u_j}
\otimes \O( -u_j p_i) =K_{\cCbb,\log},
\end{align*}
 since $\sum_{j=1}^N  a_{j\ell}\Theta_j^J =  \sum_{j=1}^N  a_{j\ell}q_\ell = u_j$ (because $Aq = V^{-1}B q  = V^{-1}(1,\dots,1)^T = u$). We denote the induced isomorphisms
by $${\vpt'}_j:
(\varrhob_*(\LL_1))^{a_{j,1}}\otimes\cdots\otimes(\varrhob_*(\LL_N))^{a_{j,N}} \irightarrow
K_{\cCbb,\log}^{u_j}$$

The tuple $(\cCbb,p_1, \dots, \hat{p_i}, \dots,
p_k,\varrhob_*(\LL_1),\dots,\varrhob_*(\LL_N),
{\vpt'}_1, \dots, {\vpt'}_N
)$ is a $W$-orbicurve of
type $\bgamma'$.  This procedure induces the desired morphism
$$\vartheta: \W_{g,k}(\bgamma) \rTo
\W_{g,k-1}(\bgamma').$$

Note that the essential property of
$\gamma_i$ that allows the forgetting tails morphism to exist is the fact that
$\sum_{j=1}^N a_{lj}\Theta_j^{\gamma_i} = u_l$ for every $l \in
\{1,\dots,s\}$.  Since the weights $q_j$ are uniquely determined
by this property (since $B$ and $A$ are of rank $N$), this means that a
marked point $p_i$ may not be forgotten unless $\gamma_i=J \in
\grp$.

\paragraph{\underline{Gluing and cutting}}

    Gluing two marked points on a stable curve or on a pair of stable curves defines a Riemann surface with a node. This procedure defines two well-known morphisms
\begin{equation}\label{eq:glueTreeCurves}
\rho_{tree}: \MM_{g_1, k_1+1}\times \MM_{g_2,
    k_2+1}\rTo \MM_{g_1+g_2, k_1+k_2},
\end{equation}\glossary{rhotree@$\rho_{tree}$ & The gluing-trees morphism for stable curves}
\begin{equation}\label{eq:glueLoopCurves}
\rho_{loop}: \MM_{g-1, k+2}\rTo \MM_{g, k}.
\end{equation}\glossary{rholoop@$\rho_{loop}$ & The gluing-loops morphism for stable curves}
More generally, if $\Gamma$ is a  dual graph, then we can cut an edge to form
$\hGamma$, \glossary{Gammah @ $\hGamma$ & The dual graph obtained from $\Gamma$ by cutting the edges}
and there is a gluing map
$$\rho: \MM(\hGamma) \rTo \MM(\Gamma) \subseteq \MM,$$ where $\MM(\Gamma)$ \glossary{MGamma@$\MM(\Gamma)$ & The closure of the substack of stable curves with dual graph $\Gamma$}
denotes the closure in $\MM_{g,k}$ of the locus of stable curves with dual graph $\Gamma$.

 Unfortunately, there is no direct lift of $\rho$ to the moduli stack of $W$-curves because there is no canonical way to glue the fibers of the line bundles $\LL_i$ on the two points that map to a node.  In fact, if anything, the morphism goes the other way; that is, restricting a $W$-structure on a nodal (i.e., glued) curve to the normalization (i.e., cutting) of that curve will induce a $W$-structure on the normalization.  Unfortunately, this does not induce a morphism from $\W(\Gamma)$ to $\W(\hGamma)$ because for many curves the normalization of the curve does not have a well-defined choice of a marking (section) for the two points that map to the node.

Nevertheless, we can use this restriction property to create a pair of morphisms that will serve our purposes just as well as a gluing morphism would.  To do this, we first consider the fiber product $$F:=\MM(\hGamma)\times_{\MM(\Gamma)} \W(\Gamma).$$ \glossary{F @$F$ &The fibered product $F:=\MM(\hGamma)\times_{\MM(\Gamma)} \W(\Gamma)$}
$F$ is the stack of triples
$(\tilde{\cC},(\cC,\fL),\beta)$, where $\tilde{\cC}$ is a pointed
stable orbicurve with dual graph $\hGamma$, and ${\cC}$ is a pointed
stable orbicurve with dual graph $\Gamma$; also, $\fL$ is a
$W$-structure on $\cC$, and  $\beta:\rho[\tilde{\cC}]
\rTo \cC$ is an isomorphism of the glued curve $\rho[\tilde{\cC}]$
with the orbicurve $\cC$.

Instead of a lifted gluing (or cutting) map, we will use the following pair of maps:
$$\W(\hGamma)\lTo^{q}
F\rTo^{pr_2}\W(\Gamma),$$
    where the morphism $q$ \glossary{q @$q$ & The morphism $F \rTo^q \W(\hGamma)$ obtained by restricting $W$-structures.}
    simply takes the triple to the $W$-curve
$(\tilde{\cC},\beta^*(\fL))$ by pulling back the $W$-structure to
$\tilde{\cC}$.  This is well-defined because the fiber product has well-defined choices of sections of $\tilde{\cC}$ mapping to the node of $\cC$.

Alternatively, we could also describe a gluing process in terms of an additional structure that we call \emph{rigidification}. Let $p$ be a
    marked point.   Let $j_p: \B G_p \rTo \cC$ be the corresponding gerbe section of $\cC$.
        A \emph{rigidification at $p$} is an isomorphism
    $$\psi: j^*_{p}(\LL_1\oplus \cdots
    \oplus\LL_N){\rTo}  [\C^N/G_{p}]$$
 such that for every $\ell\in\{1,\dots,N\}$ the following diagram commutes:
\begin{equation}\label{eq:rigid}
\begin{diagram}
j_{p}^*\left(\bigoplus_{m=1}^N\LL_m\right)  & \rTo^{\psi} & [\C^N/G_{p}]\\
\dTo^{\vpt_\ell\circ A_\ell}  &              & \dTo^{A_\ell}\\
j_{p}^*(K_{\log}^{u_\ell})  & \rTo^{res^{u_\ell}} & \C
\end{diagram}
\end{equation}
 where the map $\res^{u_\ell}$ takes $({dz}/{z})^{u_\ell}$ to $1$.  Note that the two terms in the bottom of the diagram have trivial orbifold structure. Since each monomial $W_\ell$ of $W$ is fixed by $G_W$, we also have that each monomial $A_\ell$ is fixed by $G_W$ and hence by $G_{p}$.  This means that the vertical maps are both well defined.

   One can define the equivalence class of $W$-structures with rigidification in an
    obvious fashion. The notion of rigidification is
    also important for constructing the perturbed Witten equation, but we will not use it in any essential way in this paper.

A more geometric way to understand the rigidification is follows.
Suppose the fiber of the $W$-structure at the marked point is
$\left[(\LL_1\oplus \LL_2\oplus \cdots\oplus \LL_N)/G_p\right]$. The rigidification can be
thought as a $G_p$-equivariant map $\psi: \bigoplus_i \LL_i\rTo \C^N$
commuting with the $W$-structure.  For any element $g\in G_p$, the rigidification
$g\psi$ is considered to be an equivalent rigidification.  The choice of  $\psi$ is equivalent to a choice of basis $e_i\in
\LL_i$ such that $A_j(e_1, \dots, e_N)=(dz/z)^{u_j}$ and the basis $g(e_1), \dots
g(e_N)$ is considered to be an equivalent choice. In particular, if $\LL_{i_1}, \dots,\LL_{i_m}$ are the line bundles fixed by $G_p$ (we call the corresponding variables $x_{i_j}$ the \emph{broad variables}) then in each equivalance class of rigidifications, the basis elements $e_{i_1},\dots,e_{i_m}$ for the subspace $\left.\bigoplus_{j=1}^m \LL_{i_j}\right|_p$ are unique, but the basis elements for the terms not fixed by $G_p$ (the \emph{narrow variables}) are only unique up to the action of $G_p$.

It is clear that the group $G_W/G_p$ acts transitively on the set of rigidifications within a single orbit. Let $\W^{rig_p}(\Gamma)$ \glossary{MMrigp @ $\W^{rig_p}(\Gamma)$ & The stack of $W$-curves with dual graph $\Gamma$ and rigidification at $p$}
be the closure of the substack  of equivalence classes of $W$-curves with dual graph $\Gamma$ and
a rigidification at $p$. The group $G_W/G_p$ acts on $\W^{rig_p}(\Gamma)$
by interchanging the rigidifications. The stack $\W^{rig_p}(\Gamma)$ is a principal $G_W/G_p$-bundle over $\W(\Gamma)$ and
$$\left[\W^{rig_p}(\Gamma)/(G_W/G_p)\right]=\W(\Gamma).$$

    Now we describe the gluing. To simplify notation, we ignore the orbifold
    structures at other marked points and denote the type of the marked points $p_+, p_-$ being glued by $\gamma_+,
    \gamma_-$. Recall that the resulting orbicurve must be balanced, which means that  $\gamma_-=\gamma^{-1}_+$.
Let
    $$\psi_{\pm}: j^*_{p_{\pm}}(\LL_1\oplus \cdots
    \oplus\LL_N){\rTo}  [\C^N/G_{p_{\pm}}]$$
be the rigidifications. However, the residues at $p_+, p_-$
    are opposite to each other. The obvious  identification will
    not preserve the rigidifications. Here, we fix once and for all an
    isomorphism
    $$I: \C^N\rTo \C^N$$
    such that $W(I(x))=-W(x)$. $I$ can be explicitly constructed
    as follows. Suppose that $q_i=n_i/d$ for common denominator
    $d$. Choose $\xi^d=-1$, and set $ I(x_1, \dots,  x_N)=(\xi^{n_1}x_1, \dots, \xi^{n_N}x_N)$.
If $I'$ is another choice, then $I^{-1}I'\in \genj  \le G_W$.
    Furthermore, $I^2\in \genj  \le G_W$ as well. The identification by $I$
    induces a $W$-structure on the nodal orbifold Riemann surface with a
    rigidification at the nodal point. Forgetting the rigidification at the node yields the lifted gluing morphisms
\begin{equation}\label{eq:glueTreeW}
\widetilde{\rho}_{tree,\gamma}:\W^{rig}_{g_1, k_1+1}(\gamma)\times
\W^{rig}_{g_2,
    k_2+1}(\gamma^{-1})\rTo \W_{g_1+g_2, k_1+k_2},
\end{equation}\glossary{rhotildetree@$\widetilde{\rho}_{tree}$ & The gluing-trees morphism for rigidified $W$-curves}
\begin{equation}\label{eq:glueLoopW}
\widetilde{\rho}_{loop,\gamma}: \W^{rig}_{g, k+2}(\gamma,
\gamma^{-1})\rTo \W_{g+1, k},
\end{equation}\glossary{rhotildeloop@$\widetilde{\rho}_{loop}$ & The gluing-loops morphism for rigidified $W$-curves}
where $\tilde{\rho}$ is defined by gluing the rigifications at the
extra tails and forgetting the rigidification at the node.

\paragraph{\underline{Degree of $\st$}}

There are various subtle factors in our theory arising from the orbifold degrees of the maps.   These factors can be a
major source of confusion.  The degree of the stabilization
morphism $\st_{\bga}: \W_{g,k}(\bga)\rTo \MM_{g,k}$ is especially important in this paper.

As described in Remark~\ref{rem:principal-gamma},
for a given choice of $\bga\in G_W^k$, the set of all $W$-structures of type $\bga$ on a given orbicurve $\cC$ with underlying curve $C$ is either empty or is an $H^1(C,G_W)$-torsor; therefore, $H^1(C, G_W)$ acts on the non-empty $\W_{g,k}(\bga)$ and the coarse quotient is
$\MM_{g,k}$. One might think that
$\deg(\st_{\bga})=|H^1(C, G_W)|$, but further examination shows that this is
not case because $\MM_{g,k}$ is not isomorphic to
$\left[\W_{g,k}/H^1(C,G_W)\right]$ as a stack. This is particularly evident
because $\W_{g,k}$ has a nontrivial isotropy group at each
point, while the generic point of $\M_{g,k}$ has no isotropy group. The key point is that the automorphism group of any $W$-structure over a fixed, smooth orbicurve $\cC$ is all of $G_W$. Therefore, we have
\begin{equation}
\label{eq:deg-st-gen}
\deg(\st_{\bga}) =|G_W|^{2g-1}.
\end{equation}
Since there are $|G_W|^{k-1}$ choices of $\bga$ that produce a non-empty $W_{g,n} (\bgamma)$, this shows that the total degree of  $\st:\W_{g,k} \rTo \MM_{g,k}$ is $|G_W|^{2g-2+k}$.

For any decorated graph $\Gamma$, we also have a stabilization map $\st_{\Gamma}:\W(\Gamma) \rTo \MM(\Gamma)$\glossary{stGamma @ $\st_\Gamma$ & The stabilization map restricted to $\W(\Gamma)$}, but the degree of $\st_\Gamma$ is not the same as that of $\st$.  For example, if $\Gamma$ is a graph with two vertices and one (separating) edge labeled with the element $\gamma_+$, then the number of $W$-structures over a generic point of $\MM(\Gamma)$ is still $|H^1(C,G_W)|=|G_W|^{2g}$, but, by Equation~(\ref{eq:TreeAut}), the automorphism group of a generic point of $\W(\Gamma)$  is
$G_W\times_{G_W/\langle\gamma_+\rangle}G_W$.

For a tree, the selection rules  uniquely determine the choice of $\gamma_+$; therefore, we have the following.
\begin{prop}\label{prp:ramif-tree}
For a tree $\Gamma$ with two vertices and one edge, with tails decorated with $\bga:=(\ga_1,\dots,\ga_k) \in G_W^k$ and edge decorated with $\ga_+$ 
the map $\st_{\bga}$ is ramified along $\W(\Gamma)$, and
\begin{equation}
\label{eq:deg-st-tree}
\deg(\st_{\bga}) = |\lgr{+}|\deg(\st_{\Gamma}).
\end{equation}
\end{prop}

If $\Gamma$ is a loop with one vertex and one (non-separating) edge, such that the edge is labeled with the element $\gamma$, then we have the following proposition.

\begin{prop}\label{prp:ramif-loop} Let $\bga := (\ga_1,\dots,\ga_k)\in G_W^k$ be chosen so that $\W_{g,k}(\bga)$ is non-empty.  For the loop $\Gamma$ with a single vertex and a single edge decorated with $\gamma_+$ and tails decorated with $\bga$, the stack $\W(\Gamma)$ is non-empty.  Moreover, the morphism $\st_{\bga}$ is ramified along $\st_{\Gamma}$ and
\begin{equation}\label{eq:deg-st-loop}\deg(\st_{\Gamma}) =
\frac{|G_W|^{2g-2}}{|\lgr{+}|}.
\end{equation}
\end{prop}
\begin{proof}
First, we claim that the number of $W$-structures over a generic point of $\MM_{g,k}$ which degenerate to a given $W$-structure over
$\Gamma$ is $|\lgr{}|$.  To see this, note that for any $W$-structure $\fL$ on a smooth orbicurve $\cC$ with underlying curve $C$, all other $W$-structures on  $\cC$ differ from it by an element of $H^1(C,G_W)$.  Consider any fixed $1$-parameter family of $W$-curves such that the $W$-curve $(\cC,\fL)$ degenerates to a $W$-curve $(\cC',\fL')$ with dual graph $\Gamma$, corresponding to the contraction of a cycle $\alpha\in H_1(C,\Z)$.  In this case, we may choose a basis of $H^1(C,G_W)$ such that the first basis element is dual to $\alpha$ and the second basis element is dual to a cycle $\beta$ such that $\alpha\cdot \beta = 1$, and $\beta\cdot \sigma = 0$ for any remaining basis element $\sigma$ of $H_1(C,G_W)$.

In this case, the $W$-structure obtained by multiplying $\fL$ by an element of the form $(1, \varepsilon_2,\dots,\varepsilon_{2g})\in H^1(C,G_W)$  will again degenerate (over the same family of stable underlying curves) to $\fL'$ if and only if $\varepsilon_2 \in \lgr{}$.

Second, by Equation~(\ref{eq:loopAut}), the automorphism groups for both smooth $W$-curves and these degenerate $W$-curves are isomorphic to $G_W$.  This, combined with the previous degeneration count, proves that the ramification is $|\lgr{}|$.

More generally, the pair $\left(\cC,\fL\cdot(1, \varepsilon_2,\dots,\varepsilon_{2g})\right)$ will always degenerate to a $W$-curve with dual graph $\Gamma$, and $\left(\cC,\fL\cdot(\varepsilon_1, \varepsilon_2,\dots,\varepsilon_{2g})\right)$ for $(\varepsilon_1, \varepsilon_2,\dots,\varepsilon_{2g})\in H^1(C,G_W)$ will degenerate to a $W$-curve with dual graph labeled by $\gamma\varepsilon_1$ instead of by $\gamma$.  Thus the moduli $\W(\Gamma)$ is non-empty for every choice of decoration $\gamma\in G_W$ of the edge of $\Gamma$.
\end{proof}

\subsection{Admissible groups $G$ and $\W_{g,k,G}$}\label{sec:changeGroup}

 The constructions of this paper depend quite heavily on the group of diagonal symmetries $\grp$ of the singularity $W$.  It is useful to generalize these constructions to the case of a subgroup $G$ of $\grp$.   First, the isomorphism $I$ is only well-defined up to an element
 of $\langle J\rangle$.   Therefore, we will always require that
$J \in G.$
The problem is that it is not \emph{a priori} obvious that the stack of $W$-curves with markings only coming from a subgroup $G$ is a proper stack. Namely, the
orbifold structure at nodes may not be in $G$.

\label{prop:HTH-Laurent}
However, we note\footnote{We are grateful to H.~Tracy Hall for suggesting this approach to us.} that for any Laurent polynomial $Z=\sum_j\prod_{i=1}^N x_i^{a_{ij}}$ of weighted total degree $1$, with $a_{ij} \in \Z$ for all $i$ and $j$, the diagonal symmetry group $G_{\tW}$ of $\tW:=W+Z$ is clearly a subgroup of $\grp$ containing $\genj $, and the stack ${\W}_{g,k}(\tW)$ of $\tW$-curves is a proper substack of $\W_{g,k}(W).$

\begin{prop}\label{prop:smaller-moduli}
For every quasi-homogeneous Laurent polynomial $\tW=W+Z$, where $Z$ has no monomials in common with $W$, there is a natural morphism $\inc:{\W}_{g,k}({W +Z}) \rTo \W_{g,k}(W)$
from the stack ${\W}_{g,k}(\tW)$ to an open and closed substack of $\W_{g,k}(W)$.  Moreover, this morphism is finite of degree equal to the index of $G_{\tW}$ in $G_W$.
\end{prop}

\begin{proof}
It suffices to consider the case of $\tW = W+M$, where  $M = \prod_{i=1}^Nx_i^{\beta_i}$ is a single monomial of degree $1$ (i.e., $\sum_{i=1}^N \beta_i q_i = 1$), distinct from the monomials $W_j = c_j \prod_{l=1}^s x_i^{b_{l,j}}$ of $W$.

The morphism $\inc$ is simply the functor which forgets the additional conditions arising from the monomial in $M$.  

Given a $W$-structure $(\LL_1,\dots,\LL_N, \phi_1,\dots,\phi_s)$ on an orbicurve $\cC$, we can produce $s$ choices of a $d$-th root of $\O$---one for each monomial of the original polynomial $W$---as follows.  For each $j \in \{1,\dots,s\}$ let
\begin{equation}
\N_j:=\bigotimes_{i=1}^N\LL_i^{\otimes (b_{i,j}-\beta_i)}.
\end{equation}
Using the fact that $B(q_1,\dots,q_N)^T = (1,\dots,1)^T$ and $\beta \cdot (q_1,\dots,q_N) = 1$ we see that $\N_j^{\otimes d} \cong \O$, where $d$ is defined (as in Definition~\ref{df:qhomPoly}) to be the smallest positive integer such that $(dq_1,\dots,dq_N) \in \Z$.
This gives $s$ morphisms
\begin{equation}
\W_{g,k}(W) \rTo^{\Phi_{j}} \J_{g,k,d},
\end{equation}
where $\J_{g,k,d}:=\{(\cC,p_1,\dots, p_k,\LL,\psi:\LL^d \rTo \O_{\cC})  )$ denotes the stack of $k$-pointed, genus-$g$ orbicurves with a $d$-th root $\N$ of the trivial bundle.   It is easy to see that the stack $\J_{g,k,d}$ has a connected component $\J_{g,k,d}^{0}$ corresponding to the trivial $d$-th root of $\O$.  The inverse image $\W_{g,k}^0(W) := \Phi_j^{-1}(\J_{g,k,d}^{0})$ of the trivial component for each $j\in\{1,\dots,s\}$ is independent of $j$, is open and closed, and is the image of the forgetful morphism $\inc$:
$$\W_{g,k}(\tW) \rOnto^{\inc} \W_{g,k}^0(W) \subseteq \W_{g,k}(W).  $$

The objects of the stack $\W_{g,k}(\tW)$ are $\tW$-curves $(\cC, \LL_1,\dots,\LL_N, \phi_1,\dots,\phi_{s+1})$, where $\phi_{s+1}$ is an isomorphism $\phi_{s+1} :M(\LL_1,\dots,\LL_N) \rTo K_{\log}$, whereas the objects of the stack $\W_{g,k}^0(W)$ are $W$-curves $(\cC, \LL_1,\dots,\LL_N, \phi_1,\dots,\phi_{s})$ such that there exists some isomorphism $\psi:M(\LL_1,\dots,\LL_N) \rTo K_{\log}$ which is compatible with the isomorphisms $\phi_i$ of the $W$-structure.  Any $W$-curve with such a $\psi$ is isomorphic to the image of some $\tW$-curve, but since an automorphism of a $W$-curve in $\W_{g,k}^0(W)$ need not fix the isomorphism $\psi$, the automorphism group of a generic $W$-curve in $\W_{g,k}^0(W)$ is $G_W$, while the automorphism group of a generic $\tW$-curve is $G_{\tW}$.
\end{proof}

\begin{df}
We say that a subgroup $G \le G_W$ is \emph{admissible} or is an \emph{admissible group of Abelian symmetries of $W$} if there exists a Laurent polynomial $Z$, quasi-homogeneous with the same weights $q_i$ as $W$, but with no monomials in common with $W$, such that $G = G_{W+Z}$.
\end{df}

\begin{df}
Suppose that $G$ is admissible. We define  the stack $\W_{g,k,G}(W):=\W_{g,k}(\tW)$ for any $\tW=W+Z$ with $G_{\tW}=G$.
\end{df}

The most important consequence of
Proposition~\ref{prop:smaller-moduli} is that we may restrict (pull back) the
virtual cycle $\left[\W_{g,k}(W)\right]^{vir}$ to the substack
$\W_{g,k,G}(W)$ (see Subsection~\ref{subsec:axioms}).

\begin{rem}
An admissible group $G$ may have more than one $Z$ such that
$G=G_{W+Z}$.  One can show (see \cite{ChR}) that $\W_{g,k,G}:=\W_{g,k}(W+Z)$  is independent of $Z$ and depends only on $G$.
\end{rem}

It is immediate that every admissible group contains $J$. Marc Krawitz \cite[Prop 3.4]{Krawitz}  has proved the converse.  For the reader's convenience we repeat his proof here.
\begin{prop}[Krawitz]\label{prop:admissible}
For any non-degenerate $W \in \C[x_1,\dots,x_N]$, any group
of diagonal symmetries of $W$ containing $J$ is admissible.
\end{prop}
\begin{proof}
The subring of $G$-invariants in $A:=\C[x_1,\dots,x_N]$ is finitely generated by monomials. Let $Z$ be the sum of all $G$-invariant
monomials in $A$ not divisible by monomials in $W$.  

We claim that $G$ is the maximal diagonal symmetry group of $W+Z$.  If it were not, there would be a diagonal symmetry group $H$, with $G \le H$ and $A^G   \subseteq A^H$.  The actions of $G$ and $H$ on $A$ extend to actions on the fraction field $E:=\C(x_1,\dots,x_N)$.  Since the action is diagonal, it is easy to show that this implies that the fraction field of $A^G$ equals $E^G$ and the fraction field of  $A^H$ equals $E^H$.  Since $A^G = A^H$, we have $E^G = E^H$.  Since $G$ and $H$ are finite, we have, by \cite[Cor 3.5]{milne}, that $$G = \aut(E / E^G)  = \aut(E / E^H) = H.$$   Therefore $G$ is the maximal symmetry group of $W+Z$.

Now, since $J$ preserves each of the constituent monomials of $Z$, each of these monomials has integral quasi-homogeneous degree. We may correct each of these monomials by a (negative)
power of any monomial in W to ensure that each of the monomials has quasi-homogeneous degree equal to $1$, and since we are correcting by $G$-invariants not dividing the monomials of $Z$, we do not change the maximal symmetry group of $W + Z$.
\end{proof}

\subsection{The tautological ring of $\W_{g,k}$}

    A major topic in Gromov-Witten theory is the tautological ring of $\MM_{g,k}$.
    The stack $\W_{g,k}$ is similar to $\MM_{g,k}$ in many ways, and we can
    readily generalize the notion of the tautological ring to $\W_{g,k}$. We expect
    that the study of the tautological ring of $\W_{g,k}$ will be  important to the calculation of our invariants. It is not unreasonable to conjecture that the virtual cycle constructed
    in the next section is, in fact, tautological.

Throughout this section, we will refer to the following diagram.
\begin{equation}
\begin{diagram}
\cC_{g,k} &\rTo^{\varrho} &C_{g,k}  \\
\uTo^{\sigma_i}\dTo_\pi  & \ldTo_{\varpi}\\
 \W_{g,k}\\
\end{diagram}
\end{equation}
Here, $\cC_{g,k} \rTo^{\pi} \W_{g,k}$ is the  universal orbicurve, and $\varpi:C_{g,k} \rTo \W_{g,k}$ \glossary{pivar @ $\varpi$ & The universal stable curve} is the universal underlying stable curve. The map $\sigma_i$ \glossary{sigmai @ $\sigma_i$ & The $i$th section of the universal orbicurve} is the $i$th section of $\pi$, and we denote by $\bar{\sigma}_i$ \glossary{sigmaibar @ $\bar{\sigma}_i$ & The $i$th section of the universal stable curve $\varpi$} the $i$th section of $\varpi$.  The map $\varrho$ forgets the local orbifold structure and takes a point to its counterpart in $C_{g,k}$.  On $\cC_{g,k}$ we also have the universal $W$-structure $\bigoplus\LL_i$ and the line bundles $K_{\cC,\log}$ and $K_{\cC}$.

    \subsubsection{\bf $\psi$-classes: } As in the case of the moduli of stable maps, we denote by $\psit_i$ the first Chern class of the $\cC$-cotangent line bundle on $\W_{g,k}$. That is,
\begin{equation}\psit_i := c_1 (\sigma_i^* (K_{\cC})).\end{equation}
\glossary{psi @$\psi_i$ & The descendant class}    We note that since $C_{g,k}$ is the pullback of the universal stable curve from $\MM_{g,k}$, replacing the $\cC$-cotangent bundle by the $C$-cotangent bundle would give the pullback of the usual $\psi$-class, which we also denote by $\psi$:
\begin{equation}
\psi_i:=c_1(\bar{\sigma}_i^* (K_{C})) = \st^*(\psi_i).
\end{equation}

These classes are related as follows.
\begin{prop}\label{prp:tautpsi}
If the orbifold structure along the marking $\sigma_i$ is of type $\gamma_i$, with $|\lgr{i}| = m_i$, then we have the relation
\begin{equation}\label{eq:tautpsi}
m_i \psit_i=\st^*{\psi}_i.
\end{equation}
\end{prop}
\begin{proof}
Let $D_i$ denote the image of the section $\sigma_i$ in $\cC_{g,k}$.  Note that since $\bar{{\sigma}}_i = \varrho\circ\sigma_i$, then by
Equation~(\ref{eq:CoarseFineForms}), we have
\begin{equation}
\bar{\sigma}^*_i K_{C_{g,k}} = \sigma_i^*\left(\varrho^*K_{C_{g,k}}\right)  =
\sigma_i^*\left(K_{\cC_{g,k}}\otimes \O(-(m_i-1)D_i)\right),
\end{equation}
and the residue map shows that
\begin{equation}\label{eq:sigmaklog}
\sigma^*K_{\log} = \O,
\end{equation}
 hence
\begin{equation}
\sigma_i^*(\O(-D_i)) =\sigma_i^*(K_\cC) ,
\end{equation}
which gives the relation~(\ref{eq:tautpsi}).
\end{proof}

\subsubsection{\bf $\psi_{ij}$-classes: } It seems natural to use the $W$-structure to try to define the following tautological classes:
    $$\psi_{ij}:=c_1(\sigma_i^*(\LL_j)).$$  However, these are all zero. To see this, note that for every monomial $W_\ell = \prod x_j^{b_{\ell,j}}$ and for every $i\in\{1,\dots,k\}$, we have, by the definition of the $W$-structure and by Equation~(\ref{eq:sigmaklog}),
    \begin{equation}\label{eq:tautpsiij}
    \sum_{j=1}^N b_{\ell,j} \psi_{ij} = 0.
    \end{equation}
Coupled with the nondegeneracy condition (Definition~\ref{df:nondegenerate}) on $W$, this implies that every $\psi_{ij}$ is torsion in $H^*(\W_{g,k},\Z)$ and thus vanishes in $H^*(\W_{g,k},\Q)$.

\subsubsection{\bf $\kappa$-classes: }

The traditional definition of the $\kappa$-classes on $\MM_{g,k}$ is
$$\kappa_a:=\varpi_*(c_1(K_{C,\log})^{a+1}).$$
We will define the analogue of these classes for $W$-curves as follows:
$$\kappat_a:=\pi_*(c_1(K_{\cC,\log})^{a+1}).$$
\glossary{kappa @$\kappat_a$ & The $a$th orbifold kappa class}
Note that since $K_{\cC,\log} = \varrho^*K_{C,\log}$, and since $\deg(\varrho)=1$, we have
\begin{equation}
\kappat_a=\pi_*(c_1(K_{\cC,\log})^{a+1})
 = \varpi_* \varrho_*\varrho^* (c_1(K_{\cC,\log})^{a+1}) =  \kappa_a
\end{equation}

\subsubsection{\bf $\mu$-classes:}
 The Hodge classes $\lambda_i$ for the usual stack of stable curves are defined to be the Chern classes of the K-theoretic pushforward $R\varpi_* K_C$.  We could also work on the the universal orbicurve $\cC_{g,k} \rTo^{\pi} \MM_{g,k}$, but $\varrho$ is finite, so by Equation~(\ref{eq:KbarIsK}) we have
 $$R\pi_*K_{\cC} = R\varpi_*(\varrho_*K_{\cC}) = R\varpi_* K_C.$$
Therefore,  the two definitions of lambda classes agree.  Moreover, it is known that the $\lambda$-classes can be expressed in terms of $\kappa$-classes, so they need not be included in the definition of the tautological ring.

A more interesting Hodge-like variant comes from pushing down the $W$-structure bundles $\LL_j$.  We also find it more convenient to work with the components of the Chern \emph{character} rather than the Chern classes.  We define $\mu$-classes to be the components of the Chern {character} of the $W$-structure line bundles:
$$\mu_{ij}:=Ch_i(R\pi_* \LL_j).$$
\glossary{muij @$\mu_{ij}$ & The mu classes $Ch_i(R\pi_* \LL_j)$}
By the orbifold Grothendieck-Riemann-Roch theorem, these can be expressed in terms of the kappa, psi, and boundary classes (See, for example, the proof of Theorem~\ref{thm:ConcaveCodimOne}).

\subsubsection{Tautological ring of $\W_{g,k}$:}

    \begin{df}
We define the \emph{tautological ring of $\W_{g,k}$} to be the
subring of $H^*(\W_{g,k}, \Q)$ generated by $\psit_i$, $\kappat_a$,
and the obvious boundary classes.
         \end{df}

We would like to propose the following conjecture.

\begin{conj}[Tautological virtual cycle conjecture] The virtual
     cycle (constructed in the next section) is tautological, in
     the sense that its Poincar\'e dual lies in the tensor product of
   the  tautological ring of $\W_{g,k}$ and relative cohomology.
\end{conj}

\section{The State Space Associated to a Singularity}
\label{sec:QcohGroup}
Ordinary Gromov-Witten invariants take their inputs from the cohomology of a symplectic manifold---the state space. In this section, we describe the analogue of that state space for singularity theory. As mentioned above, however, our theory depends heavily on the choice of symmetry group $G$ and not just on the singularity $W$. In this sense, it should be thought of as an orbifold singularity or orbifold Landau-Ginzburg theory of $W/G$.

We have mirror symmetry in mind when we develop our theory.  Some of the choices,
such as degree shifting number, are partially motivated by a
physics paper by Intriligator-Vafa \cite{IV} and a mathematical
paper by Kaufmann \cite{Ka1} where they studied orbifolded B-model Chiral
rings. The third author's previous work on Chen-Ruan orbifold
cohomology also plays an important role in our understanding.

\subsection{Lefschetz thimble}\label{subsec:pairing}

Suppose that a quasi-homogenous polynomial $W: \C^N\rTo \C$
defines a nondegenerate singularity at zero and that for each $i\in\{1,\dots,N\}$ the weight of the variable $x_i$ is $q_i$.  An important classical invariant of the singularity is the \emph{local algebra}, also known as the \emph{Chiral ring} or the \emph{Milnor ring}:
\begin{equation}
\milnor_W:=\C[x_1, \dots, x_N]/\Jac(W), \glossary{Q@$\milnor_W$ &
The local algebra $\C[x_1, \dots, x_N]/\Jac(W)$, also called the
Milnor ring or chiral algebra, of $W$}
\end{equation}
 where $\Jac(W)$\glossary{jacw@$\Jac(W)$ & The Jacobian ideal $(\frac{\partial W}{\partial
x_1}, \dots, \frac{\partial W}{\partial x_N})$ of $W$}  is the
Jacobian ideal, generated by partial derivatives
$$\Jac(W):=\left(\frac{\partial W}{\partial
x_1}, \dots, \frac{\partial W}{\partial x_N}\right).$$
 Let's review some of the basic facts about the local algebra. It
is clear that the local algebra is generated by monomials. The
degree of a monomial allows us to make the local algebra into a
graded algebra. There is a unique highest-degree element
$\det\left(\frac{\partial^2 W}{\partial x_i\partial x_j}\right)$
with degree \begin{equation}\hat{c}_W=\sum_i (1-2q_i).
\glossary{chat@$\chat$ & The central charge of the singularity
$W$}
\end{equation} The
degree $\hat{c}_W$ is called the \emph{central charge} and is a
fundamental invariant of $W$.

The singularities with $\hat{c}_W<1$ are called \emph{simple
singularities} and have been completely classified into the famous
ADE-sequence. Quasi-homogeneous singularities with
$\hat{c}_W=N-2$ correspond to Calabi-Yau hypersurfaces in weighted
projective space. Here, the singularity/LG-theory makes contact
with Calabi-Yau geometry. There are many examples with fractional
value $\chat_W> 1$. These can be viewed as ``fractional
dimension Calabi-Yau manifolds."

The dimension of the local algebra is given by the formula
$$\mu=\prod_i\left(\frac{1}{q_i}-1\right).\glossary{mu@$\mu$ & The dimension of the local algebra $\milnor$}$$
From the modern point of view, the local algebra is considered to be
part of the B-model theory of singularities. The A-model theory considers the
relative cohomology $H^N(\C^N, W^{\infty}, \C)$ where $W^{\infty}=(\Re W)^{-1}(M, \infty)$
for $M>\!\!>0$. Similarly, let $W^{-\infty}=(\Re W)^{-1}(-\infty, -M)$ for $M>\!\!>0$.
The above space is the dual space of the relative homology $H_N(\C^N, W^{\infty}, \Z)$. The latter
is often referred as the space of \emph{Lefschetz thimbles}.

There is a natural pairing
$$\langle\ ,  \rangle: H_N(\C^N, W^{-\infty}, \Z)\otimes H_N(\C^N, W^{\infty}, \Z)\rightarrow \Z$$
defined by intersecting the relative homology cycles. This pairing is a perfect pairing for the following reason.
Consider a family of perturbations $W_{\lambda}$ such that $W_{\lambda}$ is a holomorphic Morse function
for $\lambda\neq 0$. We can construct a basis of $H_N(\C^N, W^{\pm \infty}, \Z)$ by choosing a system of
\emph{virtually horizonal paths}.  A system of virtually horizonal paths
$u^{\pm}_i: [0,\pm \infty)\rTo \C$ emitting   from critical values $z_i$ has the properties
\begin{description}
 \item[(i)] $u^{\pm}_i$
has no self-intersection,  \item[(ii)] $u^{\pm}_i$ is horizonal for large $t$ and extends to $\pm \infty$.
 \item[(iii)] the paths
$u^{\pm}_1, \dots, u^{\pm}_{\mu}$ are ordered by their imaginary values for large $t$.
\end{description}
For each $u^{\pm}_i$, we can associate a Lefschetz thimble $\Delta^{\pm}_i\in H_N(\C^N, W^{\pm \infty}, \Z)$ as follows.
The neighborhood of the critical point of $z_i$
contains a local vanishing cycle. Using the homotopy lifting property, we can transport the local vanishing cycle
along $u^{\pm}_i$ to $\pm \infty$. Define $\Delta^{\pm}_i$ as the union of the vanishing cycles along the corresponding path $u^{\pm}_i$.
The cycles $\Delta^{\pm}_i$ define a basis of $H^N(\C^N, W^{\pm\infty}, \Z)$, and it is clear that
$$\Delta^+_i\cap \Delta^-_j=\delta_{ij}.$$
Hence, the pairing is perfect for $\lambda\neq 0$.

On the other hand, the complex relative homology
$H_N(\C^N, W^{\pm\infty}_{\lambda}, \C)$ defines a vector bundle over the space of $\lambda$'s. The integral homology classes
define a so-called Gauss-Manin connection. The intersection pairing is clearly preserved by the Gauss-Manin connection; hence,
it is also perfect at $\lambda=0$.

We wish to define a pairing on $H^N(\C^N,  W^{\infty}_{\lambda}, \C)$ alone.
As we have done in the last section, write $q_i = n_i/d$ for a common denominator $d$, and choose $\xi$ such that $\xi^d=-1$. Multiplication by the diagonal matrix $(\xi^{n_1}, \dots,
    \xi^{n_N})$ defines a map
    $I : \C^N\rTo \C^N$
    sending $W^{\infty}\rTo W^{-\infty}$. Hence, it induces an isomorphism
    $$I_*: H_N(\C^N, W^{\infty}, \C)\rTo H_N(\C^N, W^{-\infty},
    \C).$$
    \begin{df}
    We define a pairing on $H_N(\C^N, W^{\infty}, \Z)$ by
    $$\langle\Delta_i, \Delta_j\rangle=\langle \Delta_i, I_*(\Delta_j)\rangle .$$
    It induces a pairing (still denoted by $\langle \ , \ \rangle$) on the dual space  $H^N(\C^N, W^{\infty}, \C)$ which
    is equivalent to the residue pairing on the Milnor ring (see Subsection~\ref{subsec:WallsThm}).
As noted earlier, changing the choice of $\xi$ will change the isomorphism $I$ by an element of the group $\genj $, and $I^2 \in \genj $.
Therefore, the pairing is independent of the choice of $I$ on the
invariant subspace $H_N(\C^N,W^{\infty}, \Z)^{\genj }$ or on
$H_N(\C^N,W^{\infty}, \Z)^G$ for any admissible group $G$.
    \end{df}

\subsection{Orbifolding and state space}

    Now we shall ``orbifold'' the previous construction. Suppose that $G$ is an admissible subgroup.
   For each $\gamma \in G$, $W_{\gamma}$ is again nondegenerate.  

\begin{df}
We define the \emph{$\gamma$-twisted sector} $\ch_{\gamma}$ of the
state space to be the $G$-invariant part of the middle-dimensional
relative cohomology for $W_{\gamma}$;  that is,
\begin{equation}\label{eq:defChGamma}
\ch_\gamma:=
H^{N_\gamma}(\C_\gamma^{N_{\gamma}},W_\gamma^{\infty},\C)^{G}\glossary{Hgamma@$\ch_{\gamma}$
& The $G$-invariant middle-dimensional relative cohomology for
$W_{\gamma}$}.
\end{equation}
The central
charge of the singularity $W_\gamma$ is denoted $\chat_\gamma$:
\begin{equation}\label{eq:defchat}
\chat_\gamma:=\sum_{i:\Theta^\gamma_i = 0}
(1-2q_i)\glossary{chatgamma@$\chat_{\gamma}$ & The central charge
of
  the singularity $W_\gamma$}.
\end{equation}
\end{df}

As in Chen-Ruan orbifold cohomology theory, it is important to shift the
degree.
\begin{df} \label{def:iota}
Suppose that $\gamma=(e^{2\pi
    i \Theta^\gamma_1}, \dots, e^{2\pi i\Theta^\gamma_{N}})$ for rational numbers
    $0\leq \Theta^\gamma_i<1$.

    We define the \emph{degree shifting number}
\begin{align}\glossary{iotagamma@$\iota_{\gamma}$ & The degree-shifting number for $\ch_\gamma$}
 \iota_{\gamma}&=\sum_i (\Theta^{\gamma}_i
 -q_i) \label{eq:degshift-thetaq}\\
 &=\frac{\hat{c}_W-N_\gamma}{2}+\sum_{i:\Theta^{\gamma}_i\neq 0}
 (\Theta^{\gamma}_i-1/2)\label{eq:degshift-chat}\\
 &=\frac{\hat{c}_{\gamma}-N_\gamma}{2} +
 \sum_{i:\Theta^{\gamma}_i\neq0}(\Theta^{\gamma}_i-q_i).\label{eq:degshift-chatgamma}
\end{align}
For a class $\alpha \in \ch_{\gamma}$, we define
\begin{equation}\label{eq:Wdeg}
\deg_W(\alpha)=\deg(\alpha)+ 2\iota_{\gamma}.
\glossary{degw@$\deg_W(\alpha)$ & The orbifold degree
$\deg(\alpha) + 2\iota_{\gamma}$ of the (homogeneous) class
$\alpha \in \ch_{\gamma}$}
\end{equation}
      \end{df}
\begin{prop}\label{prp:iotaRel}
For any $\gamma \in G_W$ we have
\begin{equation}\label{eq:iotaRel}
\iota_{\gamma}+\iota_{\gamma^{-1}} = \hat{c}_W -N_\gamma,
\end{equation}
and for any $\alpha \in \ch_\gamma$ and $\beta\in
\ch_{\gamma^{-1}}$ we have
\begin{equation}\label{eq:degRel}
\deg_W(\alpha)+\deg_W(\beta) = 2\hat{c}_W.
\end{equation}
\end{prop}
\begin{proof}
The first relation (Equation~(\ref{eq:iotaRel})) follows
immediately from Equation~(\ref{eq:degshift-chat}) and from the
fact that if $\Theta_i^{\gamma} \neq 0$ then
$\Theta^{\gamma^{-1}}_i = 1-\Theta^{\gamma}_i$, and otherwise
$\Theta_i^{\gamma^{-1}} = \Theta_i^{\gamma}=0$.

The second relation (Equation~(\ref{eq:degRel})) follows from the
first relation and from the fact that every class in $\ch_\gamma$
has degree $N_\gamma$.
\end{proof}
    \begin{rem}
       $H^N(\C^N, W^{\infty}, \C)$ also carries an internal Hodge
       grading due to its mixed Hodge structure. This defines a
       bi-grading for $\ch_{\gamma}$.
       \end{rem}

    \begin{df}
    The    \emph{state space} (or \emph{quantum cohomology group}) of the singularity
    $W/G$ is defined as
    \begin{equation}
    \ch_{W}=\bigoplus_{\gamma\in G}\ch_{\gamma}.\end{equation}
        \end{df}

\begin{df} The $J$-sector $\ch_{J}$ is always one-dimensional,   and the constant function $1$ defines a generator
 $\bone_{1}:=1 \in \ch_{J}$\glossary{1@$\unit$ & The constant function $1$ in $\ch_{J}$ and the unit of the ring $(\ch_{W,G},\star)$}
      of degree $0$. This
     element is the unit in the ring $\ch_W$, and because of this, we often denote it by $\unit$ instead of $\bone_1$.
\end{df}

\begin{df}
For any $\gamma\in G$, we say that the $\gamma$-sector is \emph{narrow} if the fixed point locus is trivial (i.e., $N_\gamma = 0$).  If the fixed point locus is non-trivial, we say that the $\gamma$-sector is \emph{broad}.
\end{df}

    Since $\gamma$ and $\gamma^{-1}$ have the same fixed point set,
   there is an obvious isomorphism
   $$\varepsilon: \ch_{\gamma}\rTo \ch_{\gamma^{-1}}.$$
    We define a pairing on $\ch_{W}$ as the direct sum of the pairings
   $$\langle \, , \rangle _{\gamma}:\ch_{\gamma}\otimes
    \ch_{\gamma^{-1}}\rTo \C$$ by $\langle f,g\rangle
    _{\gamma}=\langle f, \varepsilon^* g\rangle ,$ where the second
    pairing is the pairing of the space of relative cohomology.  The
    above pairing is obviously symmetric and non-degenerate.

Now the pairing on $\ch_{W}$ is defined as the direct sum of the
pairings $\langle \, , \rangle_{\gamma}$.

    \begin{lm} The above
    pairing preserves $\deg_W$. Namely, if $\ch_{\ga}^a$ denotes the elements $\varkappa \in \ch_\ga$ with $\deg_W(\varkappa) = a$, then $\langle \, , \rangle$ gives a pairing of $\ch^a_W$ with $\ch^{2 \chat-a}_W:$
    $$\ch^{a}_W\otimes \ch^{2\hat{c_W}-a}_W\rTo \C.$$
    \end{lm}
    \begin{proof}
        This is a direct consequence of Proposition~\ref{prp:iotaRel}.
         \end{proof}

    \begin{rem}
        The lemma indicates that one can view $W/G$ as an object of
        complex dimension $\hat{c}_W$. Under the shift, $\ch_{J}$
        has degree $0$. On the other hand, the untwisted sector has
        degree $\hat{c}_W$, and the sector $\ch_{J^{-1}}$ has degree
        $2\hat{c}_W$.
    \end{rem}

    \begin{rem}
        In the usual orbifold theory, the unit comes from the
        untwisted sector. In our case, the unit element is from
        $\ch_{J}$. In this sense, our theory is quite different
        from usual orbifold theory and instead corresponds to the
        so-called \emph{$(a,c)$-ring} in physics.
        \end{rem}

\section{Virtual Cycles and Axioms}\label{sec:axioms}

In this section, we will discuss the main properties of the virtual
cycles $\left[\W_{g,k}(W; \bgamma)\right]^{vir}$. These are the key
ingredients in the definition of our invariants. We formulate the main
properties of the virtual cycle as axioms similar to those of the
virtual fundamental cycle of stable maps \cite{CR1} and generalizing
the axioms of $r$-spin curves listed in \cite[\S4.1]{JKV1}.

In the special case of the $A_{r-1}$ singularity, an algebraic virtual
class satisfying the axioms of \cite[\S4.1]{JKV1} has been constructed
for the twisted sectors (often called narrow sectors) by
Polishchuk and Vaintrob \cite{P,PV}.  A similar class has been
constructed by Chiodo in K-theory \cite{Ch1,Ch2}, and an analytic
class has been proposed by T.~Mochizuki \cite{Mo}---modeled after
Witten's original sketch.

\subsection{$\left[\W(\Gamma)\right]^{vir}$ and its axioms}\label{subsec:axioms}

\subsubsection{Review of the construction}

The construction of the virtual cycle $\left[\W_{g,k}(W)\right]^{vir}$ is
highly nontrivial.   The details of the construction and the proof of the axioms are presented in \cite{FJR-WEVFC}, but we will outline the main ideas here and then focus the rest of this paper on the consequences of the axioms.

The heart of our construction is the analytic
problem of solving the moduli problem for the Witten equation.
The Witten equation is a first order elliptic PDE of the form
       $$\bar{\partial} S_i+\overline{\frac{\partial W}{\bar{\partial}
    s_i}}=0,$$ where $S_i$ is a $C^{\infty}$-section of $\LL_i$.

Our goal is to construct a virtual cycle of the moduli space of
 solutions of the Witten equation.
Let's briefly outline the construction. Let $\W_{g,k}(\gamma_1,
\dots, \gamma_k)$ be the moduli space of $W$-structures decorated
with the orbifold structure defined by $\gamma_i$ at the $i$-th marked point. It can be considered
as the background data to set up the Witten equation.

To make this construction requires that we leave the algebraic world and enter the world of differential geometry and
analysis. The stack (orbifold) $\W_{g,k}(\gamma_1,
\dots, \gamma_k)$ has a geometric structure similar to $\MM_{g,k}$, including a stratification described by dual graphs and something like the gluing structure
at a node. Our starting point is to give a differential geometric structure of $\W_{g,k}(\gamma_1,
\dots, \gamma_k)$. This can be done in a fashion similar to that for $\MM_{g,k}$ \cite[\S2.2]{FJR-WEVFC}.
The variable in the Witten equation is a smooth section of the $W$-structure $\oplus_i L_i$, while the target
of the Witten equation is the space of its $(0,1)$-forms. Formally, the Witten equation can be phrased as
a Fredholm section of a Banach bundle over a fiber-wise Banach manifold. 
Unfortunately, it is rather difficult to solve the
Witten equation due to the fact that the singularity of $W$ has high multiplicity. It
is much easier to solve a perturbed equation of the form $W+W_0$, where
$W_0$ is a linear perturbation term such that
$W_{\gamma}+W_{0\gamma}$ is a holomorphic Morse function for every
$\gamma$. Here  $W_{\gamma}$ and $W_{0\gamma}$ are the restrictions of
$W$, and  $W_0$, respectively, to the fixed point set $\C^{N_{\gamma}}$. The background data for
the perturbed Witten equation is naturally  the
moduli space (stack) of \emph{rigidified} $W$-structures $\W^{rig}_{g,k}(\gamma_1,
\dots, \gamma_k)$.

 The
crucial part of the analysis is to show that a solution of the
Witten equation converges to a critical point of
$W_{\gamma_i}+W_{0\gamma_i}$. This enables us to construct a
moduli space (stack)
$
\W^s_{g,k}(\kappa_{j_1}, \dots, \kappa_{j_k})
$
of solutions of the perturbed Witten equation converging to the
critical point $\kappa_i$ at the marked point $x_i$.  We call
$W_0$ \emph{strongly regular} if (i) $W_{\gamma_i}+W_{0\gamma_i}$
is holomorphic Morse; (ii) the critical values of
$W_{\gamma_i}+W_{0\gamma_i}$ have distinct imaginary parts. The
first important result is

 \begin{thm}
 If $W_0$ is strongly regular, then $\W^s_{g,k}(\kappa_{j_1}, \dots, \kappa_{j_k})$
 is compact and has a virtual fundamental cycle $[\W^{s}_{g,k}(\kappa_{j_1}, \dots, \kappa_{j_k})]^{vir}$
 of  degree
$$
2\left((c_W-3)(1-g)+k -\sum_i \iota_{\gamma_i}\right)-\sum_i N_{\gamma_i}.
$$
Here, $\iota_{\gamma_i}$ is the degree-shifting  number defined
previously.
\end{thm}

It turns out to be  convenient to map the above
virtual cycle into $H_*(\W^{rig}_{g,k}, \Q)$, even though
       it is not an element of the latter in any way.

The state space of the theory (or rather its dual) enters in a surprising new way, as we now describe.   It
       turns out that the above virtual cycle \emph{does} depends on the
       perturbation. It will change whenever $W_0$ fails to be
       strongly regular. We observe that for a strongly regular perturbation we can construct a
       canonical system of \emph{horizontal} paths $u^{\pm}_i$'s and the associated Lefschetz thimble $\Delta^{\pm}_i$.
       When we perturb $W_0$ crossing the ``wall" (where the imaginary
       parts of critical values happen to be the same), we arrive at another canonical system of paths and its
        Lefschetz thimble $\Delta'^{\pm}_i$. The relation between $\Delta^{\pm}_i$ and $\Delta'^{\pm}_i$
        is determined by the well-known \emph{Picard-Lefschetz formula}.   The ``wall crossing formula" for virtual fundamental cycles
        can be summarized in the
       following \emph{quantum Picard-Lefschetz theorem}.  For a more precise statement of this theorem, see \cite[\S6.1, esp. Thm 6.1.6]{FJR-WEVFC}.
       \begin{thm}
       When $W_0$ varies, $[\W^{s}_{g,k}(\kappa_{j_1}, \dots, \kappa_{j_k})]^{vir}$
       transforms in the same way as the Lefschetz thimble
       $\Delta^-_{j_i}$ attached to the critical point $\kappa_{j_i}$.
       \end{thm}

       The $\Delta^+_i$'s transform in the opposite way as $\Delta^-_i$'s. It is well-known that
       the ``diagonal class" $\sum_i \Delta^-_i\otimes \Delta^+_i$ is independent of perturbation, and this suggests the following definition of an ``extended virtual class".
       To        simplify the notation, we assume that there is only one
       marked point with the orbifold decoration $\gamma$. Then,
       the wall crossing formula of $[\W^s_{g,1}(\kappa_i)]^{vir}$
       shows precisely that $\sum_j [\W^s_{g,1}(\kappa_j)]^{vir}\otimes
       \Delta^+_j$,
       viewed as a class in $H_*(\W^{rig}_{g,1}(\gamma), \Q)\otimes
       H_{N_{\gamma}}(\C^{N_{\gamma}}_{\gamma}, W^{\infty}_{\gamma},
       \Q)$, is independent of the perturbation. Now, we define
       $$[\W^s_{g,1}(\gamma)]^{vir}:=\sum_j [\W^s_{g,1}(\kappa_j)]^{vir}\otimes
       \Delta^+_j.$$
       The above definition can be generalized to multiple
       marked points in an obvious way, so that
       $$[\W^s_{g,k}(\gamma_1, \dots, \gamma_k)]^{vir}\in
       H_*(\W^{rig}_{g,k}(\gamma_1, \dots, \gamma_k), \Q)\otimes \prod_i H_{N_{\gamma_i}}(\C^{N_{\gamma_i}},
       W^{\infty}_{\gamma_i}, \Q)$$
       has  degree
       $$2\left((c_W-3)(1-g)+k -\sum_i \iota_{\gamma_i}\right).$$

 \begin{crl}
    $[\W^s_{g,k}(\gamma_1, \dots, \gamma_k)]^{vir}$ is
    independent of the perturbation $W_0$.
    \end{crl}

Of course, $W_0$ is only part of the perturbation data. Eventually, we want to work on  the stack $\W_{g,k}$. It is known  that the map $\so: \W^{rig}_{g,k}\rightarrow \W_{g,k}$, defined by forgetting all the
rigidifications, is quasi-finite and proper, so we can define
\begin{equation}\label{eq:df-vircyc}
[\W_{g,k}(\gamma_1, \dots,
\gamma_k)]^{vir}:=\frac{(-1)^{D}}{\deg(\so)}(\so)_*[\W^s_{g,k}(\gamma_1,
\dots, \gamma_k)]^{vir},
\end{equation}
where $-D$ is the sum of the indices of the $W$-structure bundles:
\begin{equation}\label{eq:D}
D:=-\sum_{i=1}^N \ind(\LL_i) =\chat_W(g-1)+\sum_{j=1}^k \iota_{\gamma_j}.
\end{equation}

\begin{rem}
The sign $(-1)^D$ is put here to match the older definition in the $r$-spin case.
\end{rem}

The fact that the above virtual cycle is independent of the rigidification
implies that
    $$[\W_{g,k}(\gamma_1, \dots, \gamma_k)]^{vir}\in
       H_*(\W_{g,k}(\gamma_1, \dots, \gamma_k), \Q)\otimes \prod_i H_{N_{\gamma_i}}(\C^{N_{\gamma_i}},
       W^{\infty}_{\gamma_i}, \Q)^{G_W}.$$

More generally, we have the following definition.
\begin{df}
Let $\Gamma$ be a decorated stable $W$-graph (not necessarily
connected) with each tail $\tau\in T(\Gamma)$ decorated by an element
$\gamma_\tau \in \grp$.  Denote by $k:=|T(\Gamma)|$ the number of
tails of $\Gamma$.   We
define the \emph{virtual cycle}
$$\left[\W(\Gamma)\right]^{vir}  \in
H_*(\W(\Gamma),\Q)\otimes \prod_{\tau \in T(\Gamma)}
H_{N_{\gamma_{\tau}}}(\C^{N_{\gamma_{\tau}}},W^{\infty}_{\gamma_\tau},
\Q)^{G_W}
$$
as given in Equation~\ref{eq:df-vircyc}.

When $\Gamma$ has a single vertex of genus $g$, $k$ tails, and no edges
(i.e, $\Gamma$ is a corolla), we denote the virtual cycle by
$\left[\W(\bgamma)\right]^{vir}$, where $\bgamma :=(\gamma_1,
\dots, \gamma_k)$.
\end{df}

\subsubsection{The virtual cycle for admissible subgroups}

We now wish to consider the more general case of admissible subgroups. Recall that $G$ is admissible if $G=G_{\tW}$ for some $\tW=W+Z$.  One can show \cite[Rem 2.3.11]{ChR} that the stack $ \W_{g,k,G} := \W_{g,k}(\tW)$ is independent of the choice of $Z$, provided $G=G_{\tW}$.

Denote by $\inc$ and $\inc^{rig}$ the natural  morphisms of stacks
$$\inc: \W_{g,k,G} = \W_{g,k}(\tW_t)\rTo \W_{g,k}(W) \dsand \inc^{rig}: \W^{rig}_{g,k,G}= \W_{g,k}(\tW_t)\rTo \W^{rig}_{g,k}(W),$$ respectively.  And denote by $\so_G$ the restriction of $\so$ to $\W^{rig}_{g,k,G}$

\begin{df}
Define
$$[\W^{rig}_{g,k,G}(W;\bgamma)]^{vir}:= \inc^{rig,*}\left([\W^{rig}_{g,k}(W;\bgamma)]^{vir}\right)\in
       H_*(\W^{rig}_{g,k,G}(W;\gamma_1, \dots, \gamma_k), \Q)\otimes \prod_i H_{N_{\gamma_i}}(\C^{N_{\gamma_i}},
       W^{\infty}_{\gamma_i}, \Q),$$
and 
$$
[\W_{g,k,G}(W;\gamma_1, \dots,
\gamma_k)]^{vir}:=\frac{(-1)^{D}}{\deg(\so_G)}(\so_G)_*[\W^{rig}_{g,k,G}(W;\gamma_1,
\dots, \gamma_k)]^{vir},
$$
so that 
$$[\W_{g,k,G}(W;\gamma_1, \dots,
\gamma_k)]^{vir} \in 
  H_*(\W_{g,k,G}(W;\gamma_1, \dots, \gamma_k), \Q)\otimes \prod_i H_{N_{\gamma_i}}(\C^{N_{\gamma_i}},
       W^{\infty}_{\gamma_i}, \Q)^{G}.$$
\end{df}
On the other hand, for any $\tW$ with $G_{\tW} = G$ one may consider the virtual cycle
$$[\W^{rig}_{g,k}(\tW;\bgamma)]^{vir}:= \inc^{rig,*}\left([\W^{rig}_{g,k}(W;\bgamma)]^{vir}\right)\in
       H_*(\W^{rig}_{g,k}(\tW;\gamma_1, \dots, \gamma_k), \Q)\otimes \prod_i H_{N_{\gamma_i}}(\C^{N_{\gamma_i}},
       \tW^{\infty}_{\gamma_i}, \Q),$$
and the pushforward 
$$
[\W_{g,k}(\tW;\gamma_1, \dots,
\gamma_k)]^{vir}:=\frac{(-1)^{D}}{\deg(\so_G)}(\so_G)_*[\W^{rig}_{g,k,G}(\tW;\gamma_1,
\dots, \gamma_k)]^{vir},
$$
in 
$H_*(\W_{g,k}(\tW;\gamma_1, \dots, \gamma_k), \Q)\otimes \prod_i H_{N_{\gamma_i}}(\C^{N_{\gamma_i}},
       \tW^{\infty}_{\gamma_i}, \Q)^{G}$

Note that we have a canonical isomorphism of $G$-representations $$H_{N_{\gamma}}\left(\C^{N_{\gamma}},\left(\tW\right)_{\ga}^{\infty}, \Q\right) = H_{N_{\gamma}}\left(\C^{N_{\gamma}},W_{\ga}^{\infty}, \Q\right).$$  

\begin{prop}
The cycles $[\W^{rig}_{g,k}(\tW;\bgamma)]^{vir}$ and $[\W^{rig}_{g,k,G}(W;\bgamma)]^{vir}$ are equal in 
$H_*(\W^{rig}_{g,k}(\tW;\gamma_1, \dots, \gamma_k), \Q)\otimes \prod_i H_{N_{\gamma_i}}(\C^{N_{\gamma_i}},
       \tW^{\infty}_{\gamma_i}, \Q)$, and thus the 
pushforwards also agree:
$$[\W_{g,k,G}(W;\gamma_1, \dots,
\gamma_k)]^{vir} = [\W_{g,k}(\tW;\gamma_1, \dots,
\gamma_k)]^{vir}$$
\end{prop}
\begin{proof}
This follows from the deformation invariance axiom of  \cite[Thm 6.2.1(9)]{FJR-WEVFC}.  Namely, if we let $t \in [0,1]\subset \R$ be a parameter and let $\tW_t$ denote the family of quasi-homogeneous polynomials $$\tW_t := W+tZ.$$  Since $W$ is non-degenerate, so is $\tW_t$ for every $t \in [0,1]$.  The definition of the stack $\W_{g,k}(\tW_t)$ is independent of $t$, provided $t\neq0$, and for notational convenience, we also define $\W_{g,k}(\tW_0)$ to be equal to $\W_{g,k}(\tW_{t\neq0})$.  It is clear that the cycles
$[\W^{rig}_{g,k}(\tW_0;\bgamma)]^{vir}$ and $[\W^{rig}_{g,k,G}(W;\bgamma)]^{vir}$ are equal, and the deformation invariance axiom of  \cite[Thm 6.2.1(9)]{FJR-WEVFC} shows that for all $t\in[0,1]$ the cycles $[\W^{rig}_{g,k}(\tW_t;\bgamma)]^{vir}$ are all equal.
\end{proof}

The following theorem  now follows immediately from \cite[Thms 1.2.5 and 6.2.1]{FJR-WEVFC}.
\begin{thm}\label{thm:main}
For any admissible group $G$ and any $W$-graph $\Gamma$. The following axioms are satisfied for
$\left[\W(\Gamma)\right]^{vir}$:
\begin{enumerate}
\item \textbf{Dimension:}\label{ax:dimension} If $D$ is not 
an integer, 
 then
  $\left[\W(\Gamma)\right]^{vir}=0$. Otherwise, the cycle
  $\left[\W(\Gamma)\right]^{vir}$ has degree
\begin{equation}\label{eq:dimension}
6g-6+2k-2\#E(\Gamma)-2D =2\left((\chat-3)(1-g) + k -\#E(\Gamma) - \sum_{\tau\in
T(\Gamma)} \iota_{\tau}\right).
\end{equation}
So the cycle lies in $H_r(\W(\Gamma),\Q)\otimes \prod_{\tau \in
  T(\Gamma)}
H_{N_{\gamma_{\tau}}}(\C^{N_{\gamma_{\tau}}},W^{\infty}_{\gamma_\tau},
\Q),$ where
$$
r:=6g-6+2k-2\#E(\Gamma) -2D -\sum_{\tau\in T(\Gamma)}{N_{\gamma_\tau}}= 2\left((\hat{c}-3)(1-g)+k- \#E(\Gamma) -\sum_{\tau\in T(\Gamma)}\iota(\gamma_{\tau})- \sum_{\tau\in T(\Gamma)}\frac{N_{\gamma_\tau}}{2}\right).
$$

\item \label{ax:symm}\textbf{Symmetric group invariance}: There is a
  natural $S_k$-action on $\W_{g,k}$ obtained by permuting the
  tails.  This action induces an action on homology.  That is, for any
  $\sigma \in S_k$ we have:
$$\sigma_*: H_*(\W_{g,k},\Q)\otimes \prod_i
H_{N_{\gamma_{i}}}(\C^{N_{\gamma_i}}, W^{\infty}_{\gamma_i}, \Q)^{G}
\rTo{}{}  H_*(\W_{g,k},\Q)\otimes \prod_i
H_{N_{\gamma_{\sigma(i)}}}(\C^{N_{\gamma_{\sigma(i)}}},
W^{\infty}_{\gamma_{\sigma(i)}}, \Q)^{G}.$$ For any decorated graph
$\Gamma$, let $\sigma\Gamma$ denote the graph obtained by applying
$\sigma$ to the  tails of $\Gamma$.

We have
\begin{equation}\sigma_*\left[\W(\Gamma)\right]^{vir} = \left[\W(\sigma\Gamma)\right]^{vir}.\end{equation}

 \item \textbf{Degenerating connected graphs:} \label{ax:ConnGraphs} Let $\Gamma$ be a
   connected, genus-$g$, stable, decorated $W$-graph.

The cycles $\left[\W(\Gamma)\right]^{vir}$ and
$\left[\W_{g,k}(\bgamma)\right]^{vir}$ are related by
\begin{equation}
 \left[\W(\Gamma)\right]^{vir}=
\tilde{i}^*\left[\W_{g,k}(\bgamma)\right]^{vir},
\end{equation}
where $\tilde{i} : \W(\Gamma) \rTo{}{} \W_{g,k}(\bgamma)$ is the
canonical inclusion map.
\item \textbf{Disconnected graphs:} Let
$\Gamma =\coprod_{i} \Gamma_i$ be a stable, decorated $W$-graph
which is the disjoint union of connected $W$-graphs $\Gamma_i$.
The classes $\left[\W(\Gamma)\right]^{vir}$ and
$\left[\W(\Gamma_i)\right]^{vir}$ are related by
\begin{equation}
\left[\W(\Gamma)\right]^{vir}=
\left[\W(\Gamma_1)\right]^{vir} \times \cdots \times
\left[\W(\Gamma_d)\right]^{vir}.
\end{equation}
\smallskip

\item \textbf{Topological Euler class for the
narrow sector:}  Suppose that all the decorations on tails of $\Gamma$ are
  \emph{narrow}, meaning that $\C^{N_{\gamma_i}}=\{0\}$, and so we can
  omit $H_{N_{\gamma_{i}}}(\C^{N_{\gamma_i}}, W^{\infty}_{\gamma_i}, \Q)=\Q$ from our notation.

Consider the universal $W$-structure $(\LL_1,\dots,\LL_N)$ on the
universal curve $\pi:\cC \rTo{}{} \W(\Gamma)$ and the two-term
complex of sheaves
$$\pi_*(|\LL_i|){\rTo} R^1\pi_*(|\LL_i|).$$
There is a family of maps
$$W_i=\frac{\partial W}{\partial x_i}: \pi_*(\bigoplus_j |\LL_j|)
\rTo \pi_*(K\otimes |\LL_i|^*)\cong R^1\pi_*(|\LL_i|)^*.$$
The above two-term complex is quasi-isomorphic to a complex
of vector bundles \cite{PV}
$$E^0_i\rTo^{d_i} E^1_i$$
such that
$$\ker( d_i)\rTo \coker(d_i)$$
is isomorphic to the original two-term complex.
    $W_i$ is naturally extended (denoted by the same notation) to
    $$\bigoplus_i E^0_i\rTo (E^1_i)^*.$$
    Choosing an Hermitian metric on $E^1_i$ defines an
    isomorphism $\bar{E}^{1*}_i\cong E^1_i.$ Define the \emph{Witten map} to be the following
    $$\wit=\bigoplus (d_i+\bar{W}_i):
    \bigoplus_i E^0_i\rTo\bigoplus_i \bar{E}^{1*}_i\cong \bigoplus_i E^1_i.$$
 Let $\pi_j: \bigoplus_i E^j_i\rTo \MM$ be
    the projection map. The Witten map defines a proper section (also denoted $\wit$) 
    $\wit: \bigoplus_i E^0_i\rTo \pi^*_0 \bigoplus_i
    E^1_i$  
    of the bundle $\pi^*_0 \bigoplus_i
    E^1_i$ over $\bigoplus_i E^0_i$.  The above data  defines a topological Euler
    class $\eul\left(\wit: \bigoplus_i E^0_i\rTo \pi_0^* \bigoplus_i
    E^1_i\right)$. Then,
    $$[\W_{\Gamma}]^{vir}=(-1)^{D}\eul\left(\wit: \bigoplus_i E^0_i\rTo \pi^*_0 \bigoplus_i
    E^1_i\right)\cap [\MM_{\Gamma}].$$

    The above axiom implies two subcases.
\begin{enumerate}
\item{\bf Concavity}:\label{ax:convex}\footnote{This axiom was
  called \emph{convexity} in \cite{JKV1} because the original form of
  the construction outlined by Witten in the $A_{r-1}$ case involved
  the Serre dual of $\LL$, which is convex precisely when our $\LL$ is
  concave.}

    Suppose that all tails of $\Gamma$ are narrow.  If
$\pi_*\left(\bigoplus_{i=1}^t\LL_i\right)=0$, then the virtual
cycle is given by capping the top Chern class of the dual $\left(R^1 \pi_* \left(\bigoplus_{i=1}^t\LL_i\right)\right)^*$ of the pushforward
with the usual fundamental cycle of the moduli space:
\begin{equation}\begin{split}
\left[\W(\Gamma)\right]^{vir}& = c_{top}\left(\left(R^1\pi_*\bigoplus_{i=1}^t\LL_i \right)^*\right) \cap \left[\W(\Gamma)\right]\\
& = (-1)^D c_{D}\left(R^1\pi_*\bigoplus_{i=1}^t\LL_i \right)
\cap \left[\W(\Gamma)\right].
\end{split}
\end{equation}
\item   {\bf  Index zero:} \label{ax:wittenmap} Suppose that  $\dim (\W(\Gamma))=0$
and all the decorations on tails  are narrow.

If the pushforwards $\pi_* \left(\bigoplus\LL_i\right)$ and $R^1\pi_*
\left(\bigoplus \LL_i\right)$ are both vector bundles of the same
rank, then the virtual cycle is just the degree $\deg(\wit)$ of the
Witten map times the fundamental cycle:
$$\left[\W(\Gamma)\right]^{vir} = \deg(\wit)\left[\W(\Gamma)\right],$$
\end{enumerate}

\item\textbf{Composition law:}\label{ax:cutting} Given any genus-$g$
  decorated stable $W$-graph $\Gamma$ with $k$ tails, and given any
  edge $e$ of $\Gamma$, let $\hGamma$ denote the graph obtained by
  ``cutting'' the edge $e$ and replacing it with two unjoined tails
  $\tau_+$ and $\tau_-$ decorated with $\gamma_+$ and $\gamma_-$,
  respectively.

The fiber product
$$F:=\W(\hGamma)\times_{\W(\Gamma)} \W(\Gamma)$$
has morphisms
 $$ \W({\hGamma})\lTo^{q} F
\rTo^{pr_2}\W(\Gamma).$$

We have
\begin{equation}\label{eq:cutting}
\left\langle \left[\W(\hGamma)\right]^{vir}\right\rangle_{\pm}=\frac{1}{\deg(q)}q_*pr_2^*\left(\left[\W(\Gamma)\right]^{vir}\right),
\end{equation}
where $\langle  \rangle_{\pm}$ is the map from
$$H_*(\W(\hGamma))\otimes\prod_{\tau \in T(\Gamma)} H_{N_{\gamma_{\tau}}}(\C^{N_{\gamma_{\tau}}},W^{\infty}_{\gamma_\tau}, \Q)^G \otimes  H_{N_{\gamma_{+}}}(\C^{N_{\gamma_{+}}},W^{\infty}_{\gamma_+}, \Q)^G\otimes H_{N_{\gamma_{-}}}(\C^{N_{\gamma_{-}}},W^{\infty}_{\gamma_-}, \Q)^G$$ to
$$ H_*(\W(\hGamma))\otimes\prod_{\tau \in T(\Gamma)} H_{N_{\gamma_{\tau}}}(\C^{N_{\gamma_{\tau}}},W^{\infty}_{\gamma_\tau}, \Q)^G$$ obtained by contracting  the last two factors via the pairing
$$\langle\, , \rangle: H_{N_{\gamma_{+}}}(\C^{N_{\gamma_{+}}},W^{\infty}_{\gamma_+}, \Q)^G\otimes H_{N_{\gamma_{-}}}(\C^{N_{\gamma_{-}}},W^{\infty}_{\gamma_-}, \Q)^G \rTo \Q.$$
\item
\textbf{Forgetting tails:}\label{ax:tails}
\begin{enumerate}
\item
 Let $\Gamma$ have its $i$th tail decorated with $J$, where $J$
 is the exponential grading element of $G$. Further let $\Gamma'$ be
 the decorated $W$-graph obtained from $\Gamma$ by forgetting the
 $i$th tail and its decoration.  Assume that $\Gamma'$ is stable, and
 denote the forgetting tails morphism by $$\vartheta: \W(\Gamma)
 \rTo{}{} \W(\Gamma').$$ We have
\begin{equation}
\left[\W(\Gamma)\right]^{vir}
=\vartheta^*\left[\W(\Gamma')\right]^{vir}.
\end{equation}
    \item In the case of $g=0$ and $k=3$, the space
      $\W(\gamma_1, \gamma_2, J)$ is empty if
      $\gamma_1\gamma_2\neq 1$ and $\W_{0,3}(\gamma, \gamma^{-1},
      J)=\BG_W$.  We omit
      $H_{N_{J}}(\C^{N_{J}},W^{\infty}_{J}, \Q)^{G_W} = \Q$
      from the notation.  In this case, the cycle
            $$\left[\W_{0,3}(\gamma, \gamma^{-1},
        J)\right]^{vir}\in H_*(\BG_W,\Q)\otimes
      H_{N_{\gamma}}(\C^{N_{\gamma}},W^{\infty}_{\gamma}, \Q)^{G} \otimes
      H_{N_{\gamma^{-1}}}(\C^{N_{\gamma^{-1}}},W^{\infty}_{\gamma^{-1}},
      \Q)^{G}$$ is the fundamental cycle of $\BG_W$ times the Casimir
      element. Here the Casimir element is defined as follows. Choose
      a basis $\{\alpha_i\}$ of
      $H_{N_{\gamma}}(\C^{N_{\gamma}},W^{\infty}_{\gamma}, \Q)^{G},$ and a
      basis $\{ \beta_j\}$ of
      $H_{N_{\gamma^{-1}}}(\C^{N_{\gamma^{-1}}},W^{\infty}_{\gamma^{-1}},
      \Q)^{G}$. Let $\eta_{ij}=\langle \alpha_i, \beta_j\rangle $ and
      $(\eta^{ij})$ be the inverse matrix of $(\eta_{ij})$. The
      Casimir element is defined as $\sum_{ij}\alpha_i\eta^{ij}\otimes
      \beta_j.$
\end{enumerate}

\item \textbf{Sums of singularities:} \label{ax:sums} If $W_1 \in \C[z_1,\dots,z_t]$
  and $W_2 \in C[z_{t+1},\dots,z_{t+t'}]$ are two quasi-homogeneous
  polynomials with admissible groups $G_1$ and $G_2$, respectively,
  then  $G_{1} \times G_{2}$ is an admissible group of automorphisms of $W_1+W_2$ whose  state space $\ch_{W_1+W_2, G_1\times G_2}$ is 
  naturally isomorphic to the tensor product
\begin{equation}
\ch_{W_1+W_2,G_1\times G_2} = \ch_{W_1,G_1} \otimes \ch_{W_2,G_2},
\end{equation}
and the stack $\W_{g,k,G_1\times G_2}$ has a natural map to the fiber
product $$\W_{g,k,G_1\times G_2}(W_1+W_2)\rTo^{\gerbeprod}  \W_{g,k,G_1}(W_1) \times_{\MM_{g,k}}
\W_{g,k,G_2}(W_2).$$

Indeed, since any $G_1\times G_2$-decorated stable graph $\Gamma$
induces a $G_1$-decorated graph $\Gamma_1$ and
$G_2$-decorated graph $\Gamma_2$ with the same underlying graph
$\barGamma$, we have
\begin{equation}
\W(W_1+W_2,\Gamma) \rTo^{\gerbeprod} \W(W_1,\Gamma_1) \times_{\MM(\barGamma)}
\W(W_2,\Gamma_2).
\end{equation}

Composing with the natural inclusion
$$\W_{g,k,G_1}(W_1) \times_{\MM_{g,k}}
\W_{g,k,G_2}(W_2)\rInto^{\Delta} \W_{g,k,G_1}(W_1) \times \W_{g,k,G_2}(W_2),$$
and using the isomorphism of middle homology gives a
homomorphism
\begin{multline*}
\gerbeprod^*\circ\Delta^*:\left(H_*(\W_{g,k,G_1}(W_1),\Q)\otimes \prod_{i=1}^k
H_{N_{\gamma_{i,1}}}(\C^{N_{\gamma_{i,1}}},(W_1)^{\infty}_{\gamma_{i,1}},
\Q)^{G_1}\right) \otimes \left(H_*(\W_{g,k,G_2}(W_2),\Q)\otimes
\prod_{i=1}^k
H_{N_{\gamma_{i,2}}}(\C^{N_{\gamma_{i,2}}},(W_2)^{\infty}_{\gamma_{i,2}},
\Q)^{G_2}\right) \\
\rTo H_*(\W_{g,k,G_1\times G_2}(W_1+W_2),\Q)\otimes \prod_{i=1}^k
H_{N_{(\gamma_{i,1},\gamma_{i,2})}}(\C^{N_{(\gamma_{i,1},\gamma_{i,2})}},W^{\infty}_{(\gamma_{i,1},\gamma_{i,2})},
\Q)^{G_1\times G_2}.
\end{multline*}

The virtual cycle satisfies
\begin{equation}
\gerbeprod^*\circ\Delta^*\left(\left[\W_{g,k,G_1}(W_1)\right]^{vir}\otimes
\left[\W_{g,k,G_2}(W_2)\right]^{vir}\right) =
\left[\W_{g,k,G_1\times G_2}(W_1+W_2)\right]^{vir}.
\end{equation}

\item \label{ax:def-invar} \textbf{Deformation Invariance:} Let $W_t \in
\C[z_1,\dots,z_N]$ be a family of non-degenerate quasi-homogeneous
polynomials depending smoothly on a parameter $t \in [a,b] \subset \R$. 
Suppose that $G$ is the
common automorphism group of $W_t$.  The corresponding stacks $\W(\Gamma_t)$ are all naturally isomorphic.  We denote this generic stack by $\W(\Gamma)$.
The virtual cycle
$[\W(\Gamma)]^{vir}$ associated to $(W_{t}, G)$ is
independent of $t$.

\item \textbf{$G_W$-Invariance}\label{ax:GW-invar}
For any admissible $G$ and any $G$-decorated graph $\Gamma$
the homology $H_*(\W_{g,k,G}(\Gamma),\Q)$ as well as the homology groups $H_{N_{\gamma_{\tau}}}(\C^{N_{\gamma_{\tau}}},(W)^{\infty}_{\gamma_{\tau}},
\Q)^{G}$ each have a natural $G_W$-action, which induces a $G_W$ action on 
$H_*(\W_{g,k,G}(\Gamma),\Q)\otimes \prod_{\tau\in T(\Gamma)} H_{N_{\gamma_{\tau}}}(\C^{N_{\gamma_{\tau}}},(W)^{\infty}_{\gamma_{\tau}},
\Q)^{G}$.

The virtual cycle $[\W(\Gamma)]^{vir} $ is invariant under this $G_W$-action.
\end{enumerate}
\end{thm}

\begin{rem}
In the case of $A_{r-1}$ our virtual cycle can be used to construct an $r$-spin virtual class in the sense of \cite[\S4.1]{JKV1}.  The details of this construction are given in \cite{FJR-ArSpin}
\end{rem}

  \begin{rem}
   As usual, we can define Gromov-Witten type correlators by
   integrating tautological classes such as $\psi_i$ and $\mu_{ij}$
   over the $\left[\W_{g,k,G}\right]^{vir}$.

A direct consequence of the above axioms is the fact that the above
correlators defined by $\psi_i$, together with the rescaled pairing
$(\, ,)_{\gamma} :=\frac{|\langle\gamma\rangle| }{|G|}\langle\, ,
\rangle_{\gamma}$, satisfy the usual axioms of Gromov-Witten theory
(without the divisor axiom) and a modified version of the unit axiom
    $$\langle\alpha_1, \alpha_2,\bone_{J}\rangle^W_0=|\langle\gamma\rangle |\langle\alpha_1, \alpha_2\rangle $$
    for $\alpha_1\in \ch_{\gamma}$ and for $\alpha_2\in \ch_{\gamma^{-1}}.$ In this paper, we favor a slightly
    different version, which we now explain.
    \end{rem}

\subsection{Cohomological field theory}\label{sec:CohFT}
One gets a cleaner formula by pushing
$\left[\W_{g,k,G}(\bgamma)\right]^{vir}$ down to $\MM_{g,k}$.

\begin{df}

Let $\Lambda^W_{g,k} \in \hom(\ch_W^{\otimes k}, H^*(\MM_{g,k}))$ be
given for homogeneous elements $\balpha :=(\alpha_1, \dots, \alpha_k)$
with $\alpha_i \in \ch_{\gamma_i}$ by
 \begin{equation}\label{eq:defLambda}
\Lambda^{W,G}_{g,k}(\balpha) := \frac{|G|^g}{\deg(\st)}PD \,\st_*\left(\left[\W_{g,k}(W,\bgamma)\right]^{vir}
\cap \prod_{i=1}^k \alpha_i \right),
\end{equation}
and then extend linearly to general elements of $\ch_{W,G}^{\otimes k}$. Here, $PD$ is the Poincare duality map.

Let $\bone_{1}:=\unit$ be the distinguished generator of $\ch_{J}$, and let $\langle\, ,\rangle^{W,G}$ denote the pairing on the state space $\ch_{W,G}$.

\end{df}

\begin{thm}\label{thm:CohFT}

The collection $(\ch_{W,G}, \langle\, , \rangle^{W,G}, \{\Lambda^{W,G}_{g,k}\}, \bone_1)$
is a cohomological field theory with flat identity.

Moreover, if $W_1$ and $W_2$ are two singularities in distinct
variables with admissible groups $G_1$ and $G_2$, respectively, then the cohomological field theory arising from $W_1+W_2, G_1\times G_2$ is the
tensor product of the cohomological field theories arising from $W_1, G_1$
and $W_2, G_2$:
$$(\ch_{W_1+W_2, G_1\times G_2},  \{\Lambda^{W_1+W_2, G_1\times G_2}_{g,k}\}) = (\ch_{W_1, G_1}\otimes\ch_{W_2, G_2}, \{\Lambda^{W_1,G_1}_{g,k} \otimes \Lambda^{W_2,G_2}_{g,k}\}). $$

\end{thm}

\begin{proof}
To show that the classes form a cohomological field theory, we must
show that the following properties hold (see, for example,
\cite[\S3.1]{JKV1}):
\begin{enumerate}
\item[\bf C1.] The element $\Lambda^{W,G}_{g,n}$ is invariant under the
  action of the symmetric group $S_k$.

\item[\bf C2.]
Let $g=g_1+g_2$; let $k=k_1+k_2$; and let
$$\rho_{tree}: \overline{\M}_{g_1, k_1+1}\times \overline{\M}_{g_2,
    k_2+1}\rTo \overline{\M}_{g, k}
$$
be the gluing trees morphism~(\ref{eq:glueTreeCurves}).
Then the forms
$\Lambda^{W,G}_{g,n}$ satisfy the composition property
\begin{multline} \label{eq:cfttree}
\rho_{\mathrm{tree}}^* \Lambda^{W,G}_{g_1+g_2,k}(\alpha_1,\alpha_2\,\ldots,\alpha_k)=
\\
\sum_{\mu,\nu} \Lambda^{W,G}_{g_1,k_1+1}(\alpha_{i_1},\ldots,\alpha_{i_{k_1}},\mu)\,
\eta^{\mu\nu} \otimes
\Lambda^{W,G}_{g_2,k_2+1}(\nu ,\alpha_{i_{{k_1}+1}},\ldots,\alpha_{i_{k_1+k_2}})
\end{multline}
for all $\alpha_i \in \ch_W$, where $\mu$ and $\nu$ run through a
basis of $\ch_W$, and $\eta^{\mu\nu}$ denotes the inverse of the
pairing $\langle\, , \rangle$ with respect to that basis.

\item[\bf C3.]
Let
\begin{equation}
\rho_{loop}: \overline{\M}_{g-1, k+2}\rTo \overline{\M}_{g, k}
\end{equation}
be the gluing loops morphism~(\ref{eq:glueLoopCurves}).
Then
\begin{equation}
\label{eq:cftloop}
\rho_{\mathrm{loop}}^*\,\Lambda^{W,G}_{g,k}(\alpha_1,\alpha_2,\ldots,\alpha_k)\,=\,
\sum_{\mu,\nu}\Lambda^{W,G}_{g-1,k+2}\,(\alpha_1,\alpha_2,\ldots,\alpha_n, \mu,
\nu)\,\eta^{\mu\nu},
\end{equation}
where $\alpha_i$, $\mu$, $\nu$, and $\eta$ are as in C2.

\item[\bf C4a.]
For all $\alpha_i$ in $\ch_W$ we have
\begin{equation}
  \label{eq:identity}
\Lambda^{W,G}_{g,k+1}(\alpha_1,\ldots, \alpha_k, \bone_1) =
\vartheta^*\Lambda^{W,G}_{g,k}(\alpha_1,\ldots, \alpha_k),
\end{equation}
where $\vartheta:\MM_{g,n+1} \rTo{}{} \MM_{g,n}$ is the universal curve.
\item[\bf C4b.]
\begin{equation}
\label{eq:identity2}
\int_{\MM_{0,3}}\,\Lambda^{W,G}_{0,3}(\alpha_1,\alpha_2,\bone_1) =
\langle\alpha_1,\alpha_2\rangle^W.
\end{equation}
\end{enumerate}

Axiom~C1 follows immediately from the symmetric group invariance
(axiom~\ref{ax:symm}) of the virtual cycle.

To prove Axioms~C2 and~C3 we first need a simple lemma: that the Casimir element is Poincar\`e dual to the
pairing.  This is well known, but we include it for completeness
because we use it often.
\begin{lm}
  Let $\alpha_i\in \ch_W$ be a basis. Consider the Casimir element $\sum_{ij}\eta^{ij}\alpha_i\otimes\alpha_j$ of
  its pairing.  For any $u,v\in \ch^*_W$, we have
$$\langle u,v\rangle =u\otimes v \cap \sum_{ij}\eta^{ij}\alpha_i\otimes \alpha_j.$$

\end{lm}
\begin{proof}
Let $\alpha^*_i$ be the dual basis and let $u:=\sum_i \langle u,
\alpha_i\rangle \alpha^*_i,$ and $v:=\sum_j \langle v,
\alpha_j\rangle \alpha^*_j$. Therefore,
$$\langle u,v\rangle =\sum_{ij}\langle u, \alpha_i\rangle \langle v,
\alpha_j\rangle \langle \alpha^*_i, \alpha^*_j\rangle .$$ Notice that
$\eta_{ij}=\langle \alpha_i, \alpha_j\rangle $ and $\eta^{ij}=\langle
\alpha^*_i, \alpha^*_j\rangle $. The right hand side is precisely
$u\otimes v\cap \sum_{ij}\eta^{ij}\alpha_i\otimes \alpha_j.$
\end{proof}

Let $\alpha_i \in \ch_{\gamma_i}$ and
let $\Gamma$ denote the $W$-graph of either the tree (two vertices, of
genus $g_1$ and $g_2$, respectively, with $k_1$ and $k_2$ tails,
respectively, and one separating edge) or of the loop (one vertex of
genus $g-1$ with $k$ tails and one edge) where the $i$th tail is
decorated with the group element $\gamma_i$.

Let $\hGamma$ denote the ``cut" version of the graph $\Gamma$.  Note
that the data given do not determine a decoration of the edge, so
$\Gamma$ and $\hGamma$ are really sums over all choices
$\Gamma_{\g}$ or $\hGamma_{\g}$ decorated with $\g\in G$
on the edge.

Using the notation of the composition axiom (\ref{ax:cutting}), we
have the following commutative diagram for each $\g$.
\begin{equation}\label{eq:q-hGamma}
\begin{diagram}
            &           &F_{\g}                      & \rTo^{pr_2}   &\W(\Gamma_{\g}) \\
            &\ldTo^{q_{\g}}    \\
\W(\hGamma_{\g})&&\dTo^{pr_1}     &             &  \dTo^{\st_{\Gamma_{\g}}}\\
            &\rdTo^{\st_{\hGamma_{\g}}}\\
            &           &\MM(\hGamma)           &       \rTo^{\hat{\rho}}     &\MM(\Gamma)\\
  \end{diagram}
  \end{equation}

And summing over all $\g\in G$, we have the following.
\begin{equation}\label{eq:summed-Gamma}
\begin{diagram}
 \bigcup_{\g\in G}\W(\Gamma_{\g}) & \rTo^{\itt} & \W_{g,k}(\bgamma)\\
\dTo_{\sum_{\g\in G}\st_{\Gamma_{\g}}}&  &\dTo_{\st}\\
\MM(\Gamma) & \rTo^i & \MM_{g,k}\\
  \end{diagram}
  \end{equation}
We have $\rho = i\circ \hat{\rho}$.  In the second diagram, note that the square is not Cartesian. In fact,
by Propositions~\ref{prp:ramif-tree} and~\ref{prp:ramif-loop}, it
fails to be Cartesian by a factor of $|\langle\g\rangle|$ on each term.

\begin{lm}\label{lm:pushpull}
For any $\alpha\in H_*(\W_{g,k}(\bga))$ we have the relation
\begin{equation}
i^*\st_*\alpha = \sum_{\g\in
  G}|\langle\g\rangle| (\st_{\Gamma_\g})_* \tilde{i}^*
\alpha.
\end{equation}
\end{lm}
\begin{crl}\label{crl:FundClasses}
The virtual fundamental classes pushed down to $\MM(\Ga)$ are related by 
the equality 
\begin{equation}\label{eq:sumgamma}
i^*\st_*\left[\W_{g,k}(\bgamma)\right]^{vir} = \sum_{\g\in
  G}|\langle\g\rangle| (\st_{\Gamma_\g})_* \tilde{i}^*
\left[\W_{g,k}(\bgamma)\right]^{vir}.
\end{equation}
\end{crl}
\begin{proof}[Proof of Lemma~\ref{lm:pushpull}]
The orbifold $\bigcup_{\g\in G}\W(\Gamma_{\g})$ is the inverse image $\st^{-1}(\MM(\Ga)$. 
We would like to be able to apply a push-pull/pull-push relation, but $\st$ is not transverse to $i$, so this will not work.  

Instead, we deform the map $\st$ in a small neighborhood of $\bigcup_{\g\in G}\W(\Gamma_{\g})$, and we deform in a normal direction to get a new map $\stt$ which is transverse to $i$ and so that the inverse image $\stt^{-1}(\MM(\Ga))$ lies in the normal bundle of $\itt$.  So we have the following diagram:
\begin{equation}
\begin{diagram}
&& \stt^{-1}(\MM(\Ga)) & \rTo_{\iit} & \W_{g,k}(\bga)
\\
&&\dTo_{\stt_{\Ga}} && \dTo_{\stt}\\
&&\MM(\Gamma) & \rTo_{i} & \MM_{g,k}.
\end{diagram}
\end{equation}

For any $\alpha\in H_*(\W_{g,k}(\bga))$ we have $\st_*\alpha = \stt_* \alpha$, since $\stt$ is a deformation of $\st$.  Now the standard push-pull/pull-push relation, which is a special case of the clean intersection formula \cite[Prop 3.3]{Qu}, says that  we have $$ \stt_{\Ga *}\iit^*\alpha = i^*\stt_* \alpha = i^* \st_*\alpha.$$
But since $\st^{-1}(\MM(\Ga))$ lies in the normal bundle of $\itt$, we can factor the map $\stt_{\Ga}$ as $$\stt_{\Ga} = \st_{\Ga} \circ pr,$$
where $pr$ is the projection of the normal bundle down to $\bigcup_{\g\in G}\W(\Gamma_{\g})$.  Moreover, since $\stt$ is a deformation of $\st$, we have that the pullbacks $\iit^*a$ and $pr^*\itt^*a$ are equal. 
The map $pr$ is finite when restricted to $\st^{-1}(\MM(\Ga))$, and for each $\g$, we denote by $\deg(pr_\g)$ its degree over the component $\W(\Gamma_{\g})$.  
Therefore, 
\begin{align*}
i^*\st_*\alpha &= \stt_{\Ga *} \iit^* \alpha\\
 & = \st_{\Ga *}pr_*pr^*\itt^*\alpha \\
 & = \sum_{\g \in G} \deg(pr_\g)\st_{\Ga_\g *}  \itt^* \alpha
\end{align*}
Now, it is easy to see that $\deg(pr_\g)$ is equal to the  number of non-isomorphic $W$-curves over a generic smooth curve that degenerate to a given generic nodal $W$-curve in $\W(\Gamma_{\g})$.  As described in Propositions~\ref{prp:ramif-tree} and \ref{prp:ramif-loop}, after accounting for automorphisms, this number is $|\langle \g \rangle|$. 
\end{proof}

Now, we prove Axioms~C2 and~C3.  To simplify computations, we choose a basis $B:=\{\mu_{\gamma,i}\}$ of
$\ch_W$ with each $\mu_{\gamma,i} \in \ch_{\gamma}$, and write all the
Casimir elements in terms of this basis.

In the case of Axiom C3 (the case that $\Gamma$ is a loop) we have
$$\begin{array}{ll}
\Lambda_{g-1,k+2}(\alpha_1,\dots,\alpha_k,\mu,\nu)\eta^{\mu\nu}
& = \sum_{\gamma\in G}\frac{|G|^{g-1}}{\deg(\st_{\hGamma_{\gamma}})}
PD\, (\st_{\hGamma_{\gamma}})_*\left(\left[\W(\hGamma_{\gamma})\right]^{vir}\cap \prod_{i=1}^k\alpha_i \cup \mu_{\gamma,i} \cup \nu_{\gamma,i}\right) \eta^{\mu_{\gamma,i}\nu_{\gamma,i}}
\\
& = \sum_{\gamma\in G}\frac{|G|^{g-1}}{\deg(\st_{\hGamma_{\gamma}})}
PD\, (\st_{\hGamma_{\gamma}})_*\left(\left\langle \left[\W(\hGamma_{\gamma})\right]^{vir}\right\rangle_{\pm}\cap \prod_{i=1}^k\alpha_i \right) \\
& = \sum_{\gamma\in G}\frac{|G|^{g-1}}{\deg(\st_{\hGamma_{\gamma}})\deg(q_{\gamma})}
PD\, (\st_{\hGamma_{\gamma}})_*\left((q_{\gamma})_*pr_2^*
\left[\W(\Gamma_{\gamma})\right]^{vir}\cap \prod_{i=1}^k\alpha_i \right) \\
& = \sum_{\gamma\in G}\frac{|G|^{g-1}}{\deg(pr_1)}
PD\, (pr_1)_*\left(pr_2^*
\left[\W(\Gamma_{\gamma})\right]^{vir}\cap \prod_{i=1}^k\alpha_i \right) \\
& = \sum_{\gamma\in G}\frac{|G|^{g-1}}{\deg(\st_{\Gamma_{\gamma}})}
PD\, \hat{\rho}^*(\st_{\Gamma_{\gamma}})_*\left(
\left[\W(\Gamma_{\gamma})\right]^{vir}\cap \prod_{i=1}^k\alpha_i \right) \\
& = \sum_{\gamma\in G}\frac{|G|^{g-1}}{|G|^{2g-3}|G/\lgr{}|}
PD\, \hat{\rho}^*(\st_{\Gamma_{\gamma}})_*\left(
\left[\W(\Gamma_{\gamma})\right]^{vir}\cap \prod_{i=1}^k\alpha_i \right) \\
& = \sum_{\gamma\in G}\frac{|G|^{g}|\lgr{}|}{|G|^{2g-1}}
PD\, \hat{\rho}^*(\st_{\Gamma_{\gamma}})_*\tilde{i}^*\left(
\left[\W_{g,k}(\bgamma)\right]^{vir}\cap \prod_{i=1}^k\alpha_i \right) \\
& = \frac{|G|^g}{\deg(\st)} PD\, \rho^*\st_*\left(
\left[\W_{g,k}(\bgamma)\right]^{vir}\cap \prod_{i=1}^k\alpha_i \right)\\
&= \rho^* \Lambda_{g,k}(\alpha_1,\dots,\alpha_k)
\end{array}$$
The second equality follows from the fact that the Casimir element in
cohomology is dual to the pairing in homology.  The sixth follows from
the explicit computation of $\deg(\st_{\Gamma})$ in
Proposition~\ref{prp:ramif-loop} and the seventh from the connected graphs axiom (Axiom~\ref{ax:ConnGraphs}).  The eighth equality follows from Equation~(\ref{eq:sumgamma}).

The case of Axiom~C2 is similar, but simpler, because there is only one
choice of decoration $\gamma$ for the edge of $\Gamma$.  In this case
we have

\begin{align*}
\Lambda_{g_1,k_1+1}(\alpha_{i_1},\dots,\alpha_{i_{k_1}},\mu)&\, \eta^{\mu\nu} \Lambda_{g_2,k_2+1}(\nu,\alpha_{i_{k_1+1}},\dots,\alpha_{i_{k_1+k_2}})\\
& = \frac{|G|^{g}}{\deg(\st_{\hGamma})}
PD\, (\st_{\hGamma})_*\left(\left[\W(\hGamma)\right]^{vir}\cap \prod_{i=1}^k\alpha_i \cup \mu \cup \nu\right) \eta^{\mu\nu} \\
& = \frac{|G|^{g}}{\deg(\st_{\hGamma})}
PD\, (\st_{\hGamma})_*\left(\left\langle \left[\W(\hGamma)\right]^{vir}\right\rangle_{\pm}\cap \prod_{i=1}^k\alpha_i \right) \\
& = \frac{|G|^{g}}{\deg(\st_{\hGamma})\deg(q)}
PD\, (\st_{\hGamma})_*\left((q)_*pr_2^*
\left[\W(\Gamma)\right]^{vir}\cap \prod_{i=1}^k\alpha_i \right) \\
& = \frac{|G|^{g}}{\deg(pr_1)}
PD\, (pr_1)_*\left(pr_2^*
\left[\W(\Gamma)\right]^{vir}\cap \prod_{i=1}^k\alpha_i \right) \\
& = \frac{|G|^{g}}{\deg(\st_{\Gamma})}
PD\, \rho^*(\st_{\Gamma})_*\left(
\left[\W(\Gamma)\right]^{vir}\cap \prod_{i=1}^k\alpha_i \right) \\
& = \frac{|G|^{g}}{\deg(\st)/|\lgr{}|}
PD\, \rho^*(\st_{\Gamma})_*\left(
\left[\W(\Gamma)\right]^{vir}\cap \prod_{i=1}^k\alpha_i \right) \\
& = \frac{|G|^{g}|\lgr{}|}{\deg(\st)}
PD\, \rho^*(\st_{\Gamma})_*\left(
\left[\W(\Gamma)\right]^{vir}\cap \prod_{i=1}^k\alpha_i \right) \\
&= \rho^* \Lambda_{g,k}(\alpha_1,\dots,\alpha_k).
\end{align*}

Axiom~C4a and Axiom~C4b follow immediately from the forgetting
tails axiom.
\end{proof}

\begin{df}\label{df:correlator}
   Define correlators
   $$\langle \tau_{l_1}(\alpha_1), \dots, \tau_{l_k}(\alpha_k)\rangle^{W,G}_g:=\int_{\left[\MM_{g,k}\right]}
   \Lambda^{W,G}_{g,k}(\alpha_1, \dots,
   \alpha_k)\prod_{i=1}^k {\psi}^{l_i}_i$$\glossary{<@$\langle \tau_{l_1}, \dots, \tau_{l_k}\rangle.$ & The correlators of the $W$-theory}
   \end{df}

\begin{df}
Let $\{\alpha_0,\dots,\alpha_s\}$ be a basis of the state space
$\ch_W$ such that $\alpha_0 =  \unit$, and let $\bt =
(\bt_0,\bt_1,\dots)$ with $\bt_l =
(t_l^{\alpha_0},t_l^{\alpha_1},\dots, t_l^{\alpha_s})$ be formal
variables.  Denote by $\Phi^{W,G}(\bt) \in \lambda^{-2}\C[[\bt,\lambda]]$
the (large phase space) potential \glossary{Phi @$\Phi^W(\bt)$ & The
  large phase-space potential of $W$-theory} of the theory:
$$\Phi^{W,G}(\bt):=\sum_{g\ge0} \Phi^{W,G}_g(\bt) := \sum_{g\ge 0} \lambda^{2g-2} \sum_{k}\frac{1}{k!} \sum_{l_1,\dots,l_k}\sum_{\alpha_1,\dots,\alpha_k} \langle \tau_{l_1}(\alpha_1) \cdots \tau_{l_k}(\alpha_k)\rangle^{W,G}_{g} t_{l_1}^{\alpha_1}\cdots t_{l_k}^{\alpha_k}.$$
\end{df}

Manin, in \cite[Thm III.4.3]{Ma-book} shows that a cohomological field theory in genus zero is equivalent to a formal Frobenius manifold.

\begin{crl}
The genus-zero theory defines a formal Frobenius manifold structure on $\Q[[\ch_{W,G}^*]]$ with pairing $\langle\, , \rangle^{W,G}$ and (large phase space) potential $\Phi_0^{W,G}(\bt)$.
\end{crl}

Three very important constraints are the \emph{string} and \emph{dilaton}  equations and the \emph{topological recursion relations} (See \cite[\S VI.5.2]{Ma-book} and \cite[\S5.2]{JKV1}.
\begin{thm}\label{thm:StringDilaton}
The potential $\Phi^{W,G}(\bt)$ satisfies analogues of the string and
dilaton equations and the topological recursion relations.
\end{thm}

\begin{proof}
Let $\vartheta:\MM_{g,k+1} \rTo \MM_{g,k}$ denote the universal curve,
and let $D_{i,k+1}$ denote the class of the image of the $i$th section
in $\MM_{g,k}$.

The dilaton and string equations for $\Phi^{W,G}$ follow directly from the
forgetting tails axiom and from fact that the gravitational
descendants ${\psi}_i$ satisfy $\vartheta^*({\psi}_i) =
{\psi}_i + D_{i,k+1}$.

The topological recursion relations hold because of the relation
$${\psi}_i = \sum_{\substack{T_+ \sqcup T_- = [k]\\ k,k-1 \in T^+\\ i\in T_-}} \delta_{0;T_+}$$
on $\MM_{0,k}$, where $\delta_{0;T_+}$ is the boundary divisor in $\MM_{0,k}$ corresponding to a graph with a single edge and one vertex labeled by tails in $T_+$.
\end{proof}
For more details about these equations in the $A_n$ case, see \cite[\S5.2]{JKV1}

\section{ADE-singularities and Mirror symmetry}\label{sec:mirror}

The construction of this paper corresponds to the A-model of the
Landau-Ginzburg model. A particular invariant from our theory is
the ring $\ch_{W, G}$.  The Milnor ring, or local algebra, $\milnor_W$ of a singularity
can be considered as the B-model. One outstanding conjecture of
Witten is the self-mirror phenomenon for ADE-singularities. This conjecture states that for any simple (i.e., ADE) singularity $W$, the ring $\ch_{W,\genj}$ is isomorphic, as a Frobenius algebra, to the Milnor ring $\milnor_W$ of the same singularity.

This is the main topic of this section. More precisely, we prove the following theorem, which resolves the conjecture and serves as the
 first step toward the proof of the integrable hierarchy theorems
 in the next section.

\begin{thm}\label{thm:ADEselfdual}[Theorem\ref{thm:mirror}]  { } \
\begin{enumerate}
    \item Except for $D_n$ with $n$ odd, the ring $\ch_{W,\genj}$ of any simple (ADE) singularity $W$ with symmetry group
    $\genj $ is isomorphic, as a Frobenius algebra, to the Milnor ring $\milnor_W$ of the same singularity.

    \item The
  ring $\ch_{D_n,G_{D_n}}$  of $D_n$ with the maximal diagonal
    symmetry group $G_{D_n}$ is isomorphic, as a Frobenius algebra,  to the Milnor ring $\milnor_{x^{n-1}y+y^2} \cong \milnor_{A_{2n-3}}$.
    \item The
  ring $\ch_{W,G_W}$ of $W=x^{n-1}y+y^2$ ($n\geq 4$)
    with the maximal diagonal symmetry group is isomorphic, as a Frobenius algebra,  to the
    Milnor ring $\milnor_{D_n}$ of $D_n$.
    \end{enumerate}
\end{thm}

 Note that the self-mirror conjecture is not quite correct.  In particular, in the case of $D_n$ for $n$ odd, the maximal symmetry group is generated by $J$, but the ring $\ch_{W,G_W} = \ch_{W,\genj}$ is not isomorphic to $\milnor_{D_n}$.  Instead it is isomorphic to the the Milnor ring of the singularity $\tW:=x^{n-1}y+y^2$, and conversely, the ring $\ch_{W',G_{W'}}$ is isomorphic to the Milnor ring $\milnor_{D_n}$, so, in fact, the mirror of $D_n$ is $W'=x^{n-1}y+y^2$.

  This is a special case of the construction of Berglund and H\"ubsch \cite{BH} for invertible singularities.
   Specifically, consider a singularity $W$ of the form
$$W=\sum_{i=1}^N W_j \quad \text{ with }\quad
W_j = \prod_{l=1}^N x_\ell^{b_{\ell,j}},$$ and with $b_{\ell,j} \in \Z^{\ge 0}$.  As we did in the proof of Lemma~\ref{lm:group}, we form the $N\times N$ matrix
$
B := (b_{\ell j}).
$
Berglund and H\"ubsch conjectured that the mirror partner to $W$ should be the singularity corresponding to $B^T$, that is
$$W^T:=\sum_{\ell=1}^N W^T_\ell, \quad \text{  where }\quad
W^T_\ell = \prod_{j=1}^N x_j^{b_{\ell,j}}.$$

Using this construction, we find that the mirror partner to $D_{n} = x^{n-1} +xy^2$ should be the singularity $D_{n}^T=x^{n-1}y+y^2$.   This singularity is isomorphic to $A_{2n-3}$, so the Milnor ring of $W$ is isomorphic to the Milnor ring of $A_{2n-3}$.  But this isomorphism of singularities does not give an isomorphism of A-model theories.  Indeed, Theorem~\ref{thm:ADEselfdual} shows that the ring $\ch_{D^T_{n},G_{D^T_{n}}}$ of ${D^T_{n}}$ is not isomorphic to the
  ring $\ch_{A_{2n-3},G_{A_{2n-3}}}$, but rather it it isomorphic to $\milnor_{D_{n}}$.

The Berglund and H\"ubsch construction also explains the self-duality of $A_n$ and $E_{6,7,8}$. In addition, their elegant construction
opens a door to the further development of the subject of Landau-Ginzburg mirror symmetry. Since the initial post of this article in 2007,
much progress on Landau-Ginzburg mirror symmetry has been made  by Krawitz and his collaborators \cite{Krawitz, Pr}.

We note that Kaufmann \cite{Ka1, Ka2,Ka3} has made a computation for a different, algebraic construction of an ``orbifolded Landau-Ginzburg model" which gives mirror symmetry results that match the results of Theorem~\ref{thm:ADEselfdual}. In particular, in his theory, just as in ours, the $D_n$ case for $n$ odd is also not self-dual, but rather is mirror dual to $D^T_{n}$.

\subsection{Relation between $\milnor_W$ and $H^N(\C^N, W^{\infty}, \C)$}\label{subsec:WallsThm}
As we mentioned earlier, the Milnor ring $\milnor_W$ represents a B-model structure. In order to obtain the
correct action, we consider $\milnor_W \omega$, where $\omega=dx_1\wedge\cdots\wedge dx_N$. Here an element of $\milnor_W\omega$ is
 of the form $\phi \omega$, where $\phi\in\milnor_W$ and $\gamma\in G_W$ acts on both $\phi$ and $\omega$. The A-model analogy is the
relative cohomology groups $H^N(\C^N, W^{\infty}, \C)$. It was an old theorem of Wall \cite{Wa1, Wa2}
 that they are isomorphic as $G_W$-spaces. Wall's theorem could almost be viewed as a sort of mirror symmetry theorem itself.

 An
``honest" mirror symmetry theorem should exchange the A-model for one singularity with the B-model for a \emph{different} singularity.  However, it is technically convenient for us to use Wall's isomorphism to label the class of
  $H^N(\C^N, W^{\infty}, \C)$. For the A-model state space, we need to consider
  $H^N(\C^N, W^{\infty}, \C)^{\langle J \rangle}$ with the intersection pairing. It is well known that Wall's isomorphism can
  be improved to show that
  $$\left(H^N(\C^N, W^{\infty}, \C)^{\langle J \rangle}, \langle \ , \ \rangle\right) \cong \left((\milnor_W \omega)^{\langle J \rangle}, \Res\right)$$
  (see a nice treatment in \cite{Ce}). It is clear that the above isomorphism also holds for the invariants of any admissible
  group $G$. With the above isomorphism, we have the identifications:
\begin{equation}\label{eq:StSpIso} \ch_{W,G}=\bigoplus_{\gamma \in G} \left(H^{mid}(\C^{N_{\gamma}},
W_{\gamma}^{\infty}, \Q) \right)^G \cong \bigoplus_{\gamma\in
G}\left( \milnor_{W_\gamma}\omega_\gamma\right)^G,
\end{equation}
where $\omega_\ga$ is the restriction of the volume form $\omega$ to the fixed locus $\fix{\ga}$.  The space $\bigoplus_{\ga} \left(\milnor_{W_\ga}\omega_\ga\right)^G$ arises in the orbifolded Landau-Ginzburg models studied by Intriligator-Vafa and Kaufmann in \cite{IV,Ka1,Ka2,Ka3}.

For computational purposes, it is usually easier to work with the sums of Milnor rings so we will use the identification~(\ref{eq:StSpIso}) for the remainder of the paper.      However, we would like to emphasize that while $\milnor_W$
      has a natural ring structure, $H^N(\C^N, W^{\infty}, \C)$
      does not have any natural ring structure.  Moreover, the ring structure induced on the state space is \emph{not} the same as the one induced by the Milnor rings via the isomorphism~(\ref{eq:StSpIso}).
       Furthermore, $\milnor_W$ has an internal
      grading, while the degree of $H^N(\C^N, W^{\infty}, \C)$ is just $N$. Hence,
      they are very different  objects, and readers should not be confused by their similarity.

     Before we start an explicit computation, we make several additional remarks.
\begin{rem}
   One point of confusion is the notation of degree in singularity
   theory versus that of Gromov-Witten theory. Throughout the rest
   of paper, we will use $\deg_{\C}$ to denote the degree in
   singularity theory (i.e., the degree of the monomial) and $\deg_W$ to
   denote its degree as a cohomology class in Gromov-Witten or
   quantum singularity theory. We have $$\deg_W=2\deg_{\C}.$$
\end{rem}

\begin{rem}\label{rem:residuepairing}
The local algebra, or Milnor ring, $\milnor_W$  carries a natural
non-degenerate pairing defined by
$$\langle  f,g\rangle=\Res_{x=0}\frac{fg\, dx_1\wedge\cdots\wedge dx_N}{ \frac{\partial
W}{\partial x_1} \cdots \frac{\partial W}{\partial x_N}}.$$ The
pairing can be also understood as follows.
The residue $\Res(f):=\Res_{x=0}\frac{f\, dx_1\wedge\cdots\wedge dx_N}{ \frac{\partial
W}{\partial x_1} \cdots \frac{\partial W}{\partial x_N}}$ has the following properties
    \begin{description}
    \item[(1)] $\Res(f)=0$ if $\deg_{\C}(f)< \chat_W$.
    \item[(2)] $\Res\left(\frac{\partial^2 W}{\partial x_i\partial
x_j}\right)=\mu$, where $\mu := \dim_\C(\milnor_W)$ is the Milnor number.
    \end{description}
Modulo the Jacobian ideal, any polynomial $f$ can be uniquely expressed as    $f=C\left(\frac{\partial^2 W}{\partial x_i\partial
x_j}\right)+ f'$,
    with $\deg_{\C}(f')<\chat_W$. This implies that $$\Res(f)=C\mu.$$
\end{rem}

\begin{rem}
For any $G\le \aut(W)$, the action of the group $G$ on the line
bundles of the $W$-structure and on relative homology is
\emph{inverse} to the action on sheaves of sections, on relative
cohomology, on the local ring, and on germs of differential forms.
For instance, the element we have called $J$ acts on homology and on
the line bundles of the $W$-structure as $(\exp(2 \pi i q_1), \dots
,\exp(2 \pi i q_N))$, but it acts on $\milnor_W$ and on
$\milnor_W \omega$ as
$$J\cdot x_1^{m_1}\cdots x_N^{m_N}  = e^{-2 \pi i \sum_i m_i q_i}x_1^{m_1}\cdots x_N^{m_N}$$ and
$$J\cdot x_1^{m_1}\cdots x_N^{m_N} dx_1\wedge \dots \wedge dx_N = e^{-2 \pi i \sum_i (m_i+1) q_i}x_1^{m_1}\cdots x_N^{m_N} dx_1\wedge \dots \wedge dx_N .$$
\end{rem}

\subsection{Self-mirror cases}

\subsubsection{The singularity {$A_n$}}\label{sec:AnFA}

The maximal diagonal symmetry group of $A_n = x^{n+1}$ is precisely the group $\genj $.  The $\genj $-invariants of the theory in the case of $A_n$ agree with the theory of $(n+1)$-spin curves in \cite{JKV1}.  In that paper it is proved that the associated Frobenius algebra is isomorphic to the $A_n$ Milnor ring (local algebra), and the Frobenius manifold is isomorphic to the Saito Frobenius manifold for $A_n$.

\subsubsection{The exceptional singularity $E_{7}$ }\label{sec:EsevenFA}

Consider now the case of $E_7= x^3 + xy^3$.
We have $$q_x = 1/3, \text{ and } q_y = 2/9.$$
By Equation~(\ref{eq:defchat}) we get $$\chat_{E_7} =8/9.$$  Furthermore, if $\xi = \exp(2 \pi i/9)$, then $J$ acts by $(\xi^3,\xi^2)$, and 
$$\Theta_x^J = 1/3 \qquad \Theta_y^J = 2/9.$$  
It is easy to check that  the maximal symmetry group is generated by $J$:
$$G_{E_7} = \genj  \cong \Z/9\Z.$$

Denote
$$\bone_{0}:=dx\!\wedge\! dy \in \Gss{0},$$
$$\bone_{{k}}:=dx\in \Gss{k} \text{ for $k = 3,6$},$$ and
$$\bone_{{k}}:= 1 \in \Gss{k} \text{ for $3 \nmid k$.}$$
We also denote the element $\unit:=\bone_1.$

Using this notation, the $G_{E_7}$-space
$ \bigoplus_{{k} \in \Z/9\Z} \Gss{k}$ can be described as follows:
\begin{equation}
\Gss{k} = \begin{cases}
                    E_7 = \spanl \bone_{0}, x^1  \bone_{0},  x^{2}  \bone_{0}, y  \bone_{0}, y^2 \bone_{0}, xy  \bone_{0}, x^2y \bone_{0}\spanr & \text{ if  $k=0$}\\
              A_2 = \spanl \bone_{{k}}, x  \bone_{{k}}\spanr & \text{ if $k\equiv 3,6 \pmod 9$} \\
                    A_1 = \spanl \bone_{{k}} \spanr & \text{ if  $3\nmid k $}.
                                \end{cases}
\end{equation}

The $G_{E_7}$-invariant elements of this space form the state space of the $E_7$ theory  $$\ch_{E_7} = \spanl
y^2 \bone_{0}, \unit,\bone_{{2}},\bone_{{4}},\bone_{{5}},\bone_{{7}},\bone_{{8}}\spanr.$$

We now compute the genus-zero, three-point correlators for the $G_{E_7}$-invariant terms of the theory.  First, the degree shift $$\iota_{J^k} = \sum_{i=1}^N (\Theta_i^{J^k} - q_i)$$ and the $W$-degree $$\deg_W(x^iy^j\bone_k) = \deg(\x^iy^j\bone_k)+2\iota_{J^k} = N_{J^k}+2\iota_{J^k}$$ depend only on $k$.  For example, we have 
$$\iota_{J^2} =  (\Theta_x^{J^2} - q_x) + (\Theta_y^{J^2} - q_y) = (2/3 -1/3)+(4/9-2/9) = 5/9,$$
and 
$$\deg_W(\bone_2) =   \deg(\bone_2)+2\iota_{J^2} = 0 + 10/9.$$
The complete set of numbers $\iota$ and $\deg_W$ are given by the following table:

\medskip

\centerline{\begin{tabular}{r||c|c|c|c|c|c|c|c|c}
k   & $0 $&$ {1} $&${2} $&${3} $&${4} $&${5} $&${6} $&${7} $&${8}$\\
\hline
$\iota_{J^{k}}$ & $ -5/9$&$     0 $&$  5/9 $&$1/9 $&$  6/9$&$   2/9$&$  -2/9$&$   3/9$&$   8/9 $\\
\hline
$\deg_W(x^iy^je_k)$& $ 8/9$&$     0 $&$  10/9 $&$11/9 $&$  12/9$&$   4/9$&$  5/9$&$   6/9$&$   16/9 $
\end{tabular}}

\medskip

For each genus-zero, three-point correlator $\langle a\bone_{k_1},b\bone_{k_2},c\bone_{k_3}\rangle_{0}^{E_7}$, we have $g=0$, $k=3$ and we compute from Equation~(\ref{eq:D}) that
$$D = -\ind(\LL_x)-\ind(\LL_y) =\chat_W(0-1) +\sum_{j=1}^3 \iota_{J^{k_j}} = -8/9 +\sum_{j=1}^3 \iota_{J^{k_j}} $$
The dimension axiom (Equation~(\ref{eq:dimension})) states that the the correlator will vanish unless  
$$\dim_{\R}(\MM_{0,3}) = -2D -\sum_{j=1}^3 {N_{J^{k_j}}}.$$
That means the correlator will vanish unless
$$ 0 = -2\chat_{E_7}  +2\sum_{j=1}^3 \iota_{J^{k_j}} + \sum_{j=1}^3 {N_{J^{k_j}}} = -2\chat_W + \sum_{j=1}^3 \deg_W(\bone_{k_j})$$

A straightforward computation shows
this only occurs for the following correlators :
$$\langle y^2\bone_{0} , y^2\bone_{0} , \unit  \rangle_0^{E_7} , \langle y^2\bone_{0} , \bone_{{5}}, \bone_{{5}}\rangle_0^{E_7} , \langle \unit , \unit , \bone_{{8}}\rangle_0^{E_7} , \langle \unit , \bone_{{2}} , \bone_{{7}}\rangle_0^{E_7} , \langle \unit , \bone_{{4}}, \bone_{{5}}\rangle_0^{E_7} , \langle \bone_{{5}},\bone_{{ 7}}, \bone_{{7}}\rangle_0^{E_7}. $$

Now we compute when the line bundles $\LL_x$ and $\LL_y$, defining the $E_7$-structure are concave.  Since we are in genus zero, this occurs precisely when the degree of the desingularization of each line bundle (see Equation~(\ref{eq:sel-rule})) is negative:
\begin{equation}
0>\deg(|\LL_x|) = \left(q_x(2g - 2 + k)
-\sum^k_{l=1}\Theta_x^{\gamma_l} \right) = 1/3 - \sum_{l=1}^3 \Theta_x^{J^{k_l}}
\end{equation}
and
\begin{equation}\label{eq:sel-rule}
0>\deg(|\LL_y|) = \left(q_y(2g - 2 + k)
-\sum^k_{l=1}\Theta_y^{\gamma_l} \right) = 2/9 - \sum_{l=1}^3 \Theta_y^{J^{k_l}}
\end{equation}
This occurs precisely for the correlators
$$ \langle \unit , \unit , \bone_{{8}}\rangle_0^{E_7} , \langle \unit , \bone_{{2}} , \bone_{{7}}\rangle_0^{E_7} , \langle \unit , \bone_{{4}}, \bone_{{5}}\rangle_0^{E_7} , \langle \bone_{{5}},\bone_{{ 7}}, \bone_{{7}}\rangle_0^{E_7}. $$
For these concave cases the virtual cycle must be Poincar\'e dual to the top (zeroth) Chern class of the bundle $R^1\pi_* (\LL_1\oplus\LL_2) = 0$, which is $1$.  Thus these correlators are all $1$.

The correlator $\langle y^2\bone_{0} , y^2\bone_{0} \unit  \rangle_0^{E_7}$, is just the residue pairing of the element $y^2$ with itself in the $J^0$-sector $\Gss{0} = \milnor_{E_7}$.  The  Hessian $h:=\frac{\partial^2 W}{\partial x_i \partial x_j}$ of $W$ is $36x^2y - 9y^4 = -21y^4$ in $\milnor_W$, and by Remark~\ref{rem:residuepairing} we have  
$\langle 1 \bone_0 , h \bone_0 \rangle^{E_7} =  \mu_{E_7} = 7$, 
so  $$\langle y^2\bone_{0} , y^2\bone_{0}, \unit  \rangle^{E_7}  = \langle 1 \bone_0 , y^4 \bone_0 \rangle^{E_7}  =  \langle 1 \bone_0 , -h/21 \bone_0 \rangle^{E_7} = -1/3.$$

Finally, we will compute the correlator $\langle y^2\bone_{0} , \bone_{{5}}, \bone_{{5}}\rangle_0^{E_7}$ by using the Composition Law (Axiom~\ref{ax:cutting}).
The cycle
$\left[\W_{0,4}(E_7; J^5,J^5,J^5,J^5)\right]^{vir}$ corresponds to a cycle on $\W_{0,4}(E_7;J^5,J^5,J^5,J^5)$ of (real) dimension $6g-6+2k-2D= 2 = \dim_{\R}\W_{0,4}(E_7;J^5,J^5,J^5,J^5) $, and thus it is just a constant times the fundamental cycle.

In this case we can compute that the line bundles $|\LL_x|$ and $|\LL_y|$ have degrees $-2$ and $0$, respectively, and thus for each fiber (isomorphic to $\CP^1$) of the universal curve $\cC$ over $\W_{0,4}(E_7; J^5,J^5,J^5,J^5)$ we have $H^0(\CP^1,|\LL_x| \oplus |\LL_y|) = 0 \oplus \C$, and $H^1(\CP^1,|\LL_x| \oplus |\LL_y|) = \C \oplus 0$.  The Witten map from $H^0$ to $H^1$ is $(3\bar{x}^{2}+\bar{y}^3, 2\bar{x}\bar{y})$.  This map has degree $-3$,  so by the Index-Zero Axiom (Axiom~\ref{ax:wittenmap}), the cycle
$\left[\W_{0,4}(E_7; J^5,J^5,J^5,J^5)\right]^{vir}$  is $-3$ times the fundamental cycle.  Pushing down to the moduli of pointed curves (see Equation~(\ref{eq:defLambda})) gives $\Lambda_{0,4}^{E_7} (\bone_5,\bone_5,\bone_5,\bone_5) = -3$ and the pullback along the gluing map $\rho$ gives $\rho^*\Lambda^{E_7}_{0,4}(\bone_5,\bone_5,\bone_5,\bone_5) = -3$.

By the Composition Axiom, we have $$-3 = \sum_{i,j}\Lambda_{0,3}^{E_7}(\bone_{{5}},\bone_{{5}},\alpha_i)\eta^{\alpha_i \beta_j} \Lambda_{0,3}^{E_7}(\beta_j,\bone_{{5}},\bone_{{5}}).$$
But the only non-zero three-point class of the form $\Lambda_{0,3}^{E_7}( \bone_{{5}},\bone_{{5}},\alpha_i)$ is $\Lambda_{0,3}^{E_7}( \bone_{{5}},\bone_{{5}},y^2\bone_{0})$.  Thus we have
$$-3 = -3 \left(\Lambda_{0,3}^{E_7}( \bone_{{5}},\bone_{{5}},y^2\bone_{0})\right)^2 ,$$
and so 
\begin{equation}\label{eq:EsevenTPC}
\corf{E_7}{\bone_{{5}},\bone_{{5}},y^2\bone_{0}} = \int_{\MM_{0,3}} \Lambda_{0,3}^{E_7}( \bone_{{5}},\bone_{{5}},y^2\bone_{0}) = \pm 1.
\end{equation}

Now the fact that our pairing matches $\ch_\gamma$ with $\ch_{\gamma^{-1}}$ means that it pairs $\ch_{J^k}$ with $\ch_{J^{9-k}}$, and if $k\neq 0$ then the sectors $(\ch_{J^k})^{J} \cong   (\ch_{J^{9-k}})^{J}$ are one-dimensional, spanned by $\bone_k$ and the pairing gives  $\langle \bone_{{k}}, \bone_{{9-k}}\rangle^{E_7} = 1$.  We can use these correlators as the structure constants for an algebra on the invariant state space.  If we define a map $\phi_\alpha:\C[X,Y] \rTo \ch_{E_7}$ by
$$X \mapsto \alpha^3\bone_{{7}} \dsand Y\mapsto \alpha^2\bone_{{5}}$$ for any $\alpha\in \C^*$,
then we can make $\phi$ into a surjective homomorphism as follows:
\begin{eqnarray*}
1 \mapsto \unit = \bone_1 & X^2\mapsto \alpha^6\bone_{{4}} &
X Y \mapsto \alpha^5\bone_{{2}}\\
 X^2 Y \mapsto \alpha^8 \bone_{{8}} &
Y^2 \mapsto \mp 3 \alpha^4 y^2\bone_{0}
\end{eqnarray*}
Moreover, we have the relations $$\phi(X)\star \phi(Y)^2 = 0$$ and
\begin{equation}\begin{split}
\phi(Y)^3 &= \phi(Y) \star (\mp3 \alpha^4 y^2\bone_{0})  = \mp3\alpha^6\sum_{\alpha,\beta} \langle \bone_{{5}}, y^2\bone_{0},\alpha\rangle_{0}^{E_7} \eta^{\alpha \beta} \beta\\
&=-3 \alpha^6\bone_{{4}} = -3\phi(X)^2.
\end{split}
\end{equation}
So the kernel of $\phi$ contains $XY^2$ and $Y^3+3X^2$, but
$\milnor_{E_7} = \C[X,Y]/(XY^2,Y^3+3X^2)$ has the same dimension as $\ch_{E_7}$; therefore
$$ \milnor_{E_7}= \C[X,Y]/ (XY^2, 3X^2+Y^3) \rTo^{\phi_{\alpha}}   (\ch_{E_7}^G, \star) $$ is an isomorphism of graded algebras for any choice of $\alpha \in \C^*$.

We wish to choose $\alpha$ so that the isomorphism $\phi_\alpha$ also preserves the pairing.  The pairing for $\milnor_{E_7}$ has
$$ \langle 1,X^2Y\rangle_{\milnor_{E_7}} = \frac{1}9 \dsand
\langle Y^2,Y^2\rangle_{\milnor_{E_7}} = -\frac{1}3,
$$
whereas for $\ch_{E_7}$ the pairing is given by
$$
\langle 1,\bone_8\rangle_{\ch_{E_7}} = 1 \dsand \langle \mp y^2\bone_0,\mp y^2\bone_0\rangle_{\ch_{E_7}} = -3.
$$
This shows that the pairings differ by a constant factor of $9$, and since $\phi(X^2Y)$ and $\phi(Y^4)$ both have degree $8$ in $\alpha$, choosing $\alpha^8 = 1/9$ makes $\phi$ into an isomorphism of graded Frobenius algebras
$$ \milnor_{E_7} \cong   (\ch_{E_7}^G, \star).$$

\subsubsection{The exceptional singularities $E_6$ and $E_8$}

Our ring $\ch_{W,G_W}$ for both of the exceptional singularities $E_6 = x^3+y^4$ and $E_8 = x^3 + y^5$ with maximal symmetry group $G_W$ can be computed easily using the Sums of Singularities  Axiom (Axiom~\ref{ax:sums}).  In this case we have
\begin{eqnarray}
\ch_{E_6,G_{E_6}} &\cong \ch_{A_2,G_{A_2}} \otimes  \ch_{A_3,G_{A_3}} \cong \milnor_{A_2} \otimes \milnor_{A_3} \cong \milnor_{E_6}\\
\ch_{E_8,G_{E_8}} &\cong \ch_{A_2,G_{A_2}} \otimes  \ch_{A_4,G_{A_4}} \cong \milnor_{A_2} \otimes \milnor_{A_4} \cong \milnor_{E_8},
\end{eqnarray}
where the second isomorphism of each row follows from the $A_n$ case.  Note that in both cases we have $\genj=G_{W}$.

Later, when we compute the four-point correlators, it will be useful to have these isomorphisms described explicitly.

\paragraph{\underline{Explicit Isomorphism for $E_6$}}
Define $E_6:= x^3 + y^4$.  The invariants are generated by the elements
$\bone_1,\bone_2,\bone_5,\bone_7,\bone_{10},\bone_{11},$ where  $\bone_{i} := 1 \in  \Gss{i}$.

Computations similar to those done above show that
the isomorphism of graded Frobenius algebras
$\milnor_{E_6} \rTo \ch_{E_6,G_{E_6}}$ is given by
$$Y\mapsto \alpha^3\bone_5 \quad\text{ and }\quad X\mapsto\alpha^4\bone_{10},$$ with $\alpha^{10} = 1/12$.

\paragraph{\underline{Explicit Isomorphism for $E_8$}}
Define $E_8:= x^3 + y^5$.  The invariants are generated by the elements
$\bone_1,\bone_2,\bone_4,\bone_7,\bone_{8},\bone_{11},\bone_{13},\bone_{14}$ where  $\bone_{i} := 1 \in  \Gss{i}$.

Again, computations similar to those done above show that
the isomorphism of graded Frobenius algebras
$\milnor_{E_8} \rTo \ch_{E_8,G_{E_8}}$ is given by
$$Y\mapsto \alpha^3\bone_{7} \quad\text{ and }\quad X\mapsto\alpha^5\bone_{11},$$ with $\alpha^{14} = 1/15$.

\subsubsection{The singularity $D_{n+1}$ with $n$ odd and symmetry group $\langle J\rangle$}\label{sec:DnOddFA}

Consider the case of $D_{n+1}$ with $W = x^{n} + xy^2$ and with $n$ odd. The weights are $q_x = 1/n$ and $q_y = (n-1)/2n$ and the central charge is $\chat_{D_{n+1}} = (n-1)/n$.  The exponential grading operator $J$ is
$$J=(\xi^{2}, \xi^{n-1})\quad \text{ where $\xi=\exp({{2\pi i}/{2n}})$}.$$ And $J$ has order $n$ in the group  $G_{D_{n+1}} = \langle
(\xi^{2},\xi)\rangle \cong \Z/2n\Z$.

As described in Section~\ref{sec:changeGroup}, we may restrict to
the sectors that come from the subgroup $\genj $ by restricting the
virtual cycle for $D_{n+1}$ to the locus corresponding to the moduli
space for  $W'$-curves, with $W' := x^n+xy^2
+ x^{(n+1)/2}y$.  For the rest of this example we assume this
restriction has been made.  To simplify computations later, we find
it easier to take $a\in (0,n]$ instead of the more traditional range
of $[0,n)$.

Denote
$$\bone_{n}:=dx\!\wedge\! dy \in \Gss{n} = \Gss{0},$$
$$\bone_{{a}}:= 1 \in \Gss{a} \text{ for $a\neq n$,}$$
so that the $G_{D_{n+1}}$-space
$\bigoplus_{{k} \in \Z/n\Z} \Gss{k}$ can be described as
\begin{equation}
\Gss{k} = \begin{cases}
                    D_{n+1} = \langle \bone_{n}, x^1  \bone_{n},  x^{2}  \bone_{n}, \dots, x^{n-1} \bone_{n}, y \bone_{n}
 \rangle & \text{ if  $k=n$}\\
                    A_1 = \langle \bone_{{k}} \rangle & \text{ if  $k\nequiv 0\pmod{n} .$}
                                \end{cases}
\end{equation}

The $\genj $-invariant elements form our state space $$\ch_{D_{n+1}} =
\langle
x^{(n-1)/2}\bone_{n}, y\bone_{n},
\bone_1,\dots,\bone_{{(n-1)}}
\rangle.$$

To prove that $(\ch_{D_{n+1}},\star) \cong \milnor_{D_{n+1}}$, we will choose constants $\alpha,\beta \in \C$ so that the ring homomorphism $$\phi:\C[X,Y] \rTo \ch_{D_{n+1}},$$
defined by $X\mapsto \bone_3$ and $Y \mapsto \alpha (x^{\frac{n-1}{2}}\bone_n) + \beta (y\bone_n)$, induces an isomorphism from $\milnor_{D_{n+1}}$ to $\ch_{D_{n+1},{\langle J\rangle}}$.

To determine properties of the homomorphism, we must better understand the genus-zero, three-point correlators for the $\genj $-invariant terms of the theory.

The degree $\deg_W(x^iy^j\bone_{a})$ is determined only by $a$ and is given as follows:
$$\deg_W(x^iy^j\bone_{a})=\begin{cases} \frac{a-1}{n} & \text{if $a$ is odd and $a\in(0,n]$}\\
\frac{n+a-1}{n} &\text{if $a$ is even and $a\in (0,n)$}.
\end{cases}
$$

For the genus-zero, three-point correlators, denote the  relevant sectors by $J^{a_i}$ for $i \in \{1,2,3\}$.   Using the dimension axiom, we see that the virtual cycle vanishes unless $$2\chat_{D_{n+1}} = \sum_l \deg_W(\bone_{a_l}).$$
This occurs precisely when
\begin{equation}\label{eq:DnSelRule}
\sum_i a_i = 2n+1 - n E,
\end{equation}
 where $E$ denotes the number of $a_i$ which are even.   Since $0< a_i \le n$ for all $i$, we have $\sum a_i \ge 3$, so $0\le E\le 1$.

Using Equation~(\ref{eq:sel-rule}) for the degree of the bundles $|\LL_x|$ and $|\LL_y|$, we have the following two cases:
\begin{enumerate}
\item\label{caseI} If $E=1$, then $\deg(|\LL_x|) = \deg(|\LL_y|) = -1$.  In this case the concavity axiom shows that the correlator is $1$.
\item\label{caseII} If $E=0$, then at most two of the $a_i$ can be $n$.  There are three cases:
\begin{enumerate}
\item\label{caseA} If $E=0$ and none of the $a_i$ is $n$, then
$\deg(|\LL_x|) =-2$ and $ \deg(|\LL_y|) = 0$.
\item\label{caseB} If $E=0$ and exactly one $a_i$ is $n$, then
$\deg(|\LL_x|) = -1$ and $ \deg(|\LL_y|) = 0$.
\item \label{caseC}If $E=0$ and exactly  two of the $a_i$ are $n$, then
$\deg(|\LL_x|) = \deg(|\LL_y|) = 0$.
\end{enumerate}
\end{enumerate}

For Case~\ref{caseII}, first note that all correlators of the form
$$\langle v, v', \unit \rangle_{0,3}^{D_{n+1}}$$ for $v, v' \in \ch$ are simply the pairing of $v$ with $v'$ in $\ch$.  In particular,
$$\langle x^{(n-1)/2}\bone_{n}, x^{(n-1)/2}\bone_{n}, \unit\rangle_{0,3}^{D_{n+1}} =  \langle x^{(n-1)/2}\bone_{n}, x^{(n-1)/2}\bone_{n}\rangle  = \frac{1}{2n}$$
$$\langle y\bone_{n}, y\bone_{n}, \unit\rangle_{0,3}^{D_{n+1}} =  \langle y\bone_{n}, y\bone_{n}\rangle  = \frac{-1}{2},$$
and
$$\langle x^{(n-1)/2}\bone_{n}, y\bone_{n}, \unit\rangle_{0,3}^{D_{n+1}} =  \langle x^{(n-1)/2}\bone_{n}, y\bone_{n}\rangle  = 0.$$

For Case~\ref{caseA} the line bundles $|\LL_x|$ and $|\LL_y|$ have degrees $-2$ and $0$, respectively, and thus $H^0(\CP^1,|\LL_x| \oplus |\LL_y|) = 0 \oplus \C$, and $H^1(\CP^1,|\LL_x| \oplus |\LL_y|) = \C \oplus 0$, and the Witten map from $H^0$ to $H^1$ is $(n\bar{x}^{n-1}+\bar{y}^2, 2\bar{x}\bar{y})$.  This map has degree $-2$,  so, as in previous arguments, the Index-Zero Axiom (Axiom~\ref{ax:wittenmap}) shows that the correlator, is $-2$.

\paragraph{\underline{Case of $n>3$}}

If we assume that $n>3$, and letting $\mu$ and $\nu$ range through the basis $\{x^{(n-1)/2}\bone_{n}, y\bone_{n},\bone_1,\dots,\bone_{{(n-1)}} \}$, we have
\begin{align*}
\bone_3\star\bone_3
&= \sum_{\mu,\nu} \langle \bone_3,\bone_3,\mu\rangle \eta^{\mu\nu}\nu\\
&= \sum_{\nu}\langle \bone_3,\bone_3,\bone_{n-5}\rangle \eta^{\bone_{n-5}\nu}\nu\\
&= \langle \bone_3,\bone_3,\bone_{n-5}\rangle \bone_5 =\bone_5.
\end{align*}
Similar computations show that for $l < (n-1)/2$ we have
$$\bone_3^l = \bone_{2l+1}.$$
In the case of $\bone_3^{(n-1)/2}$ we have
$$
\bone_3^{(n-1)/2} = \bone_3\star\bone_3^{(n-3)/2}
= \langle\bone_3,\bone_{(n-2)},x^{(n-1)/2}\bone_{n}\rangle 2n x^{(n-1)/2}\bone_{n} +
\langle\bone_3,\bone_{(n-2)},y\bone_{n}\rangle (-2) y\bone_{n}.
$$
To simplify notation we denote $r:=\langle\bone_3,\bone_{(n-2)},x^{(n-1)/2}\bone_{n}\rangle$ and $s:=\langle\bone_3,\bone_{(n-2)},y\bone_{n}\rangle$, so that
$$\bone_3^{(n-1)/2} =  (2n r x^{(n-1)/2}\bone_{n} -2 s y\bone_{n}).$$

Note that a computation like the one done above for case~\ref{caseA} shows that the restriction of the virtual cycle $\left[\W_{0,4,D_{n+1}}(J^{3},J^{(n-2)},J^{3},J^{(n-2)})\right]^{vir}
$ to the boundary is zero-dimensional and equals $-2$.  The composition axiom applied to this class shows that
$$-2 = 2n r^2  - 2 s^2.$$

This shows that
\begin{align*}
\bone_3^{(n+1)/2} &= \bone_3\star ( (2n r x^{(n-1)/2}\bone_{n} -2 s y\bone_{n}))\\
&= 2n r^2 \bone_2 - 2 s^2 \bone_2\\
&=  -2 \bone_2.
\end{align*}
Proceeding in this manner, we find that
$$\bone_3^{l} = -2 \bone_{(2l-n+1)} \qquad \text{ if $(n+1)/2 \le l \le n-1$.}$$

We wish to choose constants $\alpha$ and $\beta$ so that the homomorphism
$$\phi:\C[X,Y] \rTo \ch_{D_{n+1},{\langle J\rangle}}, \quad 1 \mapsto \bone_1, \, X \mapsto \bone_3, \ Y \mapsto \alpha x^{(n-1)/2}\bone_{n} + \beta y\bone_{n}$$
has both $XY$ and $nX^{n-1}+Y^2$ in its kernel, but so that $\phi(Y)$ is not in the span $\langle \phi(1), \phi(X),\dots,\phi(X^{(n-1)}\rangle$.

A straightforward calculation shows that
$$\phi(Y^2) = \left(\frac{\alpha^2}{2n} - \frac{\beta^2}{2}\right)\bone_{n-1}.$$
Combining this with our previous calculations, we require
$$\frac{\alpha^2}{2n} - \frac{\beta^2}{2} = 2n.$$
Moreover, one easily computes that $\phi(XY) = (\alpha r + \beta s) \bone_2$ and so $ \alpha = -\beta s/r.$  This gives
$$ \beta = \pm 2nr, \quad \text{ and thus  }\quad
\alpha = \mp2ns.$$
With these choices of $\alpha$ and $\beta$ it is easy to check that  $\phi(Y)$ is not in the span $\langle \phi(1), \phi(X),\dots,\phi(X^{(n-1)}\rangle$.  This means that
$\phi$ is surjective and the ideal $(XY, nX^{n-1} + Y^2)$ lies in its kernel, and thus it induces the desired isomorphism of graded rings $\bar{\phi}:\milnor_{D_{n+1}} \rTo (\ch_{D_{n+1},\genj},\star)$.

As in the case of $E_7$, we wish to rescale $\bar{\phi}$ to make it also an isomorphism of Frobenius algebras.  The pairing for $\milnor_{D_{n+1}}$ is $$\langle X^{n-1},1\rangle = 1/2n \quad \text { and } \quad \langle Y^2,1\rangle = -1/2,$$ whereas the paring for $\ch_{D_{n+1},\genj}$ has
$$\langle \bone_3^{n-1},\unit\rangle = \langle -2 \bone_{n-1},\unit\rangle = -2 \quad \text { and } \quad \langle \bar{\phi}(Y^2),\unit\rangle = -n \langle \bar{\phi}(X^{n-1}),\unit\rangle = 2n.$$
Thus the pairing of $\ch_{D_{n+1},\genj}$ is a constant  $-4n$ times the pairing of $\milnor_{D_{n+1}}$.  Since both rings are graded and the pairing respects the grading, rescaling the homomorphism $\bar{\phi}$ by  an appropriate factor (namely, $\bar{\phi}(X) = \sigma^2 \bone_3$, and $\bar{\phi}(Y) = \sigma^{n-1} (\alpha x^{(n-1)/2}\bone_{n} + \beta y\bone_{n})$, with $\sigma^{2n-2} = 1/(-4n)$), shows that we can construct an isomorphism of graded Frobenius algebras
$$\milnor_{D_{n+1}}\cong (\ch_{D_{n+1},\genj},\star).$$

\paragraph{\underline{Case of $n=3$}}
In the case that $n=3$ we can determine all the correlators just by the selection rule (Equation~(\ref{eq:DnSelRule})) and the pairing.
Specifically, we have the correlators
\begin{align*}
\corf{D_4}{\bone_1,\bone_1,\bone_2} = 1 \quad & \corf{D_4}{x\bone_3,x\bone_3,\bone_1} = 1/6\\
\corf{D_4}{y\bone_3,y\bone_3,\bone_1} = -1/2 \quad & \corf{D_4}{x\bone_3,y\bone_3,\bone_1} = 0,
\end{align*}
and all other three-point correlators vanish.

It is easy to verify that the map $\phi:\C\rTo \ch_{D_4,\genj}$ taking $X\mapsto x\bone_3$ and $Y\mapsto y\bone_3$ induces an isomorphism of graded Frobenius algebras
$$\milnor_{D_4} \cong (\ch_{D_4,\genj},\star).$$

\subsection{Simple singularities which are not self-mirror}

\subsubsection{The singularity $D_{n+1}$ with its maximal Abelian symmetry group.}

In this subsection we will show that, regardless of whether $n$ is even or odd, the ring $\ch_{D_{n+1},G_{D_{n+1}}}$ with its maximal symmetry group $G_{D_{n+1}}$ is isomorphic, as a Frobenius algebra, to the Milnor ring $\milnor_{x^ny+y^2}$ of $D^T_{n+1}={x^ny+y^2}$.

Regardless of whether $n$ is even or odd, the maximal Abelian
symmetry group $G:=G_{D_{n+1}}$ of $D_{n+1}=x^n+xy^2$ is isomorphic to
$\Z/2n\Z$.  It is generated by  $\lambda := (\zeta^{-2},\zeta^{1})$, with
$\zeta=\exp(2 \pi i/2n)$.
We have
$J = \lambda^{n-1}$.  If $n$ is even, then $J$ generates the
entire group $G$, but if $n$ is odd, it generates a subgroup of
index $2$ in $G$. The case of $D_{n+1}$ with $n$ odd and with
symmetry group $\genj$ has already been treated in
Subsection~\ref{sec:DnOddFA}.

Define
$$\bone_{0}:=dx\!\wedge\! dy \in H^{mid}(\C^{N_{\lambda^0}}, W_{\lambda^{0}}^{\infty}, \Q),$$
$$\bone_{{a}}:= 1 \in H^{mid}(\C^{N_{\lambda^{a}}}, W_{\lambda^{a}}^{\infty}, \Q) \text{ for $0<a<n$ or $n<a<2n$,}.$$
After computing $G$-invariants, we find that the state space
$\ch_{D_{n+1},G}$  is spanned by the elements
$$ y\bone_0,\bone_1,\bone_2,\dots,\bone_{n-1},\bone_{n+1},\bone_{n+2},\dots,\bone_{2n-1}$$

We have
\begin{align}
\Theta_x^{\lambda^a} &=
    \begin{cases}
        a/n & \text{ if $0\le a < n$} \\
        a/n-1 &  \text{ if $n\le a < 2n$}
    \end{cases} \label{eq:DnThetax}\\
\Theta_y^{\lambda^a} &=
    \begin{cases}
        0 & \text{ if $a=0$} \\
        1-a/2n &  \text{ if $0 < a < 2n$}
    \end{cases}\label{eq:DnThetay}
\end{align}
and
\begin{equation}\label{eq:DndegW}
\deg_W(\bone_a) =
    \begin{cases}
        \frac{a-1}{n} +1 & \text{ if $0\le a < n$} \\
        \frac{a-1}{n} -1 &  \text{ if $n< a < 2n$.}
    \end{cases}
\end{equation}
For three-point correlators of the form $\langle \varkappa_1,\varkappa_2,\varkappa_3 \rangle^{D_{n+1}}_{0,3}$, with each $\varkappa_i$  in the $\lambda^{a_i}$-sector, the Dimension Axiom gives the selection rule
$$2\chat_{D_{n+1}} = \sum_{i=1}^3 \deg_W(\varkappa_i)$$
which, using Equation~(\ref{eq:DndegW}), gives
$$\sum_{i=1}^3 a_i = 2nB-n+1,$$
where $B$ is the number of $a_i$ greater than $n$.

Similarly, we compute the degree of each of the line bundles in the $D_{n+1}$-structure to be
\begin{align}
\deg(|\LL_x|) &= 1-B \notag\\
\deg(|\LL_y|) &= R+B-3,\label{eq:DnLBdeg}
\end{align}
where $R$ is the number of broad sectors $\varkappa_i \in H^{mid}(\C^{N_{\lambda^0}}, W_{\lambda^{0}}^{\infty}, \Q)$.
A straightforward case-by-case analysis of the possible choices for $B$ and $R$ shows that (up to reindexing) the only correlators that do not vanish for dimensional reasons are the following:
\begin{align*}
\langle \bone_{n+a},\bone_{n+b},\bone_{n+1-a-b}\rangle_{0,3}^{D_{n+1}} & \text{ with $0<a,b$ and $a+b\le n$}\\
\langle y\bone_0,\bone_{n+1+a},\bone_{2n-a}\rangle_{0,3}^{D_{n+1}} & \text{ with $0<a<n-1$}\\
\langle y\bone_0,y\bone_0,\bone_{n+1}\rangle_{0,3}^{D_{n+1}} &  = \eta_{y\bone_0,y\bone_0} = -\frac12.
\end{align*}
Using Equation~(\ref{eq:DnLBdeg}), we see that correlators of the first type are all concave and so are equal to $1$.  Those of the second type can be computed using the composition axiom; specifically, the Index-Zero Axiom shows that the restriction of the virtual cycle $\left[\W_{0,4}(D_{n+1};\lambda^{n+1+a},\linebreak[4] \lambda^{n+1+a},\lambda^{2n-a},\lambda^{2n-a})\right]^{vir}$ to the boundary is $-2$ times the fundamental cycle.  The Composition Axiom now shows that
$$\rho^*\left(\langle y\bone_0,\bone_{n+1+a},\bone_{2n-a}\rangle_{0,3}^{D_{n+1}}\right)^2 \eta^{y\bone_0,y\bone_0} = -2,$$
which gives
\begin{equation}\label{eq:DnBroadCor}
\langle y\bone_0,\bone_{n+1+a},\bone_{2n-a}\rangle_{0,3}^{D_{n+1}} =
\pm 1.\end{equation} Using these computations, it is now
straightforward to check that, regardless of the choice of sign in
Equation~(\ref{eq:DnBroadCor}), the map $\phi:\milnor_{x^ny+y^2}=
\C[X,Y]/(X^{n-1}Y, X^n+2Y) \rTo (\ch_{D_{n+1},G},\star)$ defined by
\begin{equation*}
X^i \mapsto \begin{cases}
\bone_{n+1+i} & \text{ for $0\le i < n-1$}\\
\mp 2y\bone_{0}& \text{ for $i=n-1$}\\
\bone_{i-n+1} & \text{ for $n\le i < 2n-1$}
\end{cases}
\dsand
Y \mapsto -\frac{X^n}{2} = -\frac{\bone_1}{2},
\end{equation*}
is an isomorphism of graded algebras.  The pairing on $\milnor_{x^ny+y^2}= \C[X,Y]/(X^{n-1}Y, X^n+2Y)$ is given by
$$\langle X^{2n-2},1\rangle^{\milnor_{D_{n+1}}} = -1/n,$$ whereas the pairing on $\ch_{D_{n+1},G_{D_{n+1}}}$ is easily seen to be given by $$\langle \phi(X^{2n-2}), \unit\rangle^{\ch_{D_{n+1}}} = \langle \bone_{n-1}, \bone_{n+1} \rangle^{\ch_{D_{n+1}}} = 1.$$
Since $\phi$ and the pairing both preserve the grading, we can rescale $\phi$ to be $\phi(X) = \alpha \bone_{n+2}$ and $\phi(Y) = -\alpha^n \bone_1/2$ with $\alpha^{2n-2} = -1/n$ to obtain an isomorphism of graded Frobenius algebras:
$$\milnor_{x^ny+y^2} \cong (\ch_{D_{n+1},G_{D_{n+1}}}, \star).$$

\subsubsection{The mirror partner $D^T_{n+1}$ of $D_{n+1}$.}\label{sec:DnDual}

The mirror partner of $D_{n+1}$ is the singularity $D^T_{n+1}:=x^ny+y^2$.  In this subsection we show that the  ring $\ch_{D^T_{n+1}}$ of $D^T_{n+1}$ with its maximal Abelian symmetry group is isomorphic, as a Frobenius algebra, to the Milnor ring $\milnor_{D_{n+1}}$.  Since we have already shown that the ring  $\ch_{D_{n+1}}$ with its maximal Abelian symmetry group is isomorphic to the Milnor ring of $D^T_{n+1}$, this will complete the proof that, at least at the level of Frobenius algebras, $D^T_{n+1}$ is indeed the mirror partner of $D_{n+1}$.

For this singularity, the weights are $q_x = 1/2n$ and $q_y = 1/2$, and the central charge is $\chat_{D^T_{n+1}} = (n-1)/n$.  If $\xi:=\exp(2 \pi i/2n)$, then the exponential grading operator is $J = (\xi, \xi^{n})$.  The element $J$ generates the maximal Abelian symmetry group $\genj  = G_W \cong \Z/2n\Z$.

Denote
$$\bone_0:=dx\!\wedge\! dy \in  \Gss{0},$$
$$\bone_{{a}}:= 1 \in \Gss{a} \text{ for $ 0 < a < 2n$.}$$
The $G_W$-invariant state space is
$\ch_W=\ch_{W,G_W}=\langle
x^{(n-1)}\bone_{0},
\bone_1,\bone_3,\bone_5,\dots,\bone_{2n-1}
\rangle.$
As always, the we have $$\deg_W(x^{n-1}\bone_0) = \chat_W = 2\frac{n-1}{2n}$$
Also, we have $\Theta_x^{J^{a}} = a/2n$ for $a \in \{0,\dots,2n-1\}$ and
$\Theta_y^{J^{a}} = a/2 \pmod{1}$, so  the degree of any element $\varkappa$ in the $J^{a}$-sector is given as follows:
$$\deg_W(\varkappa)=2\frac{a-1}{2n}  \text{if $a$ is odd and $a\in(0,2n)$}.$$

For the genus-zero, three-point correlators $\langle \varkappa_1,
\varkappa_2,\varkappa_3\rangle^W_{0,3}$ with homogeneous elements
$\varkappa_i \in \ch_{J^{a_i}}$, the Dimension Axiom gives that the
virtual cycle vanishes unless $$2\chat_{D^T_{n+1}} = \sum_l
\deg_W(\varkappa_i).$$ This occurs precisely when
\begin{equation}\label{eq:DndualSel} \sum_i a_i = 2n+1 -
nR,\end{equation}
 where $R$ denotes the number of $a_i$ which are equal to $0$, that is, the number of broad sectors.

Equation~(\ref{eq:DndualSel}) shows that $R\in \{0,1,2\}$.  And a simple computation shows that the degree of the $W$-structure line bundle $\LL_x$ is not integral if $R=1$, so we have only the two cases of $R=0$ and $R=2$.  In the case of $R=2$, Equation~(\ref{eq:DndualSel}) shows that the only non-vanishing correlator is $$\langle \bone_1, x^{n-1}\bone_0,x^{n-1}\bone_0\rangle^W_{0,3} = \eta_{x^{n-1}\bone_0,x^{n-1}\bone_0} = -\frac1n.$$
In the case that $R=0$ we have
$\deg(\LL_x) = \deg(\LL_y) = -1$, so by concavity, these correlators are all $1$.

Now define a map $\phi:\C[X,Y] \rTo \ch_{D^T_{n+1}}$ by $X^i\mapsto \bone_{2i+1}$ and $Y\mapsto nx^{n-1}\bone_0$.  It is straightforward to check that $\phi$ is a graded surjective homomorphism with kernel $(nX^{n-1}+Y^2, XY)$.  So $\phi$ defines an isomorphism of graded algebras
$$\milnor_{D_{n+1}} \cong (\ch_{D^T_{n+1}},\star).$$
The pairing on each of these algebras also respects the grading, and the two pairings differ by a constant multiple of $2n$. So rescaling the homomorphism $\phi$ by $X^i\mapsto \sigma^{2i} \bone_{2i+1}$ and $Y\mapsto n\sigma^{n-1}x^{n-1}\bone_0$ with $\sigma^{2n-2} = 1/2n$ makes an isomorphism of graded Frobenius algebras.

This shows that $D^T_{n+1}$ is indeed a mirror partner to $D_{n+1}$, and it completes the proof of Theorem~\ref{thm:mirror}.

\section{ADE-hierarchies and the Generalized Witten conjecture}\label{sec:six}

    The main motivation for
    Witten to introduce his equation was  the following conjecture:

    \begin{conj}[ADE-Integrable Hierarchy Conjecture] The total
potential
    functions of  the A, D, and E singularities with group
    $\langle J \rangle$
    are $\tau$-functions of the corresponding A, D, and E integrable hierarchies.
    \end{conj}

    The $A_n$-case has been established recently by
    Faber-Shadrin-Zvonkin \cite{FSZ}.
     One of our main results
    is the resolution of Witten's integrable
    hierarchies conjecture for the $D$ and $E$ series.
    It turns out that Witten's conjecture needs a modification in
    the $D_{n}$ case for $n$ odd. This modification is extremely
    interesting because it reveals a surprising role that mirror symmetry plays in  integrable
    hierarchies.
\subsection{Overview of the Results on Integrable Hierarchies}
    Let's start from the ADE-hierarchies. As we mentioned in the introduction,  there are two
    equivalent versions of ADE-integrable hierarchies---that of Drinfeld-Sokolov
    \cite{DS} and that of Kac-Wakimoto \cite{KW}. The version directly relevant to
    us is the Kac-Wakimoto ADE-hierarchies because the following
    beautiful work of Frenkel-Givental-Milanov reduces the problem
    to an explicit problem in Gromov-Witten theory. Let's describe their
    work.

    Let $W$ be a nondegenerate quasi-homogeneous singularity and
    $\phi_i$ ($i\leq \mu$) be the monomial basis of the Milnor ring with $\phi_1=1$. Consider the miniversal deformation
    space $\C^{\mu}$ where a point $\lambda=(t_1, \dots, t_{\mu})$ parameterizes
    the polynomial $W+t_1\phi_1+t_2\phi_2\cdots+t_{\mu}\phi_{\mu}.$
    We can assign a degree to $t_i$ such that the above perturbed polynomial has the degree one, i.e.,
    $\deg(t_i)=1-\deg(\phi_i)$.
    The tangent space $T_{\lambda}$ carries an associative  multiplication $\circ$ and an Euler
    vector field $E=\sum_i \deg(t_i)\partial_{t_i}$ with the unit $e=\frac{\partial W_{\lambda}}{\partial t_1}$. It is more subtle to construct a metric. We can consider
    residue pairing
    $$\langle f, g\rangle_{\lambda}=\Res_{x=0}\frac{fg \omega}{\frac{\partial W_{\lambda}}{\partial x_1}\cdots \frac{\partial W_{\lambda}}{\partial x_N}}$$
    using a holomorphic $n$-form $\omega$.
     A deep theorem of
    Saito \cite{S} states that one can choose a \emph{primitive form $\omega$}  such that the induced metric is flat.
    Together, it defines a Frobenius manifold structure on a neighborhood of zero of $\C^{\mu}$. We should mention that
    there is no explicit formula for the primitive form in general. However, it is known that for ADE-singularities the primitive
 form can be chosen to be a constant multiple of standard volume form, i.e., $c\,dx$ for $A_n$ and $c\,dx\wedge dy$ for DE-series.

    Furthermore, one can
    define a potential
    function $\F$, playing the role of genus-zero Gromov-Witten theory with only primary fields.
    It is constructed as follows.
    We want to work in flat coordinates $s_i$ with the property that $\deg_{\C}(s_i)=\deg_{\C}(t_i)$ and $\langle s_i, s_j\rangle$ are
    constant. The flat coordinates depend on the flat connection of the metric and
    hence the primitive form. Its calculation is important and yet a
    difficult problem.
    Nevertheless, we know that the flat coordinates exist thanks
    to
    the work of Saito \cite{S}. Then,
    consider the three-point correlator
    $C_{ijk}=\langle \partial_{s_i}, \partial_{s_j}, \partial_{s_k}\rangle$ as a function near zero in $\C^{\mu}$.
    We can integrate $C_{ijk}$ to obtain $\F$. Here, we normalize $\F$ such that $\F$ has leading term of degree three. We can differentiate $\F$ by the Euler vector field.
    It has the property $L_E \F=(\hat{c}_W-3)\F$. Namely, $\F$ is homogeneous of degree $\hat{c}_W-3$. The last condition means that, in the Taylor expansion
    $$\F=\sum a(n_1, \dots, n_{\mu})\frac{s^{n_1}_1 \cdots
    s^{n_{\mu}}_{\mu}}{n_1 ! \cdots n_{\mu}!},$$ we have
    $a(n_1, \dots, n_{\mu})\neq 0$
    only when $\sum n_i-\sum n_i (1-\deg_{\C}(s_i))=\sum \deg_{\C}(s_i)=\hat{c}_W-3.$
     Note that the degree in the Frobenius manifold is different from that of the A-model. For
    example, the unit $e$ has  degree $1$ instead of zero. The A-model degree is $1$ minus the B-model degree. With this relation in
    mind,  we will treat the insertion $s_i$ with degree $1-\deg_{\C}(s_i)$. Then, the above formula is
    precisely the selection rule of quantum singularity  theory.

    It is known that the Frobenius manifold of a
    singularity is semisimple in the sense that the Frobenius
    algebra on $T_{\lambda}$ at a generic point $\lambda$ is
    semisimple. On any semisimple Frobenius manifold, Givental
    constructed a \emph{formal} Gromov-Witten potential function.
    We will only be interested in the case that the Frobenius manifold
    is the one corresponding to the miniversal deformation space
    of a quasi-homogenous singularity $W$. We denote it by
    $${\D}_{W, formal}=\exp\left(\sum_{g\geq 0}h^{2g-2}\F^g_{formal}\right).$$
    The construction of ${\D}_{W, formal}$ is complicated, but we only need its following formal properties
    \begin{description}

    \item[(1)] $\F^0_{formal}$ agrees with $\F$ for primitive fields, i.e., with no descendants.
    \item[(2)] $\F^g_{formal}$ satisfies the same selection rule as a Gromov-Witten theory with $C_1=0$ and
                 dimension $\hat{c}_W$.
                 \item[(3)] ${\D}_{W, formal}$ satisfies all the formal axioms of Gromov-Witten theory.
    \end{description}
      The first property is obvious from the construction. The second property is a consequence of the fact that
      $\D_{W, formal}$ satisfies the dilaton equation and Virasoro constraints.
     A fundamental theorem of
    Frenkel-Givental-Milanov \cite{GM,FGM} is
    \begin{thm}
    For ADE-singularities, ${\D}_{W, formal}$ is a $\tau$-function of
    the  corresponding Kac-Wakimoto ADE-hierarchies.
    \end{thm}

    \begin{rem}
     Givental-Milanov first constructed a Hirota-type equation for Givental's formal total potential function.
     Later, Frenkel-Givental-Milanov proved that Givental-Milanov's Hirota equation is indeed the same as
     that of Kac-Wakimoto.
     \end{rem}

Our main theorem is
\begin{thm}\label{thm:IntHierMirror}

\

    \begin{description}
    \item[(1)] Except for $D_{n}$ with $n$ odd and $D_4$, the total potential functions of
DE-singularities with the group $\langle J \rangle $  are equal to
the corresponding Givental formal Gromov-Witten potential
functions up to a linear change of variables.
\item[(2)]
${\D}_{D_n, G_{max}}={\D}_{A_{2n-3}, formal}$, up to a linear change of
variables.
    \item[(3)] For $D^T_n=x^{n-1}y+y^2$ ($n > 4$), ${\D}_{D^T_n, G_{max}}={\D}_{D_n, formal}$,
     up to a
    linear change of variables.
    \end{description}
    \end{thm}

    Using the theorem of Frenkel-Givental-Milanov, we obtain
    \begin{crl}
    \begin{description}

    \

\item[(1)]  Except for $D_{n}$ with $n$ odd and $D_4$, the total potential
function of DE-singularities with the group $\langle J \rangle $
is a $\tau$-function of the corresponding Kac-Wakimoto hierarchies
(and hence Drinfeld-Sokolov hierarchies).
    \item[(2)] The total potential function of all $D_n$-singularities for $n>4$ with
    the maximal diagonal symmetry group is a $\tau$-function of the $A_{2n-3}$ Kac-Wakimoto hierarchies
    (and hence Drinfeld-Sokolov hierarchies).
    \item[(3)] The total potential function of $D^T_n=x^{n-1}y+y^2$ ($n> 4$)
    with the maximal diagonal symmetry group is a $\tau$-function of  the   $D_n$ Kac-Wakimoto hierarchies (and hence Drinfeld-Sokolov hierarchies).
    \end{description}
    \end{crl}

   \begin{rem}
    There is a technical issue in Givental's formal theory, as follows. For any semisimple point $t$ of Saito's Frobenius manifold,
    he defined an {\em ancestor potential} $\A_t$. From this he obtains a descendant potential function $\D=\hat{S}_t \A_t$, where
    $\hat{S}_t$ is certain quantization of a symplectic transformation $S_t$ determined by the Frobenius manifold. Then, he showed $\D$
    is independent of $t$. However, to compare with our $A$-model calculation, we need to expand $\D$ as formal power series at $t=0$.
    Although $\D$ is expected to have a power series expansion at $t=0$, we have been informed that a proof is not yet in the literature. Our strategy to avoid this problem is to
    show that (i) the A- and B-models have isomorphic Frobenius manifolds, and     (ii) in the ADE cases the ancestor functions of both models are completely determined by their respective Frobenius manifolds.
    Therefore, the A- and B-model have the same ancestor potentials and hence the same descendant potentials.
    \end{rem}
    \begin{df}
    An ancestor correlator is defined as
    $$\langle \tau_{l_1}(\alpha_1), \cdots, \tau_{l_n}(\alpha_n)\rangle^{W,G}_g(t)
    =\sum_k \langle \tau_{l_1}(\alpha_1), \cdots, \tau_{l_n}(\alpha_n), \overbrace{t, \cdots, t}^{\text{$k$ copies}}\rangle^{W,G}_g.$$
    Then, we define ancestor generating function $\F^{W,G}_g (t)$ of our theory with these correlators similarly.
    \end{df}

    Givental ancestor potential is defined for semisimple points $t\neq 0$. In the above definition, $t$ is only a formal variable. To be able to choose an actual value $t\neq 0$, we need to show that
    the ancestor correlator is convergent for that choice of $t$. This is done in the following lemma.
    \begin{lm}\label{lm:ances-poly}
    Choose a basis $T^i$ of $\ch_{W,G}$ and write $t=\sum_i t_iT^i$. For the simple (ADE) singularities, the ancestor correlator $\langle \tau_{l_1}(\alpha_1), \cdots, \tau_{l_n}(\alpha_n)\rangle^{W,G}_g(t)$ is a polynomial in the variables $t_i$. Furthermore, if $l_1=\cdots=l_n=0$, (i.e., if there are no $\psi$-classes)
    the ancestor correlator is also a polynomial in the variables $\alpha_1,\dots,\alpha_n$.
    \end{lm}

    \begin{proof}
    Consider correlator $\langle \tau_1(\alpha_1), \cdots, \tau_n(\alpha_n), T_{i_1}, \cdots, T_{i_k}\rangle^{W,G}_g.$ The dimension condition is
    $$2((\hat{c}_w-3)(1-g)+n+k)=\sum_i (2l_i+\deg_W\alpha_i)+\sum_j \deg_W T_{i_j}.$$
    This implies that
    $$\sum_j (2-\deg_W T_{i_j})=\sum_i (2l_i+\deg_W\alpha_i)-2((\hat{c}_w-3)(1-g)+n).$$
    Therefore, if we redefine $\deg'_W T_{i_j}:=2-\deg_W T_{i_j}$, the ancestor correlator is homogeneous of a fixed degree. When $W$ is an ADE singularity, it is straightforward to check that  $2-\deg_W T_{i_j}>0$. Hence, it must be a polynomial. The same argument implies the second case.
    \end{proof}
    This lemma shows that we can consider $\F^{W,G}_t$ and $\A^{W,G}_t$ for a semisimple point $t\neq 0$.

The proof of the main theorem depends on four key
    ingredients. The first ingredient is reconstruction
    theorem for the ADE-theory which shows that the two ancestor potentials are both determined by their corresponding Frobenius manifolds.  The second step is to show that the Frobenius manifolds are completely determined by genus-zero, three-point correlators and certain explicit four-point correlators. The third ingredient
    is  the \emph{Topological Euler class axiom for narrow sectors} which
    enables us to compute all the three-point  and required four-point correlators. The last ingredient is the mirror symmetry of
    ADE-singularities we proved in last section.
     The required modification in  the $D_n$ case when $n$ is odd is transparent from mirror symmetry.

\subsection{Reconstruction Theorem}\

\

In this subsection, we will establish the reconstruction theorem
simultaneously for ADE-quantum singularity theory and Givental's
formal Gromov-Witten theory in the ADE case. We use the fact that:
(i) both theories satisfy the formal axioms of Gromov-Witten
theories; (ii) they both have the same selection rules; (iii) they
both have isomorphic quantum rings up to a mirror transformation.
The last fact has been established in the previous section. To
simplify the notation, we state the theorem for the quantum
singularity theory of the A-model.

We start with the higher genus reconstruction using an idea of
Faber-Shadrin-Zvonkine \cite{FSZ}.

\begin{thm}\label{thm:FJR-reduction} If $\chat< 1$, then the ancestor potential function is uniquely
determined by the genus-zero primary potential (i.e., without gravitational
descendants). If $\chat=1$, then the ancestor potential function is uniquely
determined by its genus-zero and genus-one primary potentials.
\end{thm}

The proof of Theorem \ref{thm:FJR-reduction} is a direct
consequence of the following two lemmas, using the
Faber-Shadrin-Zvonkine reduction technique. For this argument we always assume that $\chat\le 1$.

\begin{lm}\label{lm:degr-smal} Let $\alpha_i\in \ch_{\gamma_i,G}$ for all $i\in\{1,\dots,n\}$ and let $\beta$ be any product of $\psi$ classes.
If $\chat<1$, then the integral
$\int_{\MM_{g,n+k}}\beta\cdot\Lambda^{W}_{g,n+k}(\alpha_1,\dots,\alpha_n, T_{i_1}, \cdots, T_{i_k})$
vanishes if $\deg\beta<g$ for $g\ge 1$. If $\chat=1$, then the above
integral vanishes when $\deg\beta<g$ for $g\ge 2$.
\end{lm}

\begin{proof} The integral
$\int_{\MM_{g,n+k}}\beta\cdot\Lambda^W_{g,n+k}(\alpha_1,\dots,\alpha_n, T_{i_1}, \cdots, T_{i_k})$
does not vanish only if
$$
\deg \beta=3g-3+n+k-D-\sum_{\tau=1}^{n+k} N_{\ga_\tau}/2,
$$
where $D=\chat(g-1)+\sum_{\tau}\iota_{\gamma_\tau}$. Recall that
$\iota_{\gamma}=\sum_{i=1}^N(\Theta^{\gamma}_i-q_i)$.

Now we have the inequality:
\begin{equation}
\begin{split}
\deg\beta=(3-\chat)(g-1)+\sum_{\tau=1}^{n+k}(1-\iota_{\gamma_\tau}) &=(3-\chat)(g-1)+\sum_{\tau=1}^{n+k}(1-\chat+\chat-\iota_{\gamma_\tau}-N_{\ga_\tau}/2)\\
&\ge (3-\chat)(g-1)+(n+k)(1-\chat),
\end{split}
\end{equation}
where we used the fact (easily verified for the simple singularities AD and E) that the complex degree $\deg_\C \alpha_\ga = \iota_{\ga}+N_{\ga}/2$ of a class  $\alpha_\ga\in\ch_\ga$ always satisfies
$$\deg_\C \alpha_\ga= \iota_{\gamma} + N_{\ga}/2 \le \chat.$$
Hence if $g\ge 2$ we have
$
\deg\beta\ge g$.
If $g=1$, then
$
\deg\beta>0
$ for $\chat<1$, and
$\deg\beta\ge 0$
for $\chat=1$, where the equality holds if and only if
$\deg_\C \alpha_{\ga_\tau}=\chat$ for all $\tau$.
\end{proof}

The following lemma treats the integral for higher-degree $\psi$
classes. It was proved in \cite{FSZ}, where it was called  \emph{$g$-reduction}.
\begin{lm}{\label{lm:degr-larg}} Let $P$ be a monomial in the $\psi$ and $\kappa$-classes in
$\MM_{g,k}$ of degree at least $g$ for $g\ge 1$ or at least $1$
for $g=0$. Then the class $P$ can be represented by a linear
combination of dual graphs, each of which has at least one
edge.
\end{lm}
\begin{proof}[\textbf{Proof of Theorem \ref{thm:FJR-reduction}}]
Take any correlators:
$$
\langle
\tau_{d_1}(\alpha_1)\cdots\tau_{d_k}(\alpha_n), T_{i_1}, \cdots, T_{i_k}\rangle_{g,n+k}=\int_{\MM_{g,n+k}}
\psi_1^{d_1}\cdots\psi_k^{d_n}\Lambda^W_{g,n+k}(\alpha_1,\dots,\alpha_n, T_{i_1}, \cdots, T_{i_k}).
$$
The total degree of the $\psi$-classes must either match
the hypothesis of Lemma \ref{lm:degr-smal} or match the hypothesis of
Lemma \ref{lm:degr-larg}.  If the total degree is small, then it
vanishes by Lemma \ref{lm:degr-smal}; If it is large, then the
integral is changed to the integral over the boundary classes
while decreasing the degree of the total integrated $\psi$ or
$\kappa$ classes. Applying the degeneration and composition laws,
the genus of the moduli spaces involved will also decrease. It is easy to see that one can continue this process until the original integral is represented by a linear combination of integrals over moduli spaces of genus zero and genus one, without gravitational descendants.
\end{proof}
\begin{rem} There is an alternative higher-genus reconstruction,
using Teleman's recent announcement \cite{Te} of a proof of
Givental's conjecture \cite{Gi4}. However, in the ADE-case the above argument is much simpler and achieves the same goal.
\end{rem}

The above theorem implies that all the ancestor correlators are determined by genus-zero ancestor correlators without $\psi$ classes. On the $B$-model side, Givental's genus-zero generating function is equal to Saito's genus-zero generating function. Hence, it is
well defined at $t=0$. Furthermore, Lemma~\ref{lm:ances-poly} shows that both the A- and B-model genus-zero functions without descendants are polynomials
and are defined over the entire Frobenius manifold. Finally, we observe that the genus-zero ancestor generating function is determined by
the ordinary genus-zero generating function (i.e., at $t=0$). Therefore, it is enough to compare the ordinary genus zero generating
functions.

Next, we consider the reconstruction of genus-zero correlators
using WDVV.
\begin{df}\label{df:primitive}
We call a class $\gamma$ \emph{primitive} if it cannot be written as $\gamma=\gamma_1\star\gamma_2$ for $ 0< \deg_{\C}
(\gamma_i)< \deg_{\C}(\gamma)$ (or, in the case of our A-model singularity theory, $0 <\deg_W(\gamma_i)<\deg_W(\gamma)$).
\end{df}
 We have the following  lemma.
\begin{lm}[Reconstruction Lemma]
Any genus-zero $k$-point correlator of the form $\corf{}{\ga_1,\dots,\ga_{k-3},\alpha, \beta,\ve\star\phi}$ can be rewritten as
\begin{equation}\label{eq:reconstruct}
\begin{split}
\langle \ga_1, \dots,\ga_{k-3},\alpha, \beta,\ve\star\phi\rangle =
S &+ \corf{}{\ga_1, \dots,\ga_{k-3},\alpha, \ve, \beta\star\phi}\\
&+ \corf{}{\ga_1, \dots,\ga_{k-3},\alpha\star\ve, \beta,\phi}\\
&- \corf{}{\ga_1, \dots,\ga_{k-3},\alpha\star\beta,\ve,\phi},\\
\end{split}
\end{equation}
where $S$ is a linear combination of genus-zero correlators with fewer than $k$ insertions.

Moreover, all the genus-zero $k$-point correlators $\corf{}{\gamma_1,
   \dots, \gamma_k}$ are uniquely determined by the pairing, by the three-point
   correlators, and by correlators of the form $\corf{}{\alpha_1, \dots,
   \alpha_{k'-2}, \alpha_{k'-1}, \alpha_{k'}}$ for $k'\leq k$, and
   such that $\alpha_i$ primitive for all $i\leq k'-2$.
   \end{lm}
\begin{proof}
Choose a basis $\{\delta_i\}$ such that $\delta_0=\ve\star\phi$ and let
   $\delta'_i$ be the dual basis with respect to the pairing (i.e.,
   $\langle\delta_i, \delta'_j\rangle=\delta_{ij}$). Using WDVV and the definition of the multiplication $\star$,
   we have the formula
\begin{align}
\corf{}{\ga_1,\dots,\ga_{k-3},\alpha, \beta,\ve\star\phi}
=& \corf{}{\ga_1,\dots,\ga_{k-3},\alpha,
\beta,\ve\star\phi}\corf{}{\delta'_0,\ve,\phi} \notag
\\
=&\sum_{{k-3=I\cup J}}\sum_{\ell} \corf{}{\gamma_{i\in I}, \alpha, \ve ,\delta_\ell}\corf{}
{\delta'_\ell, \phi, \beta,   \gamma_{j\in J}} \notag \\ 
&\quad -\sum_{\substack{k-3=I\cup J\\ J\neq \emptyset}}\sum_\ell \corf{}{\gamma_{i\in I}, \alpha,
\beta ,\delta_\ell}\corf{}{\delta'_\ell, \phi, \ve,  \gamma_{j\in J}}.\label{eq:fullreconst}
\end{align}
   All of the terms on the right-hand side are $k'$-point correlators with
   $k'<k$ except
\begin{align*}
&\sum_\ell\corf{}{\gamma_{i\leq k-3}, \alpha, \ve,
\delta_\ell}\corf{}{\delta'_\ell, \phi,\beta}+ \sum_\ell
\corf{}{\alpha, \ve, \delta_\ell}\corf{}{\delta'_\ell,\phi,\beta, \gamma_{j\leq k-3}}\\
&\quad- \sum_\ell
\corf{}{\alpha,\beta,\delta_\ell}\corf{}{\delta'_\ell,\phi, \ve,
\gamma_{j\leq k-3}}\\
&=\corf{}{\gamma_{j\leq k-3},  \alpha, \ve, \phi
\star\beta}+\corf{}{\alpha\star\ve, \phi,\beta,\gamma_{j\leq k-3}}-
\corf{}{\alpha \star \beta,\ve, \phi,\gamma_{j\leq k-3}},
\end{align*}
as desired.  This proves Equation~(\ref{eq:reconstruct}).

Now, suppose that $\corf{}{\ga_1,\dots,\ga_k}$ is such that $\ga_k$ is not primitive, so $\ga_k = \ve\star\phi$ with $\ve$ primitive.  Applying Equation~(\ref{eq:reconstruct}) shows that
$\corf{}{\ga_1,\dots,\ga_k}$ can be rewritten as a linear combination of correlators $S$ with fewer insertions  plus three more terms
\begin{equation*}
\begin{split}
\corf{}{\ga_1,\dots,\ga_k} = S + & \corf{}{\gamma_{j\leq k-3},  \ga_{k-2}, \ve, \ga_{k-1}\star\phi}\\
   &+ \corf{}{\gamma_{j\leq k-3}, \ga_{k-2}\star\ve, \ga_{k-1}, \phi}
    - \corf{}{\gamma_{j\leq k-3},\ga_{k-2} \star \ga_{k-1},\ve, \phi}.
\end{split}
\end{equation*}
Note that we have replaced $\gamma_{k-2}, \gamma_{k-1}, \gamma_k$  in the original correlator  by  $\gamma_{k-2}, \ve, \phi$ in the first and third terms, and by $\gamma_{k-1}, \phi, \gamma_{k-2}\star \ve$ in the second term.  So the first and third terms now have a primitive class $\ve$ where there was originally $\ga_{k-1}$. The second
   term has replaced $\ga_{k}$ by $\phi$, which has lower degree. We repeat the above argument on the second term $\corf{}{\gamma_{j\leq k-3},\gamma_{k-2}\star\ve, \gamma_{k-1},\phi}$
to show that the original correlator
$\corf{}{\ga_1,\dots,\ga_{k-3}, \gamma_{k-2}, \gamma_{k-1},
   \gamma_k}
$  can be rewritten in terms of correlators that are either shorter ($k'<k$) or which have replaced one of the three classes  $\ga_{k-2}$,$\ga_{k-1}$, or $\ga_{k}$ by a primitive class.

Now move the primitive class into the set
   $\gamma_{i\leq k-3}$. Pick another non-primitive class and
   continue the induction. In this way, we can replace all the
   insertions by primitive classes except the last two.
\end{proof}

\begin{df}
    We call a correlator a \emph{basic correlator} if it is of the form described in the previous lemma, that is, if all insertions are primitive but the last two.
\end{df}

    For a basic correlator, we still have the dimension formula
\begin{equation}\label{eq:dim}
\sum_i \deg_{\C}(a_i)=\chat+k-3.
\end{equation}
    This gives the following lemmas.
 \begin{lm}
     If $\deg_{\C}(a)\leq \chat$ for all classes $a$, and if $P$ is the maximum complex degree of any primitive class,
      then all the genus-zero correlators are uniquely determined by
   the pairing and $k$-point correlators with
  \begin{equation}\label{eq:dim-1}
  k\leq 2 +\frac{1+\chat}{1-P}\end{equation}
\end{lm}
\begin{proof}
       Let $\corf{}{a_1, \dots, a_{k-2}, a_{k-1}, a_k}$ be a basic correlator,
       so $a_{i\leq k-2}$'s are primitive. Then, $\deg_{\C}(a_i)\leq P$ for $i\leq k-2$
       and $deg_{\C}( a_{k-1}), deg_{\C}( a_{k})\leq \chat$. By the dimension
       formula we have
       $$\chat+k-3\leq (k-2)P+2\chat.$$
       \end{proof}

\begin{lm}\label{lm:recon-4point} All the genus-zero correlators for
the $A_n,D_{n+1},E_6,E_7,E_8$ and $D^T_{n+1}$ singularities, in either
the A-model or the B-model, are uniquely determined by the pairing, the three-point correlators, and the four-point correlators.
\end{lm}

\begin{proof} Since the pairing, the three-point correlators and the selection rules in the A-model and the B-model have been shown to be mirror, it suffices to prove the conclusion in the A-model side.

Let $P$ be the maximum complex degree of any primitive class. It is
easy to obtain the data for these singularities:
$$\begin{array}{ll}
A_n: \quad P=\frac{1}{n+1},\ \chat=\frac{n-1}{n+1}.
&E_6:  P=\frac{1}{3},\ \chat=\frac{5}{6}.\\
E_7: P=\frac{1}{3},\ \chat=\frac{8}{9}.
&E_8: P=\frac{1}{3},\ \chat=\frac{14}{15}.\\
D_{n+1}(\text{$n$ even}): P=\frac{1}{n},\ \chat=\frac{n-1}{n}.
&D_{n+1} (\text{$n$ odd}): P=\frac{n-1}{2n},\ \chat=\frac{n-1}{n}.\\
D^T_{n+1}: P=\frac{n-1}{2n},\ \chat=\frac{n-1}{n}.
\end{array}$$
 By formula (\ref{eq:dim-1}), we know that:
\begin{itemize}
\item[(1)] $k\le 4$ for $A_n,E_6,E_7,E_8$ and $D_{n+1}(n$ even) singularities;
\item[(2)] $k\le 5$ for $D_{n+1} (n$ odd) and  $D^T_{n+1}$ singularities.
\end{itemize}
For the singularities $D_{n+1} (n$ odd) and  $D^T_{n+1}$, we need a
more refined estimate.

For the singularity $D_{n+1}$ ($n$ odd), we have the isomorphism
$$
(\ch_{D_{n+1},\genj },\star)\cong \milnor_{D_{n+1}}.
$$
Here $\milnor_{D_{n+1}}$ is generated by $\{1,X,\dots,X^{n-1},Y\}$
and satisfies the relations $nX^{n-1}+Y^2\equiv 0$ and $XY\equiv
0$. $X$ and $Y$ are the only primitive forms, and they have complex degrees as follows.
$$
\deg_\C X=\frac{1}{n}\;,\deg_\C Y=\frac{n-1}{2n}.
$$
The basic genus-zero, five-point correlators may have the form $\langle X,
Y, Y,\alpha, \beta\rangle_0$. By the dimension formula
(\ref{eq:dim}) for $k=5$, we have
$$
\deg_\C \alpha+\deg_\C
\beta=\chat+2-\frac{n-1}{n}-\frac{1}{n}=\frac{2n-1}{n}>\frac{2n-2}{n}=2\chat.
$$
This is impossible, since for any element $a$  we have $\deg_\C(a) \le
\chat$. Similarly we can rule out the existence of the basic 5-point
functions of the form $\langle X, X, X,\alpha, \beta\rangle_0$ and $\langle X, X,
Y,\alpha, \beta\rangle_0$. Therefore the only possible basic 5-point
functions have the form $\langle Y, Y, Y,\alpha, \beta\rangle_0$. In
this case, we have the degree formula
$$
\deg_\C \alpha+\deg_\C \beta=\frac{3n+1}{2n}.
$$
Because of the fact that $X\star Y\equiv 0$, and for dimension
reasons, $\alpha$ or $ \beta$ can't contain $Y$. So the only possible
form of the basic five-point correlators are
$$
\langle Y, Y, Y,X^i, X^{\frac{3n+1}{2}-i}\rangle_0,\;i>0.
$$
Using formula (\ref{eq:reconstruct}) with $\alpha=Y,
\beta=X^i,\ve=X$ and $\phi=X^{\frac{3n-1}{2}-i}$, we have
\begin{align*}
&\langle Y, Y, Y,X^i, X^{\frac{3n+1}{2}-i}\rangle_0\\
=&S+\langle Y, Y, Y,X, X^{\frac{3n-1}{2}}\rangle_0+\langle Y, Y,
X\star Y,X^i, X^{\frac{3n-1}{2}-i}\rangle_0\\
-&\langle Y, Y, Y\star X^i, X^{\frac{3n-1}{2}-i},X\rangle_0\\
=&S
\end{align*} This shows that
any basic, genus-zero, five-point correlators can be uniquely determined by
two-, three-, and four-point correlators.

For the $D^T_{n+1}$ singularity, we have the isomorphism
$$
(\ch_{D^T_{n+1}},\star)\cong \milnor_{D_{n+1}}=\C[X,Y]/\langle
nX^{n-1}+Y^2,XY\rangle.
$$
and the degrees for the primitive classes $X$ and $Y$
$$
\deg_\C X=\frac{1}{n}\;\deg_\C Y=\frac{n-1}{2n}.
$$
Hence the reduction from basic five-point correlators to the fewer-point
correlators is exactly the same as for the singularity $D_{n+1}$ with
$n$ odd.
\end{proof}

The Reconstruction Lemma yields more detailed
       information for the basic correlators as well.
   \begin{thm}\label{thm:FourPtReduction}

\

   \begin{enumerate}
   \item All genus-zero correlators in the $A_{n-1}$ case for both our (A-model) and the Saito (B-model) theory  are uniquely
   determined by the pairing, the three-point correlators and a single
   four-point correlator of the form $\corf{}{X,X,X^{n-2}, X^{n-2}}$, where $X$ denotes the primitive class which is the image of $x$ via the Frobenius algebra isomorphism from $\milnor_{A_n} = \C[x]/(x^{n-1})$.
   \item All genus-zero correlators in the $D_{n+1}$ case of our (A-model) theory with maximal symmetry group, and in the $D_{n+1}^T$ case of the Saito (B-model), are uniquely
   determined by the pairing, the three-point correlators, and a single four-point
   correlator of the form: $\corf{}{X, X, X^{2n-2}, X^{2n-2}}$.  Again, $X$ denotes the primitive class which is the image of $x$ via the Frobenius algebra isomorphism from $\milnor_{D_{n+1}} = \C[x,y]/(nx^{n-1}+y^2,xy)$.

   \item All genus-zero correlators in the $D^T_{n+1}$ case of our (A-model) theory, in the $D_{n+1}$ case of our theory with $n$ odd and symmetry group $\genj$, and in the $D_{n+1}$ case of the Saito (B-model) theory are uniquely
   determined by the pairing, the three-point correlators, and four-point correlators of the form $\corf{}{X, X, X^{n-1}, X^{n-2}}$ and $\corf{}{X, X, Y, X^2}$.  The second of these occurs only in the case that $n=3$.
   Here $X$ and $Y$ denote the primitive classes which are the images of $x$ and $y$, respectively, via the Frobenius algebra isomorphism from $\milnor_{D^T_{n+1}} = \C[x,y]/(x^{n-1}y,x^n+2y)$.

   \item In the $E_6$ case of our theory (A-model) with maximal symmetry group, and in the $E_6$ case of the Saito (B-model) theory, all genus-zero correlators are uniquely
   determined by the pairing, the three-point correlators, and the correlators $\corf{}{Y,Y,Y^2,XY^2}$ and   $\corf{}{X,X,XY,XY}$.
   Here $X$ and $Y$ denote the primitive classes which are the images of $x$ and $y$, respectively, via the Frobenius algebra isomorphism from $\milnor_{E_6} = \C[x,y]/(x^2,y^3)$.
   \item In the $E_7$-case of our theory (A-model) with maximal symmetry group, and in the $E_7$ case of the Saito (B-model) theory, all genus-zero correlators are uniquely
   determined by the pairing, the three-point correlators, and
   the correlators   $\corf{}{X,X, X^2,XY}$, $\corf{}{X,Y,X^2,X^2}$,
   and $\corf{}{Y,Y,XY,X^2Y}$. Here $X$ and $Y$ denote the primitive classes which are the images of $x$ and $y$, respectively, via the Frobenius algebra isomorphism from $\milnor_{E_7} = \C[x,y]/(3x^2+y^3,xy^2)$.

   \item In the $E_8$-case of our theory (A-model) with maximal symmetry group, and in the $E_8$  Saito (B-model) theory, all genus-zero correlators are uniquely
   determined by the pairing, the three-point correlators, and by the correlators
   $\corf{}{Y,Y, Y^3,XY^3}$, and  $\corf{}{X,X,X,XY^3}.$
  Here $X$ and $Y$ denote the primitive classes which are the images of $x$ and $y$, respectively, via the Frobenius algebra isomorphism from $\milnor_{E_8} = \C[x,y]/(x^2,y^4)$.

    \end{enumerate}
\end{thm}
\begin{proof}   Applying Lemma \ref{lm:recon-4point}, all genus zero correlators are
   uniquely determined by the pairing, three- or four-point correlators. Let's study
   the genus zero four-point correlators in more detail.

   In the $A_{n-1}$ case, $X$ is the only ring generator, and hence
   the only primitive class. It has $\deg_\C X = 1/(n+1)$.  A dimension count shows that the only four-point
   correlator of the form $\corf{}{X, X, \alpha, \beta}$ is $\corf{}{X, X, X^{n-2},
   X^{n-2}}$.

   A similar argument shows that in the $D_{n+1}$ A-model with the maximal symmetry group and $D^T_{n+1}$ B-model cases
  the only basic four-point correlator is  $\corf{}{X,X,X^{2n-2},X^{2n-2}}$.

In the case of the $D^T_{n+1}$ A-model, and for the $D_{n+1}$
A-model for $n$ odd with symmetry group $J$, and for the  $D_{n+1}$
B-model, the central charges are the same, $\chat=\frac{n-1}{n}$, and
all have only two primitive classes $X$ and $Y$ with the same
degrees
$$
\deg_\C X=\frac{1}{n},\;\deg_\C Y=\frac{n-1}{2n}.
$$
Hence the basic four-point correlators are the same for the three
cases. Let's consider the case $D_{n+1}$ A-model for $n$ odd with
symmetry group $J$. There are several cases for the form of the
basic four-point correlators:
\begin{itemize}
\item[Case A:] form $\corf{}{X,X,\alpha,\beta}$. The dimension
formula shows that $\deg_\C\alpha +\deg_C \beta=\frac{2n-3}{n}$. So
the only possibility is $\corf{}{X,X,X^{n-1},X^{n-2}}$.
\item[Case B:] form $\corf{}{X,Y,\alpha,\beta}$. By the dimension
formula, we have $\deg_\C\alpha +\deg_C \beta=\frac{3n-3}{2n}$.
There are two cases:
\begin{itemize}
\item[Case B1:] $\alpha,\beta$ don't contain $Y$. Then the correlator
has the form $\corf{}{X,Y,X^i,X^j}$ for $j>1$. Setting $\alpha=Y,
\beta=X^i,\ve=X^{j-1}, \phi=X$ in  formula
(\ref{eq:reconstruct}), we have
\begin{align*}
&\corf{}{X,Y,X^i,X^j}=S+\corf{}{X,Y,X^{j-1},X^{i+1}}\\
&+\corf{}{X,Y\star X^{j-1},X^i,X}-\corf{}{X,Y\star X^i,X^{j-1},X}\\
&=S+\corf{}{X,Y,X^{j-1},X^{i+1}}=\cdots=S+\corf{}{X,Y,X,X^{i_0}}.
\end{align*}
The dimension formula shows that the only four-point correlator
$\corf{}{X,Y,X,X^{i}}$ does not vanish only if $n=3$ and in
this case $i=2$.
\item[Case B2:] $\alpha,\beta$ contain $Y$. In this case,
$\corf{}{X,Y,\alpha,\beta}$ has the form $\corf{}{Y,Y,X,\beta}$
which can be included in the following Case C.
\end{itemize}

\item[Case C:] form $\corf{}{Y,Y,\alpha,\beta}$. We have the degree
formula $\deg_\C\alpha +\deg_C \beta=1$. There are two cases:
\begin{itemize}
\item [Case C1:] $\alpha,\beta$ don't contain $Y$. We have the form
$\corf{}{Y,Y,X^i,X^j}$ with $j>1$. Let $\alpha=Y, \beta=X^i,
\ve=X^{j-1},\phi=X$ in the formula (\ref{eq:reconstruct}); we obtain
\begin{align*}
&\corf{}{Y,Y,X^i,X^j}=S+\corf{}{Y,Y,X^{j-1},X^{i+1}}\\
&+\corf{}{Y,Y\star X^{j-1},X^i,X}-\corf{}{Y,Y\star X^i,X^{j-1},X}\\
&=S+\corf{}{Y,Y,X^{j-1},X^{i+1}}=\cdots=S+\corf{}{Y,Y,X,X^{n-1}}.
\end{align*}
Now
\begin{align*}
&\corf{}{Y,Y,X,X^{n-1}}=\corf{}{X,Y,Y,X^{n-1}}\\
&=S+\corf{}{X,Y,X,Y\star X^{n-2}}\\
&+\corf{}{X,X\star Y,Y,X^{n-2}}-\corf{}{X,Y^2,X^{n-2},X}\\
&=S.
\end{align*}

\item[Case C2:] $\alpha,\beta$ contain $Y$. The basic correlator has the form
$\corf{}{Y,Y,Y,X^{\frac{n+1}{2}}}$. Similarly, we have
\begin{align*}
&\corf{}{Y,Y,Y,X^{\frac{n+1}{2}}}=S+\corf{}{Y,Y,X,Y\star X^{\frac{n-1}{2}}}\\
&+\corf{}{Y,Y\star X,Y,X^{\frac{n-1}{2}}}-\corf{}{Y,Y^2,X^{\frac{n-1}{2}},X}\\
&=S.
\end{align*}
\end{itemize}
\end{itemize}
In summary, if $n>3$, then the basic four-point correlator is only
$\corf{}{X,X,X^{n-1},X^{n-2}}$; if $n=3$, then the basic four-point
correlators are $\corf{}{X,X,X^{2},X}$ and $\corf{}{X,Y,X,X^{2}}$.

   In the $E_6$ case, the primitive classes are $X, Y$. The dimension
   condition shows that the only four-point correlators with two
   primitive insertions are
   $$\corf{}{Y,Y,Y^2,XY^2},\  \corf{}{X,X,X,XY^2}, \ \corf{}{X,X,XY,XY}.$$
   Applying Equation~(\ref{eq:reconstruct}) and the fact that $X^2=0$, we can reduce
   $\corf{}{X, X,X, XY^2}$ to $\corf{}{X, X, XY, XY}$.

   In the $E_7$ case, the primitive classes are $X$ and $Y$ with $\deg_{\C}X = 1/3$, $\deg_{\C}Y = 2/9$ and and $\chat = 8/9$. The dimension
   condition shows that the only basic four-point correlators are
   $$\corf{}{X,X, X^2,XY},\ \corf{}{X,X, X,X^2Y},\ \corf{}{X,Y,X^2,X^2},\ \corf{}{X,Y,Y^2,X^2Y},\ \corf{}{Y,Y,XY,X^2Y}.$$
   We can use Equation~(\ref{eq:reconstruct}) to further reduce
   $\corf{}{X,X,X, X^2Y}$ to the remaining four, and to reduce $\corf{}{X,Y,Y^2,X^2Y} = \corf{}{Y,X,Y^2,X^2Y}$ to  $\corf{}{Y,Y, XY, X^2Y}$.

   Finally, in the $E_8$ case, a dimension count shows that the only basic four-point correlators are
   $$\corf{}{X,X,X,XY^3},\ \corf{}{X,X,XY,XY^2},\ \corf{}{Y,Y,Y^3,XY^3}.$$
Again Equation~(\ref{eq:reconstruct}) shows that $\corf{}{X,X,XY,XY^2}$ can be expressed in terms of $\corf{}{X,X,X,XY^3}$.
\end{proof}

\subsection{Computation of the basic four-point correlators in the A-model}

\subsubsection{Computing classes in complex codimension one.}

\begin{df}
Let $\bGagkW$ denote the set of all connected single-edged $W$-graphs of genus $g$ with $k$ tails decorated by elements of $\ch_W$.  Further denote by $\bGagkWcut$ the set of all $W$-graphs with no edges (possibly disconnected), but with one pair of tails labeled $+$ and $-$, respectively, such that gluing the tail $+$ to the tail $-$ gives an element of $\bGagkW$.  We furthermore require that the decorations $\ga_+$ and $\ga_-$ satisfy $\ga_+\ga_- = 1$.

Similarly, let $\bGagkW(\ga_1, \dots, \ga_k)$ and $\bGagkWcut(\ga_1, \dots, \ga_k)$ denote the subset of $\bGagkW$ and of $\bGagkW$, respectively, consisting of decorated $W$-graphs with the $i$th tail decorated by $\ga_i$ for each $i\in \{1,\dots,k\}$.

For any graph  $\Gac\in \bGagkWcut$, we denote by $\Ga \in \bGagkW$ the uniquely determined graph in $\bGagkW$ obtained by gluing the two tails $+$ and $-$.  We further denote the underlying undecorated graph by $|\Ga|$  and we denote the closure in $\MM_{g,k}$ of the locus of stable curves with dual graph $|\Ga|$ by $\MM(|\Ga|)$.  Finally, denote the Poincar\'e dual of this locus by $\left[\MM(|\Ga|)\right]\in H^*(\MM_{g,k},\C)$.
\end{df}

\begin{rem}
In genus zero the local group at an edge (or at the tails labelled $+$ and $-$) is completely determined by the local group at each of the tails.
\end{rem}

\begin{thm}\label{thm:ConcaveCodimOne}
Assume the  $W$-structure is concave (that is $\pi_*\left(\bigoplus_{i=1}^t\LL_i\right)=0$) with all marks narrow.
If the $i$th mark is labeled with group element $\gamma_i$, and if the complex codimension $D$ is $1$,
then the class $\Lambda^W_{g,k}(\bone_{\gamma_1}, \dots, \bone_{\gamma_k}) \in H^*(\MM_{g,k},\C)$ is given by the following:
\begin{multline}\label{eq:CodimOneConvex}
\Lambda^W_{g,k}(\bone_{\gamma_1}, \dots, \bone_{\gamma_k}) =
\sum_{\ell=1}^N \left[\left(\frac{q_\ell^2}{2} - \frac{q_\ell}{2} + \frac{1}{12}\right)\kappa_1 -  \sum_{i=1}^k\left(\frac{1}{12} - \frac{1}{2}\Theta_\ell^{\gamma_i}
(1-\Theta_\ell^{\gamma_i})\right) \psi_i \right.\\
\left.+ \frac12\sum_{\Gac \in \bGagkWcut({\ga_1},\dots,{\ga_k})}  \left(\frac{1}{12} - \frac{1}{2}\Theta_\ell^{\gap}
(1-\Theta_\ell^{\gap})\right) \left[\MM(|\Ga|)\right]\right]
\end{multline}
\end{thm}

\begin{proof}
The proof follows from the orbifold Grothendieck-Riemann-Roch (oGRR) theorem \cite{Toen}.  After we finished this paper, we became aware of an elegant alternative treatment of this type of problem by Chiodo \cite{Ch3}. We now review the oGRR theorem in the case we are interested in, namely, orbicurves. For more details on oGRR, see \cite[Appendix A]{Tseng}.

For any $k$-pointed family of stable orbicurves $(\cC \rTo^{\pi} T, \sigma_1, \dots,\sigma_k)$ over a scheme $T$, with $W$-structure $(\LL_1,\dots,\LL_N,\phi_1,\dots,\phi_s)$, if the $W$-structure has type $\bgamma = (\ga_1,\dots,\ga_k)$ then
the inertia stack $\IC = \coprod_{g\in G} \cC_{(g)}$ consists of the following sectors:
$$\IC = \cC \sqcup \coprod_{i=1}^k\coprod_{j = 1}^{r_i-1} \cS_i(\ga_i^{j}) \sqcup \coprod_{\Ga\in\bGa_{g,k,W}(\bgamma)} \coprod_{j = 1}^{r_{\Ga}-1} \cZ_{\Ga}(\ga_{\Ga}^{j}).$$
Here $r_i$ is the order of the element $\ga_i$ and $r_\Ga$ is the order of the element $\ga_\Ga$.  Also, $\cS_i(\ga_i^{j}) := \cS_i$ is the $i$th gerbe-section of $\pi$, that is, the image of $\sigma_i$ with the orbifold structure inherited from $\cC$.  The notation $\cS_i(\ga_i^{j})$ just indicates that this is part of the $\ga_i^{j}$-sector of $\IC$.  Similarly,  $\cZ_{\Ga}(\ga_{\Ga}^{j}) := \cZ_{\Ga}$ is the locus of nodes in $\cC$ with dual graph $\Ga$ lying in the $\ga_{\Ga}^{j}$-sector.  As in the case of marks, the nodal sector $\cZ_{\Ga}$ should be given the orbifold structure it inherits from $\cC$.

Let  $\r:\IC \rTo \cC$ denote the obvious union of inclusions. Furthermore, let $\Ipi:  \IC \rTo T$ denote the composition $\Ipi = \pi \circ \r$.  And let $\rho: K(\IC) \rTo K(\IC)\otimes \C$ denote the Atiyah-Segal decomposition
$$\rho(E) = \sum_{\zeta} \zeta E_{\gamma,\zeta},$$
where for each sector $\cC_{(\ga)}$ the sum runs over eigenvalues $\zeta$ of the action of $\ga$ on $E$, and $E_{\ga,\zeta}$ denotes the eigenbundle of $E$ where $\ga$ acts as $\zeta$.

Define $$\cht = Ch \circ \rho \circ \r^* :K(\cC) \rTo H^*\left(\IC,\C\right)$$ and
$$\tdt(E):= \frac{Td((\r^*E)_1)}{Ch(\rho\circ\lambda_{-1}(\sum_{\zeta\neq 1}(\r^*E)_\zeta)^{\vee})}.$$

The oGRR theorem states that for any bundle $E$ on $\cC$ we have
\begin{equation}
\cht(R\pi_*E) =
\Ipi_*(\cht(E) \tdt(T_{\pi}))
\end{equation}
Writing this out explicitly for one of the $W$-structure bundles $\LL_\ell$ on our $W$-curve $\cC\rTo^{\pi}T$ we have
\begin{multline}\label{eq:OGRRmess}
Ch(\pi_*\LL_\ell \ominus R^1\pi_* \LL_\ell)\\ =
\pi_*(Ch(\LL_\ell)Td(T_{\pi})) + \sum_{i=1}^k\sum_{j=1}^{r_i-1}
\pi_*\left(\frac{\exp\left(2 \pi i \Theta_\ell^{\ga_i^j}
c_1(\r^*\LL_{\ell})\right)}{\left(1-\exp(2 \pi i j q_\ell
c_1(\r^*K))\right)} \right)\\
 + \frac12\sum_{\Gac\in\bGagkWcut(\bgamma)} \sum_{j = 1}^{r_{\Gac}-1}\pi_*\left(\frac{\exp\left(2\pi i
 \Theta_{\ell}^{{\ga^j_+}} c_1(\r^*\LL_{\ell})\right)}{\left(1-\exp(2\pi i j q_{\ell} c_1(\r^*K))
 \right)\left(1-\exp(-2\pi i j q_{\ell} c_1(\r^*K))\right)}\right).
\end{multline}

For our present purposes, we need only compute the codimension-one part of this sum.
Denote the first Chern class of $\LL_\ell$ on $\cC$ by $L_\ell$.
Note that because $\LL_\ell$ is part of the $W$-structure, and because the singularity is nondegenerate (so the matrix $B$ has maximal rank), we have $$L_\ell = c_1(\LL_\ell) = q_\ell K_{\log}.$$

Copying Mumford's argument given in \cite[\S5]{Mum},
one computes that the codimension-one part of the untwisted sector contribution to this sum is
$$\pi_*\left(L^2_\ell/2 - L_\ell K /2 + K^2/12 + \frac{1}{24}\sum_{\Gac\in\bGagkWcut(\bgamma)} i_{\Ga *}(1)\right),$$
where $i_\Ga$ is the inclusion into $\cC$ of the nodes corresponding to the edge of $\Ga$.

For each sector $\cS_i(\ga_i^j)$, the induced map $\pi_*$ is just $\frac{1}{r_i}\sigma_i^*$; therefore, on these sectors we have
$$\pi_*L_\ell = \frac{1}{r_i} \sigma_i^*L_\ell  = \frac{q_\ell}{r_i} \sigma_i^*K_{\log} = 0.$$
Now $\ga_i$ acts on the canonical bundle $K$ at the mark $\cS_i$ by multiplication by  $\xi_i:=\exp(2 \pi i /r_i)$, and it acts on $\LL_\ell$ at $\cS_i$ by $\exp(2 \pi i \Theta_\ell^{\ga_i}) = \xi_i^{a_i}$ for $a_i:=r\Theta_\ell^{\ga_i} \in [0,r_i) \cap \Z$.
Expanding the denominator in Equation~(\ref{eq:OGRRmess}), one sees that the codimension-one part of the contribution from the marks is
\begin{equation}
\sum_{i=1}^k \sum_{j=1}^{r_i-1} \frac{\xi_i^{(a_i+1)j} }{r_i (1-\xi_i^j)^2}
 \psit_i
\end{equation}
Similarly, letting $\xi_\Ga:=\exp(2 \pi i/r_\Ga)$ and choosing $a_\Ga:=r_\Ga\Theta_\ell^{\ga_\Ga} \in [0,r_\Ga)\cap \Z$ so that $\xi_\Ga^{a_\Ga} =\exp(2\pi i \Theta_\ell^{\ga_\Ga})$, one sees that the contribution to Equation~(\ref{eq:OGRRmess}) from the nodes is
\begin{equation}
\frac12\sum_{\Gac\in\bGagkWcut(\bgamma)} \sum_{j=1}^{r_\Gac-1} \frac{-\xi_i^{(a_\Gac+1)j} }{(1-\xi_\Gac^{j})^2} \pi_*(i_{\Ga *}(1))
\label{eq:HTH-roots}
\end{equation}

A long but elementary computation shows that for any primitive $r$th root $\zeta$ of unity and any $a\in[0,r)\cap\Z$, we have\footnote{We would like to thank H.~Tracy Hall for showing us this relation.}
\begin{equation}\label{eq:Ttwoa}
\sum_{j=1}^{r-1} \frac{\zeta^{(a+1)j} }{(1-\zeta^j)^2} = \frac{1 - r^2}{12} + \frac12 a (r-a).
\end{equation}

Using the definition $\kappa_1 = \pi_* (c_1(K_{\log}))^2 = \pi_*(c_1(K))^2 + \sum_{i=1}^k \psit_i$, together
with Equation~(\ref{eq:Ttwoa}) and the fact
that $a_i/r_i = \Theta_\ell^{\ga_i}$ and   $a_\Ga/r_\Ga = \Theta_\ell^{\ga_\Ga}$, we now have
\begin{align*}
&Ch(\pi_*\LL_\ell \ominus R^1\pi_* \LL_\ell)\\
=& \left(\frac{q_\ell^2}{2} - \frac{q_\ell}{2}   + \frac{1}{12}
\right) \kappa_1 - \sum_{i=1}^k \frac{\psit_i}{12} +
\sum_{\Ga\in\bGagkW(\bgamma)} \frac{\pi_*i_{\Ga *}(1)}{12}
-\sum_{i=1}^k \frac{r_i}{12}\left(\frac{1}{r_i^2} - 1 + 6\,
\Theta_\ell^{\ga_i}(1-\Theta_\ell^{\ga_i})\right)
 \psit_i \\
-& \frac12\sum_{\Gac\in\bGagkWcut(\bgamma)}
\frac{r_\Gac^2}{12}\left(\frac{1}{r_\Gac^2} - 1 + 6
\Theta_\ell^{{\gap}}(1-\Theta_\ell^{{\gap}})\right) \pi_*i_{\Ga
*} (1)\\
=& \left(\frac{q_\ell^2}{2} - \frac{q_\ell}{2} +
\frac{1}{12}\right)\kappa_1 -  \sum_{i=1}^k\left(\frac{1}{12} -
\frac{1}{2}\Theta_\ell^{\gamma_i} (1-\Theta_\ell^{\gamma_i})\right)
\psi_i\\
+& \frac12\sum_{\Gac \in \bGagkWcut(\bgamma)}  r_\Gac \left(\frac{1}{12} -
\frac{1}{2}\Theta_\ell^{{\gap}} (1-\Theta_\ell^{{\gap}})\right)
\left[\W(\Ga)\right],
 \end{align*}
where the last equality follows from the fact that $\psit_i = \psi_i/r_i$ and $\pi_* i_{\Ga *}(1) = \left[\W(\Ga)\right]/r_\Gac$.

Finally, in the concave case, we have $\pi_*(\LL_\ell)=0$, so
pushing down to $\MM_{g,k}$ gives
\begin{align*}
&\Lambda^W_{g,k}(\bone_{\ga_1},\dots,\bone_{\ga_k}) \\
&= \frac{1}{\deg(\st)} \sum_{\ell=1}^N \st_*c_1(-R^1\pi_* \LL_\ell) \\
&= \sum_{\ell=1}^N \left[\left(\frac{q_\ell^2}{2} - \frac{q_\ell}{2}
+ \frac{1}{12}\right)\kappa_1 - \sum_{i=1}^k\left(\frac{1}{12} -
\frac{1}{2}\Theta_\ell^{\gamma_i} (1-\Theta_\ell^{\gamma_i})\right)
\psi_i\right.\\
+&\left. \frac12\sum_{\Gac \in \bGagkWcut(\bgamma)}  \left(\frac{1}{12} -
\frac{1}{2}\Theta_\ell^{{\gap}} (1-\Theta_\ell^{{\gap}})\right)
\left[\MM(|\Ga|)\right]\right],
\end{align*}
since $\kappa_1$ and $\psi_i$ on $\W_{g,k}(W)$ are equal to the pullbacks $\st^*\kappa_1$ and $\st^*\psi_i$, respectively, and $\left[\W(\Ga)\right] =  \st^*\left[\MM(|\Ga|)\right]/r_\Ga$.
\end{proof}

\subsubsection{Four-point correlators for $E_7$}

Now we compute the genus-zero four-point correlators for $E_7$ with symmetry group $G_{E_7} = \genj$.  We will continue to use the notation of Subsection~\ref{sec:EsevenFA}.   By Theorem~\ref{thm:FourPtReduction}(5) we need only compute the following correlators to completely determine the Frobenius manifold, and thereby the entire cohomological field theory:
$$\corf{E_7}{Y,Y, XY, X^2Y}, \corf{E_7}{X, Y, X^2, X^2}, \corf{E_7}{X, X, X^2, XY}$$
We use the identification of $X, Y$ with the A-model classes from last section. To simplify the notation, we choose $\alpha=1$ instead of
$\alpha^8=\frac{1}{9}$. Later, we will re-scale the primitive form to take care of discrapency between the pairing.
These correspond to the correlators
$$\langle \bone_2 , \bone_5 , \bone_5 , \bone_8 \rangle_{0}^{E_7},\
\langle \bone_4 , \bone_4 , \bone_5 , \bone_7 \rangle_{0}^{E_7},\
\langle \bone_2 , \bone_4 , \bone_7 , \bone_7 \rangle_{0}^{E_7}.$$
These are all concave and have only narrow markings,  so we may use Theorem~\ref{thm:ConcaveCodimOne} to compute them. To apply that theorem, we need to use the fact that  $$\int_{\MM_{0,4}} \kappa_1 = \int_{\MM_{0,4}} \psi_i = \int_{\MM_{0,4}} \left[\MM(|\Ga|)\right]=1$$ for every $i\in\{1,\dots,4\}$ and every graph $\Ga\in\bGa_{0,4}$.  We also need to compute the group element $\ga_\Ga$ for each of the four-pointed, genus-zero, decorated $W$-graphs.  This is uniquely determined by the fact that the sum of the powers of $J$ on each three-point correlator must be congruent to $1 \mod 9$.  We  will work out the details in the case of  $\langle \bone_2 , \bone_4 , \bone_7 , \bone_7 \rangle_{0}^{E_7}$---the others are computed in a similar manner.

There are three graphs in $\bGa_{0,4,E_7}(J^2 , J^4 , J^7 , J^7)$; the first we will denote by $\Ga_1$ and is depicted in Figure~\ref{fig:Corr-tfss}.  There are two cut graphs $\Ga_{1,\mathrm{cut}}, \Ga'_{1,\mathrm{cut}}\in \bGa_{0,4,E_7,\mathrm{cut}}(J^2 , J^4 , J^7 , J^7)$ that glue to give the graph $\Ga_1$. These have the tail $+$ labeled with $\ga_{+} = J^{4}$ and the tail $-$ labeled with $\ga_{-} = J^{5}$, or in the second case, $\ga_{+} = J^{5}$ and $\ga_{-} = J^{4}$.  The formula gives the same result for each of these two cases, canceling the factor of $\frac12$ outside the sum for this term.

The other two graphs are both decorated as in Figure~\ref{fig:Corr-tsfs}.  We will abuse notation and denote both of them by $\Ga_2$ and simply count the contribution of $\Ga_2$ twice.  The edge of $\Ga_2$ is labeled with $\ga_{+} = J$ or $\ga_+ = J^8$, and again, the contribution to the formula from these two choices is identical and cancels the factor of $\frac12$ outside the sum.
\begin{figure}
\includegraphics{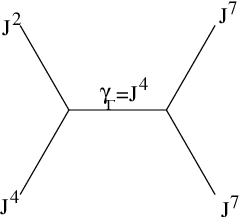}
\caption{The graph $\Ga_1 \in \bGa_{0,4,E_7}(J^2 , J^4 , J^7 , J^7)$. \label{fig:Corr-tfss}}
\end{figure}
\begin{figure}
\includegraphics{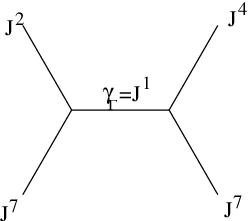}
\caption{The graph $\Ga_2$.  Two of the three graphs in $\bGa_{0,4,E_7}(J^2 , J^4 , J^7 , J^7)$ are decorated as in this figure.\label{fig:Corr-tsfs}}
\end{figure}

Now, it is easy to check that the degree of $\LL_x$ is $-1$, so $R^1\pi_*\LL_x = 0$ and this will not contribute to the correlator.  We have
\begin{align*}
&\langle X, X, X^2, XY\rangle_{0}^{E_7}= \int_{\MM_{0,4}}
\Lambda^{E_7}_{0,4}(\bone_2 , \bone_4 , \bone_7 , \bone_7)\\
&= \left(\frac{q_y^2}{2} - \frac{q_y}{2} +
\frac{1}{12}\right) -  \sum_{i=1}^4\left(\frac{1}{12} -
\frac{1}{2}\Theta_y^{\gamma_i} (1-\Theta_y^{\gamma_i})\right) +
\left(\frac{1}{12} - \frac{1}{2}\Theta_y^{\ga_{\Ga_1}}
(1-\Theta_y^{\ga_{\Ga_1}})\right)\\
 &+ 2 \left(\frac{1}{12} -
\frac{1}{2}\Theta_y^{\ga_{\Ga_2}}
(1-\Theta_y^{\ga_{\Ga_2}})\right) \\
&= \left(\frac{q_y^2}{2} - \frac{q_y}{2}\right) +
\left(  \frac{1}{2}\Theta_y^{J^{2}}
(1-\Theta_y^{J^{2}})\right)
+\left(  \frac{1}{2}\Theta_y^{J^{4}}
(1-\Theta_y^{J^{4}}\right)
+2\left(  \frac{1}{2}\Theta_y^{J^{7}}
(1-\Theta_y^{J^{7}})\right)\\
& \qquad-  \left(  \frac{1}{2}\Theta_y^{{J^{4}}}
(1-\Theta_y^{J^{4}})\right)
- 2 \left(\frac{1}{2}\Theta_y^{J^{1}}
(1-\Theta_y^{J^{1}})\right) \\
&= \left(\frac{4}{2\cdot 81} - \frac{2}{2\cdot9}\right) +  \left( \frac{1}{2}\cdot\frac49\cdot
\frac59\right)
+ \left( \frac{1}{2}\cdot\frac89
\cdot\frac19\right)
+  2 \left( \frac{1}{2}\cdot\frac59
\cdot\frac49\right)  \\
& \qquad -
 \left( \frac{1}{2}\cdot\frac89\cdot
\frac19\right)
-
2 \left( \frac{1}{2}\cdot\frac29
\cdot\frac79\right)
\\
&=\frac{1}{9}.
\end{align*}
And similar computations show that
$$\corf{E_7}{X, Y, X^2, X^2} = -\frac {1}{9},$$ 
and 
$$\corf{E_7}{Y,Y,XY,X^2 Y} = \frac13.$$

\subsubsection{Four-point correlators for $E_6$}

By Theorem~\ref{thm:FourPtReduction} we need only compute the correlators $\corf{E_6}{Y, Y, Y^2, XY^2}$
and $\corf{E_6}{X, X, XY, XY}$. Here, again, we choose $\alpha=1$.
These correspond to the correlators $\corf{E_6}{\bone_{10},\bone_{10},\bone_{7},\bone_{11}}$ and  $\corf{E_6}{\bone_{5},\bone_{5},\bone_{2},\bone_{2}}$.

The correlators in question are easily seen to be concave.
Applying Theorem~\ref{thm:ConcaveCodimOne} in a manner similar to the previous computations, we find that
$$\corf{E_6}{Y, Y, Y^2, XY^2}=\corf{E_6}{\bone_{10},\bone_{10},\bone_{7},\bone_{11}} = \frac{1}{4}$$
and
$$\corf{E_6}{X, X, XY, XY}=\corf{E_6}{\bone_{5},\bone_{5},\bone_{2},\bone_{2}} = \frac{1}{3}.$$

\subsubsection{Four-point correlators for $E_8$}

By Theorem~\ref{thm:FourPtReduction} we need only compute the correlators $\corf{}{Y, Y, Y^3, XY^3}$
and $\corf{}{X, X, X, XY^3}$ with $\alpha=1$.
These correspond to the correlators $\corf{}{\bone_{7},\bone_{7},\bone_{4},\bone_{14}}$ and  $\corf{}{\bone_{11},\bone_{11},\bone_{11},\bone_{14}}$.

The correlators in question are easily seen to be concave.
Applying Theorem~\ref{thm:ConcaveCodimOne} in a manner similar to the previous computations, we find that
$$\corf{E_8}{Y, Y, Y^3, XY^3}= \corf{E_8}{\bone_{7},\bone_{7},\bone_{4},\bone_{14}} = \frac{1}{5}$$
and
$$\corf{E_8}{X, X, X, XY^3}=\corf{E_8}{\bone_{11},\bone_{11},\bone_{11},\bone_{14}} = \frac{1}{3}.$$

\subsubsection{Four-point correlators for $D_{n+1}$ with $n$ odd and symmetry group $\genj$}

Next consider the case of $D_{n+1}$ for $n$ odd with symmetry group $\genj$.  We will use the notation of Subsection~\ref{sec:DnOddFA} but with
$\sigma=1$ instead.  By Theorem~\ref{thm:FourPtReduction} we need only compute the correlator
$$\corf{D_{n+1}}{X, X, X^{n-1}, X^{n-2}}
    =\corf{D_{n+1}}{\bone_3,\bone_3,-2\bone_3^{n-1},-2\bone_3^{n-2}}.$$
To apply Theorem~\ref{thm:ConcaveCodimOne} we need only compute the group element acting at the node over the three boundary graphs.

There are three uncut graphs in $\bGa_{0,4,D_{n+1}}(\bone_3 , \bone_3 , \bone_{n-1} , \bone_{n-3})$.   The first we will denote by $\Ga_1$ and it is depicted in Figure~\ref{fig:DCorr-ttf}.  As before, the choice of labeling the internal edge with $+$ and $-$ gives each term in the sum twice and will exactly cancel the factor of $\frac12$ in each case.  The edge of $\Ga_1$ is labeled with $\ga_{\Ga_1} = J^{-a}$ for
$a = n-5$, assuming $n>3$.
This gives
\begin{equation*}
\Theta_x^{\ga_{\Ga_1}}(1-\Theta_x^{\ga_{\Ga_1}}) = \frac{5(n-5)}{n^2}\dsand
\Theta_y^{\ga_{\Ga_1}}(1-\Theta_y^{\ga_{\Ga_1}}) = \frac{n^2-25}{4n^2}.
\end{equation*}

The other two graphs are both decorated as in Figure~\ref{fig:Dcorr-tno}.  We will abuse notation and denote both of them by $\Ga_2$ and simply count the contribution of $\Ga_2$ twice.  The edge of $\Ga_2$ is labeled with $\ga_{\Ga_2} = J^{(n-1)}$.
\begin{figure}
\includegraphics{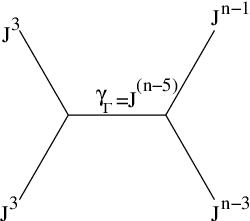}
\caption{\label{fig:DCorr-ttf} The graph $\Ga_1 \in \bGa_{0,4,D_{n+1}}(J^3 , J^3 , J^{n-1} , J^{n-3})$.}
\end{figure}
\begin{figure}
\includegraphics{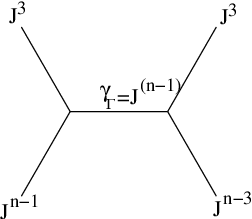}
\caption{\label{fig:Dcorr-tno} The graph $\Ga_2$.  Two of the three graphs in $\bGa_{0,4,D_{n+1}}(J^3 , J^3 , J^{n-1} , J^{n-3})$ are decorated as in this figure.}
\end{figure}
This gives
$$\Theta_x^{\ga_{\Ga_2}}(1-\Theta_x^{\ga_{\Ga_2}}) = \frac{n-1}{n^2}\dsand
\Theta_y^{\ga_{\Ga_2}}(1-\Theta_y^{\ga_{\Ga_2}}) = \frac{n^2-1}{4n^2}.
$$
Putting these into Equation~(\ref{eq:CodimOneConvex}) gives
$$\corf{D_{n+1}}{\bone_3,\bone_3,\bone_3^{n-1},\bone_3^{n-2}} = 1/n, \dsand \corf{D_{n+1}}{X, X, X^{n-1}, X^{n-2}}  = 1/n. $$

In the case of $n=3$ we have to compute the correlators $\corf{D_{4}}{X, X, X^{2}, X} = \corf{D_{4}}{x\bone_3,x\bone_3,\bone_2/6,x\bone_{3}}$ and $\corf{D_{4}}{X, X, Y,X^{2}}  = \corf{D_{4}}{x\bone_3,x\bone_3,y\bone_0, \bone_2/6}$.  Unfortunately, because of the broad sectors, we cannot use the standard tools for computing these correlators.

\subsubsection{Four-point correlators for $D_{n+1}$ with maximal symmetry group.}

By Theorem~\ref{thm:FourPtReduction} we need only compute the correlator $\corf{D_{n+1}}{X,X,X^{2n-2},X^{2n-2}}$. Here, we use the corresponding notation
 from Section~\ref{sec:mirror}
  with $\alpha=1$. This corresponds to the correlator $\corf{D_{n+1}}{\bone_{n+2},\bone_{n+2},\bone_{n-1},\bone_{n-1}}$.

By Equation~(\ref{eq:sel-rule}) we compute that the degrees of the structure bundles are
$\deg(|\LL_x|) =-2$ and $\deg(|\LL_y|) =-1$.  This shows that the correlator is concave and that $R^1\pi_*\LL_y = 0$, so the $y$ terms makes no contribution to that correlator.

To apply Theorem~\ref{thm:ConcaveCodimOne} we need to know (using Equation~(\ref{eq:DnThetax})) that
$$\Theta_x^{\lambda^{n+2}} = 2/n \quad \text{ and } \quad \Theta_x^{\lambda^{n-1}} = (n-1)/n.$$
We also need to compute the contribution of the different boundary (nodal) terms.  It is easy to check, in the same manner as we did in the case of $\corf{E_7}{\bone_2 , \bone_4 , \bone_7 , \bone_7}$,
 that there is one graph $\Ga_1$ with $\Theta_x^{\ga_{\Ga_1}} = (n-3)/n$
 and two copies of a graph $\Ga_2$ with $\Theta_x^{\ga_{\Ga_2}}= 0$.

By Theorem~\ref{thm:ConcaveCodimOne} we have
\begin{align*}
&\corf{D_{n+1}}{\bone_{n+2},\bone_{n+2},\bone_{n-1},\bone_{n-1}}=\int_{\MM_{0,4}}
 \Lambda_{0,4}^{D_{n+1}}(\bone_{n+2},\bone_{n+2},\bone_{n-1},\bone_{n-1})\\
&= \cdot  \frac12\Biggl(\Bigl(\frac{1}{n^2}
-\frac1n\Bigr)\int_{\MM_{0,4}}
\kappa_1 + \sum_{i=1}^4 \Theta_x^{\ga_i}(1-\Theta_x^{\ga_i}) \int_{\MM_{0,4}}\psi_i \\
&-\sum_{\Ga\in\bGa_{0,4,E_7}(\bone_{n+2},\bone_{n+2},\bone_{n-1},\bone_{n-1})}^4 \Theta_x^{\ga_\Ga}
(1-\Theta_x^{\ga_\Ga}) \int_{\MM_{0,4}}\left[\MM(|\Ga|)\right]\Biggr)\\
&=\frac{1}{2}\left(\frac{1}{n^2} -  \frac1n +
2\frac{2}{n}\frac{n-2}{n} +  2\frac{1}{n}\frac{n-1}{n}   -
\frac{3}{n}\frac{n-3}{n}\right) = \frac{1}{n}
\end{align*}
This gives
$$
\corf{D_{n+1}}{X,X,X^{2n-2},X^{2n-2}}= \corf{D_{n+1}}{\bone_{n+2},\bone_{n+2},\bone_{n-1},\bone_{n-1}}=\frac{1}{n}.
$$

\subsubsection{Four-point correlators for $D^T_{n+1}$}

By Theorem~\ref{thm:FourPtReduction} we need only compute the correlator $\corf{D^T_{n+1}}{X,X,X^{n-1},X^{n-2}}$. Here, we choose $\sigma=1$.  This corresponds to the correlator $\corf{D^T_{n+1}}{\bone_{3},\bone_{3},\bone_{2n-1},\bone_{2n-3}}$.

A now-familiar computation shows that the correlator in question is concave (and all markings are narrow), so we may apply Theorem~\ref{thm:ConcaveCodimOne}.  Applying that theorem in a manner similar to the previous computations, we find that $$\corf{D^T_{n+1}}{X,X,X^{n-1},X^{n-2}}=\corf{D^T_{n+1}}{\bone_{3},\bone_{3},\bone_{2n-1},\bone_{2n-3}} = \frac{1}{2n}.$$

\subsection{Computation of the basic four-point correlators in the B-model}

The primary potentials on Saito's Frobenius manifolds of the $A,D,E$
singularities have been computed by a variety of computational methods
(see \cite{DV,NY,Wi2,KTS} and etc.). However, these results are
scattered in different papers and are difficult to follow. For the reader's convenience,
we present  explicit computations of the basic four-point correlators using the Noumi-Yamada formula for the flat coordinates
 of the A, D, and E singularities \cite{No1, NY}.
Recall that the primitive forms
for the ADE-singularities are $C\,dx$ for $A_n$-case and $C\, dx\wedge dy$ for the $DE$-cases. The calculation of the
flat coordinates does not depend on the leading constant $C$, but the pairing and potential function will be re-scaled by $C$.

\subsubsection{The Noumi-Yamada formula for flat coordinates}

To write the Noumi-Yamada formula for the flat coordinates, we must first make several definitions.
\begin{df}
Let $\N$ be the following set of exponents for a monomial basis of the Milnor ring $\milnor_W$:
$$\N:=\begin{cases}
\{\nu \in \Nb\suth    0\le \nu \le n-1\} & \text{if $W=A_n$ }\\
\{(\nu_1,0) \in \Nb^2\suth     0\le \nu_1 \le n-2\} \union \{(0,1)\} & \text{if $W=D_{n}$}\\
\{(\nu_1,\nu_2) \in \Nb^2\suth     0\le \nu_1 \le 2, \, 0\le \nu_2\le 1 \}  & \text{if $W=E_6$}\\
\{(\nu_1,\nu_2) \in \Nb^2\suth     0\le \nu_1 \le 2, \, 0\le \nu_2\le 1 \}\union \{0,2\}  & \text{if $W=E_7$}\\
\{(\nu_1,\nu_2) \in \Nb^2\suth     0\le \nu_1 \le 3, \, 0\le \nu_2\le 1 \}  & \text{if $W=E_8$}
\end{cases}
$$
\end{df}

For each $\nu\in \N$ we let $\phi_\nu = \bx^\nu$ be the corresponding monomial in $\milnor_W$.  Recall that a miniversal deformation of $W$ is a family of polynomials
$W_{\lambda}=W+\sum_{\nu\in\N} t_\nu\phi_\nu$.  We want to find flat coordinates $\{s_\nu\}$ with the property $\langle s_\nu, s_\upsilon \rangle=\delta_{\nu\upsilon}$.
One can formally write $s_\nu$ in terms of power series in $t_\upsilon$. One special property of the simple singularities is that
the $s_\nu$ are always a polynomial, but this is not true for general singularities.

\begin{df}
For $W\in \C[x_1,\dots,x_N]$ quasi-homogeneous, with the weight of each variable $x_i$ equal to $q_i$, and for any $\nu\in \N$ we define the \emph{weight of $s_\nu$} to be $$\sigma_\nu:=\wt(s_\nu):=1-\sum_{i=1}^N \nu_i q_i.$$  For any $\alpha\in \Nb^\N$ we define the \emph{weight of $\alpha$} to be
$$\wt(\alpha) :=\langle \alpha,\sigma\rangle :=\sum_{\nu \in \N} \alpha_\nu \sigma_\nu$$

We also define a mapping
$\ell:\Nb^\N\to \Nb^N$ by
$$\ell(\alpha):=\sum_{\nu\in \N}\nu \alpha_\nu \in \Nb^N.$$
\end{df}

\begin{thm}(See \cite[Thm 1.1]{NY})
The formula for the flat coordinates for the simple singularities with primitive form $\bigwedge_{i=1}^N dx_i$ is as
follows:
\begin{equation}
s_\nu=t_0\delta_{\nu,0}+\sum_{\substack{\alpha\in \Nb^\N \\ \langle
\sigma,\alpha\rangle=\sigma_\nu}}c_\nu(\ell(\alpha))\frac{t^\alpha}{\alpha!},
\end{equation}
where the function $c_\nu:\Nb^N\rTo \C$ is given below.
\end{thm}
\begin{itemize}
\item[\textbf{Case $(A_n)$}]
For any $\nu \in \N = \{0,1,\cdots,n-1\}$ let $L(\nu):=\{\alpha\in \Nb\suth    \alpha\equiv\nu \mod
(n+1)\}=\{\nu+k(n+1)\suth    k\ge 0\}$.  Define $$ c_\nu(\alpha)=
\begin{cases}
    (-1)^k(\frac{\nu+1}{n+1};k) &\text{if $\alpha\in L(\nu)$}\\
        0 & \text{otherwise,}
\end{cases}
$$
where $(z;k):=\Gamma(z+k)/\Gamma(z)$ denotes the shifted factorial  function.
\item[\textbf{Case $(D_n)$}]
For any $\nu\in \N$ let
$L(\nu):= \Nb^2 \cap (\nu+\spa{ (n-1,0),(1,2)}) \linebreak[4]
=  \left\{(\nu_1+k_1(n-1)+k_2 ,\linebreak[2] \nu_2+2k_2\suth  \linebreak[3]   k_2 \ge 0 , \, k_1 \ge -(\nu_1+k_2)/(n-1) \right\}.$
Now define
$$
c_\nu(\alpha):=\begin{cases}
(-1)^{k_1+k_2}(\frac{\nu_1+1}{n-1}-\frac{\nu_2+1}{2(n-1)};k_1)(\frac{\nu_2+1}{2};k_2) & \text{if $\alpha \in L(\nu)$}\\
0 & \text{otherwise,}
\end{cases}
$$
where $(z;k):=\Gamma(z+k)/\Gamma(z)$ denotes the shifted factorial  function.

\item[\textbf{Case $(E_6)$}]
For any $\nu \in \N=\{(\nu_1,\nu_2)\suth \nu_1=0,1,2,\, \nu_2=0,1\}$ let
$L(\nu):=\{(\alpha_1,\alpha_2)\in \Nb^2\suth \alpha_1\equiv
\nu_1\mod 4, \alpha_2\equiv \nu_2\mod 3\}=\{(\nu_1+4k_1,\nu_2+3k_2)\suth k_1,k_2\ge 0\}$.
Now define
$$
c_{\nu}(\alpha):=\begin{cases}
  (-1)^{k_1+k_2}(\frac{\nu_1+1}{4};k_1)(\frac{\nu_2+1}{3};k_2) & \text{if $\alpha\in
  L(\nu)$}\\
  0 &\text{otherwise,}
\end{cases}
$$
where $(z;k):=\Gamma(z+k)/\Gamma(z)$ denotes the shifted factorial  function.

\item[\textbf{Case $(E_7)$}]
For any $\nu \in \N$ let $L(\nu):=\Nb^2 \cap \spa{(3,0),(1,3)} =\{(\nu_1+3k_1 + k_2,\nu_2+3k_2);k_2\ge 0,k_1\ge
-(\nu_2+k_1)/{3}\}$.
Now define
$$
c_\nu(\alpha):=
\begin{cases}
   (-1)^{k_1+k_2}(\frac{\nu_1+1}{3};k_1)(\frac{\nu_2+1}{3}-\frac{\nu_1+1}{9};k_2) & \text{for $\alpha\in
   L(\nu)$}\\
   0 & \text{otherwise,}
\end{cases}
$$
where $(z;k):=\Gamma(z+k)/\Gamma(z)$ denotes the shifted factorial  function.

\item[\textbf{Case $(E_8)$}]
For any $\nu\in \N =\{(\nu_1,\nu_2);\nu_1\equiv 0,1,2,3;\nu_2=0,1\}$ let $L(\nu_1,\nu_2)=\{(\alpha_1,\alpha_2)\in \Nb^2 \suth \alpha_1\equiv
\nu_1\mod 5, \alpha_2\equiv \nu_2\mod 3\} =\{(\nu_1+5k_1,\nu_2+3k_2);k_1,k_2\ge 0\}$
Now define
$$
c_\nu(\alpha):=
\begin{cases}
(-1)^{k_1+k_2}(\frac{\nu+1}{5};k_1)(\frac{\nu_2+1}{3};k_2)& \text{if $\alpha \in
L(\nu)$}\\
0 & \text{otherwise,}
\end{cases}
$$
where $(z;k):=\Gamma(z+k)/\Gamma(z)$ denotes the shifted factorial  function.
\end{itemize}

\subsubsection{Four-point correlators for $E_7$}

We start with the primitive form $dx\wedge dy = dx_1\wedge dx_2$. Assume that the deformation of $E_7$ is given by
$$
W=x_1^3+x_1x_2^3+t_1x_1^2x_2+t_3
x_1^2+t_4x_1x_2+t_5x_2^2+t_6x_1+t_7x_2+t_9.
$$
Then the flat coordinates $s$ and $t$ have the asymptotic expansion
formula (to 2nd order):
$$\begin{array}{ll}
t_1\doteq s_1   &t_3 \doteq s_3\\
t_4\doteq s_4+\frac{4}{9}s_3s_1 &t_5\doteq s_5\\
t_6\doteq s_6+\frac{1}{3}s_5s_1+\frac{5}{18}s_3^2  & t_7\doteq s_7+\frac{1}{9}s_6s_1+\frac{1}{9}s_4s_3\\
t_9\doteq s_9+\frac{2}{9}s_6s_3+\frac{1}{3}s_5s_4.
\end{array}
$$

To compute the four-point correlators, we first use the residue formula
computing the three-point correlators of the deformed chiral ring and
then take the possible first order derivatives with respect to the
flat coordinates. We have
\begin{align*}
\partial_{x_1}W=&3x_1^2+x_2^3+2s_1x_1x_2+2s_3
x_1+(s_4+\frac{4}{9}s_1s_3)x_2+s_6+\frac{1}{3}s_5s_1\\
\partial_{x_2}W=&3x_1x_2^2+s_1x_1^2+(s_4+\frac{4}{9}s_1s_3)x_1+2s_5x_2+s_7+\frac{1}{9}s_6s_1+\frac{1}{9}s_4
s_3\\
\Hess(W)=&36x_1^2x_2-9x_2^4+\text{lower order terms}\\
=&63x_1^2x_2+\text{lower order terms}\;\text{or}\;-21
x_2^4+\text{lower order terms}\\
\hat{H}_W:=&\Hess (W)/7\\
=&9x_1^2x_2+\text{lower order terms}\;\text{or}\;-3
x_2^4+\text{lower order terms}
\end{align*}

Let $C_{ijk}(s):=\Res_{s}(\frac{\partial W}{\partial
s_i},\frac{\partial W}{\partial s_j},\frac{\partial W}{\partial
s_k})$. Then
$$
C_{ijk}(s)=\left(\frac{\partial W}{\partial s_i}\frac{\partial
W}{\partial s_j}\frac{\partial W}{\partial s_k}\right)/\hat{H}_W \mod
\Jac_W.
$$
For example, $C_{991}(0)=x_1^2x_2\cdot 1\cdot 1/9x_1^2x_2=1/9$. All
the possible three-point correlators can be obtained below:
$$
\begin{matrix}
C_{991}(0)=1/9& C_{946}(0)=1/9&C_{577}(0)=-1/3\\
C_{559}(0)=-1/3& C_{667}(0)=1/9& C_{937}(0)=1/9.
\end{matrix}
$$
Now, we change primitive form from $dx_1 \wedge dx_2$ to $9dx_1 \wedge dx_2$. This rescales the pairing and entire potential function by $9$.
The cubic term of the primary potential function is
$$
F_3=\frac{1}{12}s_1s_9^2+s_4s_6s_9-\frac{3}{2}s_5s_7^2-\frac{3}{2}s_5^2s_9+\frac{1}{12}s_6^2s_7+s_3s_7s_9.
$$
Recall that the ring structure with current rescaled pairing has already been proved to be isomorphic to the
quantum ring in the A-model; and moreover the three-point correlators
in the B-model and the A-model are identical.

Using the isomorphism of the ring structure, we make the following identification
 between the basic four-point correlators in the A-model and
those in the B-model:
\begin{align*}
\langle X, X, X^2, XY\rangle_0 &\longleftrightarrow s_6^2 s_3 s_4\\
\langle X, Y, X^2, X^2\rangle_0 &\longleftrightarrow s_6 s_7 s_3^2\\
\langle Y, Y, XY, X^2Y\rangle_0 &\longleftrightarrow s_7^2 s_4 s_1.
\end{align*}

We have the formula for the four-point correlators
$$
C_{ijkl}=\restr{\frac{d}{ds_l}C_{ijk}}{s=0}.
$$
Now it is easy to obtain
$$C_{6634}=-1/9, \quad
C_{6733}=1/9, \dsand
C_{7741}=-1/3.
$$

The part of the fourth-order term of the primary potential we need is
$$
F_4=-\frac{1}{18}s_3s_4s_6^2-\frac{1}{6}s_1s_4s_7^2+\frac{1}{18}s_3^2s_6s_7.
$$
The computation here coincides with the result in \cite{No2} and in
\cite{KTS} (under a quasi-homogeneous coordinate transformation).

\subsubsection{Four-point correlators for $E_6$}

Assume that the deformation of $E_6$ is given by
$$
W=x_1^3+x_2^4+t_2x_2^2x_1+t_5 x_1x_2+t_6x_2^2+t_8x_1+t_9x_2+t_{12}.
$$
We choose the primitive form $12 dx_1\wedge dx_2$.  The metric and the third- and fourth-order
terms of the potential are given below:
\begin{align*}
&\eta_{i j}=\delta_{i,14-j},\;\text{for}\;i,j \in \{2,5,6,8,9,12\}\\
&F_3=
s_6s_8s_{12}+s_5s_9s_{12}+\frac{1}{2}s_2s_{12}^2+\frac{1}{2}s_8s_9^2\\
&F_4= -\frac{1}{8}s_5s_9s_6^2-\frac{1}{12}s_8^2
s_5^2-\frac{1}{18}s_2s_8^3-\frac{1}{8}s_2s_6s_9^2.
\end{align*}

We make the following identification between the A- and the B-models
$$
\langle Y,Y,Y^2,XY^2\rangle \longleftrightarrow s_9^2 s_6
s_2\dsand \langle X,X,XY,XY\rangle \longleftrightarrow s_8^2 s_5^2,
$$
and we get the basic four-point correlators in the B-model:
\begin{equation}
C_{9962}=-\frac{1}{4}\dsand C_{8855}=-\frac{1}{3}.
\end{equation}

\subsubsection{Four-point correlators for $E_8$}

Assume that the deformation of $E_8$ is given by
$$
W=x_1^3+x_2^5+t_1x_2^3x_1+t_4
x_2^2x_1+t_6x_2^3+t_7x_2x_1+t_9x_2^2+t_{10}x_1+t_{12}x_2+t_{15}.
$$
We choose the primitive form $dx_1\wedge dx_2$. In the same manner as before, we obtain:
\begin{align*}
&\eta_{i j}=\delta_{i,16-j},\;\text{for}\;i,j \in \{1,4,6,7,9,10,12,15\}\\
&F_3=
s_4s_{12}s_{15}+s_7s_9s_{15}+s_6s_{10}s_{15}+\frac{1}{2}s_1s_{15}^2+s_9s_{10}s_{12}+\frac{1}{2}s_7s_{12}^2\\
&F_4=
-\frac{1}{18}s_7^3s_{10}-\frac{1}{10}s_6s_7s_9^2-\frac{1}{10}s_7s_6^2s_{12}-\frac{1}{15}s_4s_9^3
-\frac{1}{6}s_4s_7s_{10}^2\\
&-\frac{1}{5}s_4s_6s_9s_{12}-\frac{1}{18}s_1s_{10}^3-\frac{1}{10}s_1s_9^2s_{12}
-\frac{1}{10}s_1s_6s_{12}^2.
\end{align*}
By the following correspondence between the A- and the B-models
$$
\langle Y,Y,Y^3,XY^3\rangle \longleftrightarrow s_{12}^2 s_6 s_1
s_2\dsand \langle X,X,X,XY^3\rangle \longleftrightarrow s_{10}^3 s_1,
$$
we get the basic four-point correlators in the B-model:
\begin{equation}
C_{(12)(12)61}=-\frac{1}{5}\dsand \ C_{(10)(10)(10)1}=-\frac{1}{3}.
\end{equation}

\subsubsection{Four point correlators for $D_{n+1}$}

Assume that the deformation is given by
$$
W=x_1^n+x_1x_2^2+\sum_{i=0}^{n-1}t_i x_1^i+t_{01}x_2.
$$
Then we have the flat coordinates by Noumi's formula
$$
\left\{
\begin{array}{l}
s_\nut \doteq t_\nut +\ct_\nut \sum_{k\ge 1}t_{\nut +k}t_{n-k}\\
s_{01}\doteq t_{01}
\end{array}\right.
$$
Here $\ct_\nut $ is just the Noumi-Yamada function $c_{\nut ,0}$ defined before, but if
$\nut +k=n-k$ then $\ct_\nut :=c_{\nut ,0}/2$.

Now the inverse function is given by
$$
\left\{
\begin{array}{l}
t_\nut \doteq s_\nut -\ct _\nut \sum_{k\ge 1}s_{\nut +k}s_{n-k}\\
t_{01}\doteq s_{01}
\end{array}\right.
$$

We have the derivative formula
\begin{equation}
\frac{\partial t_\nut }{\partial s_j}=\left\{\begin{array}{ll}
0\;&\text{if }\;j<\nut \nonumber\\
(1-\delta_{\nut  j})(-c_\nut  s_{n+\nut -j})+\delta_{\nut
j},\;&\text{if}\;j\ge \nut .
\end{array}\right.
\end{equation}
Here the indices should satisfy the restriction
$$
n+\nut -j\ge 1,\;j\ge \nut +1.
$$
We have the basic computation
\begin{align*}
&\partial_{x_1}W=nx_1^{n-1}+x_2^2+\sum_{i=1}^{n-1}i t_i x_1^{i-1},
\;\;\partial_{x_2}W=2x_1x_2+t_{01}\\
&\Hess _W=(-2)(n+1)x_2^2.
\end{align*}
The $n+1$ primary fields $\phi_i(s),0\le i\le n-1$ and
$\phi_{(01)}(s)$ are given as below,  which are functions of the
flat coordinates $s$:
\begin{align*}
&\phi_i(s)=\frac{\partial W}{\partial
s_i}=\sum_{j=0}^{n-1}\frac{\partial t_j}{\partial s_i}x_1^j\\
&\phi_{(01)}(s)=\frac{\partial W}{\partial s_{(01)}}=x_2
\end{align*}
Choose primitive form $2n\,dx_1\wedge dx_2$. Then, we re-scale pairing and potential function by $2n$. Let $\langle\phi \rangle:=2n
\Res_W(\frac{\phi}{\partial_{x_1}W\cdot\partial_{x_2}W})$. Then in
flat coordinates, we can normalize the metric $\eta$ and the three-point functions such that
\begin{align*}
&\eta_{pq}=\langle\phi_p\phi_q \rangle\\
&C_{pqr}(s)=\langle\phi_p\phi_q\phi_r\rangle.
\end{align*}

Actually $C_{pqr}(s)$ is the coefficient of the equality
$$
\phi_p\phi_q\phi_r=C_{pqr}\cdot(\Hess _W/(n+1)) \mod
\partial_{x_i}W.
$$

After a straightforward calculation, we obtain

\begin{prop} The three-point correlators of $D_{n+1}$ are given as follows:
\begin{equation}\left\{
\begin{array}{ll}
C_{ijk}=\delta_{i+j+k,n-1},\;\text{for}\;0\le i,j,k\le n-1\label{prop:3-poin-Dn}\\
C_{i(01)(01)}=-n\delta_{0i},\;\text{for}\;0\le i\le n-1;\\
C_{(01)(01)(01)}=0
\end{array}\right.
\end{equation}
The four-point correlators are
\begin{equation}\left\{
\begin{aligned}
&C_{ijkl}=(-\frac{1}{n})(l-(n-i-j-\frac{1}{2})\delta_{i+j\le
n-1}-(n-k-j-\frac{1}{2})\delta_{k+j\le
n-1}-(n-i-k-\frac{1}{2})\delta_{i+k\le n-1}),\nonumber\\
&\quad\text{for}\;0\le i,j,k\le n-1\label{prop:4-poin-Dn}\\
&C_{ij(01)(01)}=-\frac{1}{2}\delta_{i+j,n},\;\text{for}\;0\le i\le n-1;\\
&C_{i(01)(01)n-i}=-\frac{1}{2}\;\text{for}\;0\le i\le n-1;
\end{aligned}\right.
\end{equation}
The function $\delta_{x\le y}$ is defined as
$$
\delta_{x\le y }=\left\{\begin{array}{ll} 1, &\;\text{if}\;x\le
y\\
0\;&\text{if}\;x>y
\end{array}
\right.
$$
\end{prop}

\begin{crl} The basic four-point correlator for $n>3$ is
$C_{11(n-1)(n-2)}={1}/{2n}$.
\end{crl}

\subsubsection{Four-point correlators for $D_{n+1}^T$}

The singularity $D_{n+1}^T$ is isomorphic to $A_{2n-1}=x'^{2n}+y'^2$ by the quasi-homogeneous isomorphism
$$x=(2i)^{\frac{1}{n}}x', \  y=y'-ix'^n.$$
This induces an isomorphism of Saito's Frobenius manifolds with  primitive forms
$c dx\wedge dy \rightarrow c(2i)^{\frac{1}{n}}dx'\wedge dy'.$

\subsubsection{Four-point correlators for $A_n$}

 The three- and four-point correlators have already been calculated
in \cite{Wi2}.  Suppose that the deformation is given by
$$
W=x^{n+1}+\sum_{i=0}^{n-1}t_ix^i.
$$
We choose primitive form $dx$. We list the metric, three- and four-point correlators below:
\begin{align*}
&\eta_{ij}=(n+1)\delta_{i+j,n-1},\\
&C_{ijk}=\delta_{i+j+k,n-1},\\
&C_{ijkl}=-\frac{1}{n+1}\left(l+(n-j-k)\delta_{j+k\le
n-1}+(n-i-k)\delta_{i+k\le n-1}+(n-i-j)\delta_{i+j\le n-1}\right),
\end{align*}
for $0\le i,j,k\le n-1$.

\subsection{Proof of Theorem~\ref{thm:IntHierMirror}}

Because of our reconstruction theorem, to prove Theorem~\ref{thm:IntHierMirror}, it suffices to compare the
two-point, three-point, and the basic four-point functions in our theory (A-model) with
their analogues in the  B-model.

We have established the isomorphism of Frobenius algebras in Section~\ref{sec:mirror}. This means that
we have matched the unit, the pairing, and the multiplication and hence all three-point functions, by the explicit identification of state spaces. The remaining task is
to match the four-point basic correlators. We shall keep the identification of the unit and multiplication fixed.
The main idea is to explore the flexibility of
re-scaling the primitive form by a constant.  Re-scaling the primitive form by $c$ corresponds to re-scaling
$h$ by $1/c$. Hence, it still satisfies the corresponding
hierarchies. However, the corresponding Frobenious manifold
structure is different in general. 
This approach seems to give a better conceptual picture.
For the reader's convenience, we shall list the explicit value of constant we used in the proof.

Our main technical tool is the
following observation.
Let $F^A_3, F^A_4$ be the three- and basic four-point functions of the A-model and $F^B_3, F^B_4$ be the
three- and basic four-point functions of the B-model. Suppose that $F^A_3=F^B_3$. Now, we re-scale the primitive form by $c$ and make
an additional  change of variable $s_i\rightarrow \lambda^{1-deg_{\C}(s_i)}s_i$. Notice that the above change of variable preserves the unit $e$ of the Frobenius algebras.
This change of variables gives $F^B_3\rightarrow c\lambda^{\hat{c}_W}F^B_3$ and $F^B_4\rightarrow c\lambda^{\hat{c}_W+1}F^B_4$. If we choose $c=\lambda^{-\hat{c}_W}$, then
$F^B_3$ remains the same and $F^B_4\rightarrow \lambda F^B_4$. Since the linear map $s_i\rightarrow \lambda^{1-deg_{\C}(s_i)}s_i$ preserves
the unit, it preserves the metric as well.

\subsubsection{\textbf {$E_7$ A-model versus $E_7$ B-model of primitive form $(-1)^{-\frac{8}{9}}9dx_1\wedge dx_2$.}}

The Frobenius manifold in the A-model is given by the small phase space
quantum cohomology. Take the flat coordinates
$\{T_1,T_3,T_4,T_5,T_6,T_7,T_9\}$ corresponding to the primary
fields $\{\be_8=X^2Y, \be_4=X^2,\be_2=XY,\pm
y^2\be_0=Y^2,\be_7=X,\be_5=Y,\be_1=1\}$. The three-point correlators give
the cubic term of the primary potential $F^A_3$:
$$
F^A_3=\frac{1}{2}T_9^2T_1+T_9T_7T_3+T_9T_6T_4-\frac{3}{2}T_9T_5^2-\frac{3}{2}T_7^2T_5+\frac{1}{2}T_7T_6^2.
$$
The basic four-point function is
\begin{equation}
\frac{1}{18}T_6^2 T_3 T_4+\frac{1}{6}T_7^2 T_4 T_1-\frac{1}{18}T_6
T_7 T_3^2.
\end{equation}
Choose primitive form $9\,dx_1\wedge dx_2$ on the B-model side. We obtain the same metric
and the same cubic terms. However, $F^B_4=-F^A_4$.  Then, we choose $\lambda=-1$ and $c=(-1)^{-\frac{8}{9}}$.
It means that we choose primitive form $(-1)^{-\frac{8}{9}}9 \, dx_1 \wedge dx_2$. The corresponding linear map between state spaces is
$$T_i \to (-1)^{1-deg_{\C}(s_i)}s_i.$$

\subsubsection{\textbf{$E_6$ A-model versus $E_6$ B-model of primitive form $(-1)^{-\frac{5}{6}}12\,dx_1\wedge dx_2$.}}

Consider the A-model. Let the flat coordinates
$\{T_2,T_5,T_6,T_8,T_9,T_{12}\}$ correspond to the primary fields
$\{XY^2,XY,Y^2,X,Y,1\}$. Then we obtain the three-point potential
functions:
\begin{equation}
F^A_3=\frac{1}{2}T_2T_{12}^2+T_5T_9T_{12}+T_6T_8T_{12}+\frac{1}{2}T_8T_9^2.
\end{equation}
The polynomial corresponding to the basic four-point correlators is
$$
\frac{1}{8}T_2T_6T_9^2+\frac{1}{12}T_5^2T_8^2.
$$

On the B-model side, we start from primitive form $12\,dx_1\wedge dx_2$ and a linear map between state spaces $T_i\rightarrow s_i$.
It matches the unit, pairing and multiplications and hence $F^A_3=F^B_3$. But we have
\begin{equation}\label{eq:mirr-E6}
F^B_4=-F^A_4.
\end{equation}
Similar to the $E_7$, a choice of $\lambda=-1$ and $c=(-1)^{-\frac{5}{6}}$ will match the A-model to the B-model of
the primitive form $(-1)^{-\frac{5}{6}}12 \, dx_1 \wedge dx_2$.

\subsubsection{\textbf{ $E_8$ A-model versus $E_8$ B-model of primitive form $(-1)^{-\frac{14}{15}}\,dx_1\wedge dx_2$.}}

Let $\{T_1,T_4,\linebreak[4] T_6,T_7,T_9,T_{10}, T_{12}, T_{15}\}$ be
the flat coordinates in the A-model corresponds to the primary fields
$\{XY^3,XY^2,Y^3, XY, Y^2,X,Y,1 \}$. we can obtain the three-point
potential, the basic polynomials, and the basic four-point potential:
\begin{equation}
F^A_3=\frac{1}{2}T_1T_{15}^2+T_7T_9T_{15}+T_6
T_{10}T_{15}+T_4T_{12}T_{15}+\frac{1}{2}T_7T_{12}^2+T_9T_{10}T_{12}.
\end{equation}
$$
F^A_4=\frac{1}{10}T_1T_6T_{12}^2+\frac{1}{18}T_1T_{10}^3.
$$

Choose primitive form $\,dx_1\wedge dx_2$ and the linear map $T_i\rightarrow s_i$.  Then, we match the unit, pairing, multiplication. Hence, we have
$F^A_3=F^B_3$. But $F^A_4=-F^B_4$. Then, a choice of $\lambda=-1$ and $c=(-1)^{-\frac{14}{15}}$
will match the A-model with the B-model.

\subsubsection{\textbf{$(D_{n+1},\langle J\rangle,(n\;odd\;))$ A-model versus $D_{n+1}$ B-model of primitive form $(-1)^{1-\frac{n-1}{n}}4n\, dx_1 \wedge dx_2$.}}

Consider the A-model. Let $\{T_0,T_1,\cdots, T_{n-1},T_{01}\}$ be
the flat coordinates corresponding to the primary field $1,X,\cdots,
X^{n-1},Y$. Here $\{X,Y\}$ has already been identified with
$\{\be_3, \pm 2nr x^{\frac{n-1}{2}}\be_n+\mp 2ns y\be_n\}$ in the
state space $\ch_{D_{n+1},\langle J\rangle}$. We have the
computation of the $2$-point correlators (metric)
$$
\langle X^{n-1},1\rangle=-2\quad \text{and}\quad \langle Y^2,
1\rangle=2n.
$$
The
three-point potential is
$$
 F^A_3=-2\sum_{i+j+k=(n-1)}a_{ijk}T_iT_jT_k
+n T_0T_{01}^2,
$$
where
$$
a_{ijk}=\left\{\begin{array}{ll} 1& \quad\text{if $i,j,k$ are
mutually
not equal;}\\
1/2&\quad\text{if only two of $i,j,k$ are equal;}\\
1/6&\quad\text{if $i=j=k$}
\end{array},\right.
$$
and the basic four-point polynomial for $n>3$ is
$$
F^A_4=\frac{1}{2n}T_1^2T_{n-2}T_{n-1}.
$$
In the B-model, by choosing primitive form $-4n\,dx_1\wedge dx_2$ and linear map $T_i\rightarrow s_i$, we have the same pairing as the A-model and cubic term
$F^B_3=F^A_3$
and the basic four-point polynomial for $n>3$
$$
-\frac{1}{2n} s_1^2 s_{n-2}s_{n-1}.
$$
Then, a choice of $\lambda=-1$ and $c=(-1)^{-\frac{n-1}{n}}$ will match the A-model with the B-model.

\subsubsection{\textbf{  $D_{n+1}(G_{D_{n+1}})$ A-model versus $A_{2n-1}$ B-model of primitive form $2n(\frac{n}{5-4n})^{\frac{1-n}{n}}dx$.}}

Recall that
\begin{align*}
X^i &\mapsto \begin{cases}
\bone_{n+1+i} & \text{ for $0\le i < n-1$}\\
\mp 2y\bone_{0}& \text{ for $i=n-1$}\\
\bone_{i-n+1} & \text{ for $n\le i < 2n-1$}
\end{cases}
\end{align*}
is an isomorphism of graded algebras $\ch_{D_{n+1}, G_{D_{n+1}}}\rightarrow \milnor_{A_{2n-1}}$.  The pairing on $\milnor_{A_{2n-1}}$ is given by
$\langle X^{2n-2},1\rangle^{\milnor_{A_{2n-1}}} = 1/2n,$ whereas the pairing on $\ch_{D_{n+1},G_{D_{n+1}}}$ is easily seen to be given by $$\langle X^{2n-2}, \unit\rangle^{\ch_{D_{n+1}}} = \langle \bone_{n-1}, \bone_{n+1} \rangle^{\ch_{D_{n+1}}} = 1.$$
The basic four-point correlator is
$$
\corf{D_{n+1}}{X,X,X^{2n-2},X^{2n-2}}= \corf{D_{n+1}}{\bone_{n+2},\bone_{n+2},\bone_{n-1},\bone_{n-1}}=\frac{1}{n}.
$$
We start with the primitive form $2n\,dx$ on the B-model side. Then, we have an isomorphism between the A-model and the B-model ring with pairing, and hence
the potential functions have the same cubic terms, i.e., $F^A_3=F^B_3$. The basic four-point correlator of the B-model is $C_{11(2n-2)(2n-2)}=-(4n-5).$
Hence, $F^B_4=-\frac{n}{4n-5} F^A_4$. Now, a choice of $\lambda=-\frac{n}{4n-5}$ and $c=\lambda^{-\frac{n-1}{n}}$ will
match the A-model with the B-model.

\subsubsection{\textbf{  $D_{n+1}^T$ A-model versus $D_{n+1}$ B-model of primitive form $2n \, dx_1 \wedge dx_2$.}}

In the $D_{n+1}^T$ A-model, the state space $\ch_{W,G_W}$ is
generated by $n+1$ elements
$\{nx^{n-1}\be_0,\be_1,\be_3,\cdots,\be_{2i+1},\be_{2n-1}\}$.
Identify $\be_{2i+1}$ with $X^i$ and $nx^{n-1}\be_0$ with $Y$. We
have computed the metric and the three-point correlators:
$$
\langle X^i, X^j, X^k \rangle=1\quad \text{if}\quad i+j+k=n-1,\quad
\langle 1, Y,Y\rangle=-n,
$$
and the other three-point correlators are zero.

The basic four-point correlator is
$$
\langle X, X, X^{n-1},X^{n-2}\rangle=\frac{1}{2n}.
$$
Let $\{T_0,T_1,\cdots, T_{n-1},T_{01}\}$ be the flat coordinates
with respect to the primary fields $\{1,X,\cdots,X^{n-1},Y\}$.
On the B-model side, we choose primitive form $2n \, dx_1 \wedge dx_2$ and the linear map $T_i\rightarrow s_i$. Comparing the A-model and the B-model, we have the identity
\begin{equation}
F^A_{3+4}(T)=F^B_{3+4}(T).
\end{equation}
This shows that $F^A=F^B$ with primitive form $2n\, dx_1 \wedge dx_2$ and  completes the proof of Theorem~\ref{thm:IntHierMirror}.

\bibliographystyle{amsplain}

\providecommand{\bysame}{\leavevmode\hbox
to3em{\hrulefill}\thinspace}

\end{document}